\documentclass[final,12pt]{colt2019} 

\title[Stochasticity using potential functions]{Stochastic first-order methods:\\ non-asymptotic and computer-aided analyses via potential functions}
\usepackage{times}
\usepackage{booktabs} 
\usepackage{multirow}
\usepackage{array}
\usepackage[shortlabels]{enumitem}	
\setlist[itemize,1]{label=$\diamond$}
\setlist[enumerate,1]{leftmargin=.5cm}
\newcolumntype{R}[1]{>{\raggedleft\let\newline\\\arraybackslash\hspace{0pt}}m{#1}}

\usepackage{framed}

\usepackage {tikz}
\usepackage{pgfkeys}
\usetikzlibrary{shapes,arrows}
\usepackage{pgfplots}
\pgfplotsset{compat=1.13}

\tikzstyle{loosely dashed}=          [dash pattern=on 3pt off 6pt]
\pgfplotsset{plotOptions/.style={%
		width=\linewidth,
		xlabel={Iteration $k$},
		label style={font=\scriptsize},
		legend style={font=\scriptsize},
		xtick={0,25,50,75,100},
		tick label style={font=\scriptsize},
		solid,
		very thick
	}}

	\definecolor{silver}{cmyk}{0,0,0,0.3}
	\definecolor{yellow}{cmyk}{0,.5,.3,0}
	\definecolor{orange}{cmyk}{0,0.5,1,0}
	\definecolor{red}{cmyk}{0,1,1,0}
	\definecolor{reddishyellow}{cmyk}{0.5,0.2,0,0}
	\definecolor{black}{cmyk}{0,0,0.0,1.0}
	\definecolor{darkYellow}{cmyk}{0.8,0,0.0,0.5}
	\definecolor{darkSilver}{cmyk}{0,0,0,0.1}
	
	\definecolor{lightyellow}{cmyk}{0.0,0,0,0.0}
	\definecolor{lighteryellow}{cmyk}{0,0,0.0,0.0}
	\definecolor{lighteryellow}{cmyk}{0,0,0.0,0.0}
	\definecolor{lightestyellow}{cmyk}{0.0,0,0.0,0.0}
	
	\definecolor{colorP1}{RGB}{55,126,184}  
	\definecolor{colorP2}{RGB}{228,26,28}  
	\definecolor{colorP3}{RGB}{152,78,163} 
	\definecolor{colorP4}{RGB}{77,175,74}  
	\definecolor{colorP5}{rgb}{0.6, 0.4, 0.08} 
	
	\coltauthor{\Name{Adrien Taylor} \Email{adrien.taylor@inria.fr}\\
		\addr INRIA, D\'epartement d'informatique de l'ENS, Ecole normale sup\'erieure, CNRS, PSL Research University, Paris, France
		\AND
		\Name{Francis Bach} \Email{francis.bach@inria.fr}\\
		\addr INRIA,  D\'epartement d'informatique de l'ENS, Ecole normale sup\'erieure, CNRS, PSL Research University, Paris, France
	}
	\newcommand{\Rd}{\mathbb{R}^d}
	\newcommand{\R}{\mathbb{R}}
	\newcommand{\Sb}{\mathbb{S}}
	\newcommand{\inner}[2]{{\langle #1; #2\rangle}}
	\newcommand{\normsq}[1]{{\lVert #1\rVert ^2}}
	\newcommand{\norm}[1]{{\lVert #1\rVert}}
	
	\newcommand{\V}{{\mathcal{V}}}
	\newcommand{\F}{{\mathcal{F}}}
	\newcommand{\E}{{\mathbb{E}}}
	\newcommand{\N}{{\mathbb{N}}}
	\newcommand{\bO}{{O}}
	\newcommand{\FmuL}{{\mathcal{F}_{\mu,L}}}
	\newcommand{\FL}{{\mathcal{F}_{L}}}
	\newcommand{\Fccp}{{\mathcal{F}_{0,\infty}}}
	\newcommand{\edit}[1]{{#1}}
	\newcommand{\edittwo}[1]{{#1}}
	\newcommand{\correc}[1]{{\color{red}#1}}

	\newcommand{\bx}{{\mathbf{x}}}
	\newcommand{\by}{{\mathbf{y}}}
	\newcommand{\bz}{{\mathbf{z}}}
	\newcommand{\bg}{{\mathbf{g}}}
	\newcommand{\bG}{{\mathbf{G}}}
	\newcommand{\bw}{{\mathbf{w}}}
	\newcommand{\bfu}{{\mathbf{f}}}
	
	\newcommand{\bU}{{\mathbf{U}}}

	\newcommand{\secref}[1]{{Sec.~\ref{#1}}}
	\newcommand{\appref}[1]{{App.~\ref{#1}}}
	\newcommand{\prox}[2]{{\mathrm{prox}_{#1}\left(#2\right)}}
	\newcommand{\argm}[2]{{\mathrm{argmin}_{#1}\left\{#2\right\}}}
	\newcommand{\tra}{{\mathrm{Trace}}}
	\newcommand{\sspan}{{\mathrm{span}}}
	\newcommand{\dom}{{\mathrm{dom}}}
	\newcommand{\defeq}{:=}
	\newcommand{\st}{\text{ subject to }}
	
\begin{document}
		
		\maketitle
		\vspace{-1cm}
		\begin{flushleft}\correc{This version contains additional {clarifications} to the text, as well as {corrected typos} in the appendix (see~\hyperref[s:acks]{Acknowledgments}). All those changes are highlighted in {\bf{}red} throughout the text.}
		\end{flushleft}
		\vspace{-.15cm}
		
		\begin{abstract}
			We provide a novel computer-assisted technique for systematically analyzing first-order methods for optimization. In contrast with previous works, the approach is particularly suited for handling sublinear convergence rates and stochastic oracles. The technique relies on semidefinite programming and potential functions. It allows simultaneously obtaining worst-case guarantees on the behavior of those algorithms, and assisting in choosing appropriate parameters for tuning their worst-case performances. The technique also benefits from comfortable \emph{tightness guarantees}, meaning that unsatisfactory results can be improved only by changing the setting. We use the approach for analyzing deterministic and stochastic first-order methods under different assumptions on the nature of the stochastic noise. Among others, we treat unstructured noise with bounded variance, different noise models arising in over-parametrized expectation minimization problems, and randomized block-coordinate descent schemes.
		\end{abstract}
			
			\section{Introduction}
			In this work, we study methods for solving convex (stochastic) minimization problems of the form
			\begin{equation}
			\min_{x\in\Rd} f(x),\label{eq:OPT}\tag{Opt}
			\end{equation}
			with $f\in\F$ some class of convex, proper and closed functions. To perform the minimization, we are given access to an approximate first-order oracle
			$G(x;i)\approx f'(x)$, where $i$ is some random variable uniformly sampled in an index set $I$. This includes unbiased stochastic oracles satisfying $\E_i G(x;i)=f'(x)$, but also biased oracles used in block-coordinate methods $G(x;i)=\nabla_i f(x)$ (directional derivative along the $i^\text{th}$ block of coordinates).

			We present a generic approach, based on potential functions, for analyzing and designing first-order methods in the case where $I$ is a finite set---two such problems are the \emph{empirical risk minimization} setting where $f(x)=\tfrac1n \sum_{i=1}^n f_i(x)$ and $ G(x;i)=f_i'(x)$, and the block-coordinate setting. Even though most proofs presented in the sequel turned out not to depend on the cardinality of $I$, and are therefore valid for expectation minimization problems $f(x)=\E_i f(x;i)$, cardinality can play a major role in specific settings (e.g., finite sums or coordinate descent). Therefore, we do not explicitly look for results that are independent of it, but rather note that it naturally does not intervene in, e.g., analyses of stochastic gradient-based methods that we propose in this paper.
			
			\subsection{Preliminaries}\label{s:prelim}
			A continuously differentiable function $f:\Rd\rightarrow\R$ is called $L$-smooth if its gradient satisfies a Lipschitz condition with parameter $L>0$:
			\[ \norm{f'(x)-f'(y)}\leq L \norm{x-y} \quad \text{ for all } x,y\in\Rd,\]
			\edit{where $\inner{.}{.}$ denotes} the Euclidean inner product and $\norm{.}$ is the induced norm. We denote by $\FL(\Rd)$ the class of $L$-smooth convex functions over $\Rd$ and by $\FL$ the class of smooth convex functions where $d$ is left unspecified. In addition, we denote by $x_\star$ some optimal solution to~\eqref{eq:OPT}, and by $f_\star:=f(x_\star)$ the optimal value. Those assumptions are standard in the optimization literature~\citep{Book:polyak1987,Book:Nesterov}.
			
			\subsection{Contributions}
			The main contribution of this work is to propose a framework for constructing potentials for first-order stochastic algorithms; in contrast with previous related works on the topic, the technique is specialized for establishing sublinear convergence rates. The methodology benefits from an advantageous tightness property, meaning that it fails only when it is impossible to prove the desired result using potential functions with the chosen structure. The framework allows dealing with, among others, all stochastic settings presented in Table~\ref{tab:setting_summary}. We use the methodology for designing novel analyses of SGD and averaging schemes in different stochastic optimization settings. Based on the methodology for constructing potentials, we propose a complementary automatic parameter selection technique in \appref{sec:paramSelec}, whose main idea is roughly to optimize the algorithmic parameters while designing the potentials.
			
			\begin{table}[!ht]
				\begin{center}
					{\renewcommand{\arraystretch}{1.3}
						\begin{tabular}{@{}ccccc@{}}
							\specialrule{2pt}{1pt}{1pt}
							$\F$ & Noise model &  $\E_i\normsq{G(x;i)}\leq$ &  Note & Sections \\
							\specialrule{2pt}{1pt}{1pt}
							$f\in\FL$ & $G(x;i)=f'(x)$ &  --- & No noise & \ref{sec:GM}, \ref{sec:GM_app} \\
							\multirow{2}{.1\linewidth}{ $f\in\FL$ } & \multirow{2}{.2\linewidth}{\centering{\color{red}$\E_i G(x;i)=f'(x)$}} &\multirow{2}{.27\linewidth}{\centering $\sigma_\star^2+2\rho_1 L (f(x)-f_\star)+\rho_2 \normsq{f'(x)}$ for all $x\in\Rd$} &  \multirow{2}{.17\linewidth}{\centering Unified variance model} & \multirow{2}{.12\linewidth}{\centering \ref{sec:boundedVar}, \ref{sec:boundedvariances}, \ref{sec:weakgrowth_proofs}}\\
							& & \\
							$f_i\in\FL$ & $G(x;i)=f_i'(x)$ & $\sigma^2_\star$ for some $x=x_\star$ & Variance at $x_\star$ & \ref{sec:zeroVarOpt}, \ref{sec:overparameterized}, \ref{sec:boundedVaratOpt} \\
							$f_i\in\FL$ & $G(x;i)=f_i'(x)$ & --- & Finite sums & (not presented) \\
							$f_i\in\FL$ & $G(x;i)=\nabla_i f(x)$ & --- & Block-coordinate  & \ref{sec:coordinatedescent} \\
							\specialrule{2pt}{1pt}{1pt}
						\end{tabular}}\vspace{-.6cm}
				\end{center}
				\caption{Stochastic settings summary: non-exhaustive list of assumptions on the classes of functions $\F$ and the nature of the noisy oracle $G(x;i)$ that can be directly embedded within the framework. Further examples are discussed in \secref{sec:ccl}.}\label{tab:setting_summary}
			\end{table}
			
			\subsection{Prior works}\label{sec:prior}
			In what follows, we take a \emph{worst-case} point of view, that is standard in the optimization and machine learning communities, as in the original works~\citep{Book:NemirovskyYudin,Book:polyak1987,Book:Nesterov}.
			
			\paragraph{Stochastic first-order methods.} Stochastic approximation algorithms date back to the works of~\citet{robbins1951stochastic}, and numerous analyses and improvements can be found in the literature (see e.g.,~\citet{bottou2018optimization} and the references therein). Among others, averaging plays a crucial role for improving their convergence guarantees~\citep{ruppert1988efficient,polyak1992acceleration}. 
			
			\sloppy Stochastic methods are usually analyzed using a uniformly bounded variance assumption (i.e., $\E_i\normsq{G(x;i)-f'(x)}\leq \sigma^2$ for all $x\in\Rd$), or bounded gradients. This intrinsically has limited applicability ranges (e.g., it does not even hold for quadratic minimization), although theoretical guarantees for stochastic methods involving momentum were out of reach so far without such assumptions~ \citep{hu2009accelerated,xiao2010dual,devolder2011stochastic,lan2012optimal,cohen2018acceleration}. A better understanding of stochastic gradient methods (in particular, those involving momentum) can therefore only be achieved by studying alternatives to standard assumptions. In particular, let us mention the non-asymptotic analyses of~\citet{moulines2011non} (not relying on the uniformly bounded assumption), and~\citet{schmidt2013fast,belkin18,vaswani2018fast} (strong and weak growth conditions), and numerous works on stochastic methods for quadratic minimization, see e.g.,~\citet{bach2013non,dieuleveut2017harder,jain2018accelerating,jain2018parallelizing}.
			
			\paragraph{Potential functions for first-order methods.} Potential functions have been used a lot for studying convergence properties of first-order methods. This kind of analyses is typically natural for obtaining linear convergence results---potentials are then often being referred to as Lyapunov functions~\citep{lyapunov}, as in the analyses of dynamical systems---, but is typically also used for certifying sublinear convergence rates, as discussed in \secref{sec:Lyap}. As being nicely reviewed by~\citet{bansal2017potential}, the use of potential functions is not new in the optimization literature, and is closely related to the machinery of \emph{estimate sequences}~\citep{Book:Nesterov,Book:Nesterov2,wilson2016lyapunov}. Successful uses of such techniques include the original developments underlying accelerated gradient~\citep[Theorem 1]{Nesterov:1983wy} and FISTA~\citep[Lemma 4.1]{beck2009fast}.
			
			\paragraph{Computer-assisted analyses of first-order methods.} \sloppy Recently, linear matrix inequalities (LMI) and semidefinite programming (SDP) \edit{(see e.g.,~\cite{vandenberghe1996semidefinite})} techniques were used for automatically generating worst-case guarantees for first-order methods. This trend started with \emph{performance estimation} as initiated by~\cite{drori2014performance} and was taken \edit{further in different directions: for designing optimal} methods~\citep{drori2014performance,drori2016optimal,kim2016optimized,kim2018optimizing}, lower bounds~\citep{drori2017exact}, or to be featured with automated tightness guarantees and broader range of applications~\citep{taylor2017exact,taylor2017smooth}. A~competing strategy, inspired on the one hand by performance estimation and on the other one by control theory, was developed by~\cite{lessard2016analysis}. This technique is based on \emph{integral quadratic constraints} and was initially specialized for obtaining linear convergence rates. \edit{This work adapts the performance estimation approach for using potential functions, easing the development of proofs in settings involving sublinear convergence rates. Detailed relations to works in this research stream are provided in \appref{app:rel_works}.}

			\subsection{Organization}
			The flow of this work is as follows. First, we summarize the overall methodology and recall the general principle behind potential-based proofs in \secref{sec:Lyap}. After that, the procedure for designing potential functions is presented in \secref{sec:designpotentials}; in particular, the methodology is illustrated on gradient descent and a few stochastic variants in the same section.  Then, we present simple results obtained with the analysis technique in different stochastic optimization settings (a few samples from Table~\ref{tab:setting_summary}) in \secref{sec:zeroVarOpt} and in the appendix. Most technical tools are presented in appendix, including an automatic parameter selection technique, and the application to accelerated first-order methods, to stochastic optimization under weak growth conditions, and to coordinate descent. The organization and content of the appendix is summarized in \appref{sec:gen}.
			\section{Potential functions for a restricted class of first-order methods}\label{sec:Lyap}
			Let us restrict ourselves to a specific class of methods encapsulating SGD, along with possibly averaging and momentum. This restriction is made for readability purposes only. We consider the following class of (stochastic) first-order methods:
			\begin{equation}\label{eq:algo}\tag{SFO}
			\begin{aligned}
			y_{k+1}&=y_k+\alpha_k \, (x_k-y_k)+ \alpha_k' \, (z_k-y_k),  \\
			x_{k+1}^{(i_{k})}&=y_{k+1}+\beta_k\,(x_k-y_{k+1})+\beta_k'\, (z_k-y_{k+1}) - \delta_k G(y_{k+1};i_{k}),\\
			z_{k+1}^{(i_{k})}&=y_{k+1}+\gamma_k\,(x_k-y_{k+1})+\gamma_k'\, (z_k-y_{k+1}) - \epsilon_k G(y_{k+1};i_{k}),\\
			\end{aligned}
			\end{equation} 
			where the superscript $(i_k)$ corresponds to the sampled random variable that was used for performing iteration $k$.
			In what follows, we provide a generic approach to study its worst-case properties. For now, let us ask the question \emph{how can we prove such an algorithm work?} A possible methodology for showing convergence consists in exhibiting a potential function, also often referred to as a \emph{Lyapunov function}. For example, when $x_{k+1}=x_k-\tfrac1L f'(x_k)$ (i.e., gradient descent), it is possible to show (see, e.g.,~\cite{bansal2017potential}) that for all $f\in\FL$ and $k\geq0$ the inequality $\phi_{k+1}^f(x_{k+1})\leq\phi_k^f(x_k)$ holds with \[ \phi^f_k(x_k) = k(f(x_{k})-f_\star)+\frac L2 \normsq{x_{k}-x_\star},\]
			leading to $N(f(x_{N})-f_\star)\leq {\phi^f_{N}\leq \phi^f_{N-1} \leq \hdots \leq \phi^f_0}{=\frac{L}2 \normsq{x_0-x_\star}}$. Therefore, $f(x_{N})-f_\star\leq \frac{L\normsq{x_0-x_\star}}{2N}$. This proof relies on two key ideas: (i) forget how $x_k$ was generated and study only one iteration at a time, and (ii) choose an appropriate sequence of potentials. Such proofs are philosophically simple, but it is generally unclear how to chose such potentials. Choosing an appropriate sequence $\{\phi_k^f\}$ usually requires a lot of intuitions and potentially tedious investigations. Such proofs may therefore be seen as reserved to experts, and the purpose of this work is to alleviate as much as possible this burden, by proposing a systematic way of designing and verifying potentials. All proofs developed hereafter follow the same principles, and reduce to proving inequalities of type
			\begin{equation}\label{eq:pot}\tag{Pot}
			\E_{i_k} \phi^f_{k+1}(y_{k+1},x_{k+1}^{(i_k)},z_{k+1}^{(i_k)})\leq \phi^f_{k}(y_k,x_k,z_k) + e_k,
			\end{equation}
			for all $f\in\F$ and all $x_k,y_k,z_k$ used to generate $y_{k+1},x_{k+1}^{(i_k)},z_{k+1}^{(i_k)}$ with the method of interest. The term $e_k$ is typically used for encapsulating the \emph{variance} of stochastic algorithms. As before, a recursive use of this inequality allows obtaining
			\[ \E \phi^f_{N}(y_{N},x_{N},z_{N})\leq \phi^f_{0}(y_0,x_0,z_0) + \sum_{k=0}^{N-1}e_k, \]
			and the game consists in choosing appropriate sequences for $\phi^f_k$ and $e_k$. The expectation $\E$ is taken over all sequences of indices $(i_1,i_2,\hdots,i_N)$ with $i_k\in I$. \edit{Such convergence results in expectations can then typically be converted to almost sure convergence using, e.g., Robbins-Siegmund supermartingale theorem~\citep{robbins1971convergence,Duflo97}.}
			
			\section{Design methodology for potential functions}\label{sec:designpotentials}
			We propose a systematic way to verify that a given tuple $(\phi_{k+1}^f,\phi_{k}^f,e_k)$ satisfies inequality~\eqref{eq:pot}. First of all, it is clear that the set of such acceptable tuples is convex. Even more, when $\phi_{k}^f,\phi_{k+1}^f$ are both \emph{quadratic} functions of the first-order information $G(.;i)$ and the coordinates $x$ and \emph{linear} functions of the function values $f(.)$, then verifying that the tuple satisfies~\eqref{eq:pot} can equivalently be formulated as a linear matrix inequality (LMI). This section aims at providing strategies for finding sequences of potentials $\{\phi_k^f\}_k$ based on a few examples and on Proposition~\ref{prop:verify_potential} that follows.
			\subsection{Verifying a potential}
			The main tool we use for designing potentials is summarized through the following proposition. Even though its proof may appear as straightforward (we do not provide it), the main component in our strategies is the possibility of efficiently formulating~\eqref{eq:tt} (for verifying a potential) using LMIs. In the sequel, we present the methodology in a high-level form and delay all LMI formulations to appendix for readability purposes.
			\begin{proposition}\label{prop:verify_potential} 
				Let $\F$ be a class of functions, $I$ be an index set,  $G(x;i)$ (with ${i}\in I$) be satisfying one of the noise model of Table~\ref{tab:setting_summary}, \eqref{eq:algo} be the class of algorithms under consideration, and a given tuple $(\phi_{k+1}^f,\phi_k^f,e_k)$. We have \[\E_{i_k} \phi^f_{k+1}(y_{k+1},x_{k+1}^{({i_k})},z_{k+1}^{({i_k})})\leq \phi^f_{k}(y_k,x_k,z_k) + e_k \quad \text{($\E_{i_k}$ denotes the expectation over ${i_k}\in I$)}\] for all $d\in\N$, $f\in\F(\Rd)$ and all $(y_k,x_k,z_k)\in\Rd\times\Rd\times\Rd$ if and only if
				\begin{equation}\label{eq:tt}\begin{aligned}
				0\geq \edit{\sup_{\substack{d,f,y_k,x_k,z_k\\ \{G(x; i)\}_{i\in I}}}} \ & \E_{i_k} \phi^f_{k+1}(y_{k+1},x_{k+1}^{({i_k})},z_{k+1}^{({i_k})})-\phi^f_{k}(y_k,x_k,z_k)- e_k\\
				\text{s.t. } &(y_{k+1},x_{k+1}^{({i_k})},z_{k+1}^{({i_k})}) \text{ generated by~\eqref{eq:algo} from } (y_{k},x_{k},z_{k})\\
				& \{G(x; i)\}_{i\in I} \text{ compatible with $f$ and the noise model for all $x\in\mathrm{dom}f$} \\
				& f\in\F(\Rd) \text{ and } f'(x_\star)=0.
				\end{aligned}
				\end{equation}
			\end{proposition}
			\begin{remark}\label{rem:lossless} In many standard settings (including noise models presented in Table~\ref{tab:setting_summary} and the use of quadratic potentials---see examples below), the decision problem~\eqref{eq:tt} can be reformulated as a LMI. This lossless reformation into a suitable LMI directly follows from the derivations presented by~\citet[Section 2]{taylor2017exact} for the deterministic setting. Those reformulations can be extended in a straightforward manner to all settings presented in Table~\ref{tab:setting_summary}, so we only present them in appendix for the different examples treated hereafter. As an introductory example, one can find the $3\times 3$ LMI reformulation for gradient descent in \appref{sec:Vk_pgm}; the other examples are summarized in \appref{sec:gen}.
			\end{remark}
			When the class of functions $\F$, the method and the noise model are clear from the context, we abusively denote the set of tuples $(\phi^f_{k+1},\phi^f_{k},e_k)$ that satisfies~\eqref{eq:tt} by $\V_k$.
			\subsection{Strategies for designing sequences of potentials}
			Taking advantage of Proposition~\ref{prop:verify_potential}, we propose a few strategies for choosing sequences of potential functions $\{\phi_k^f\}_k$ based on two examples. For simplicity, we start in a deterministic setting. The codes used for generating the numerics below are provided in \secref{sec:ccl}.
			\subsubsection{Example I: gradient descent}\label{sec:GM}
			Say we want to bound $\normsq{f'(x_N)}$, where $x_N$ is the iterate obtained after performing $N$ iterations of gradient descent $x_{k+1}=x_k-\tfrac1L f'(x_k)$ where $f\in\FL$. Let us choose the family of potentials:
			\begin{equation}\label{eq:pot_GD}
			\begin{aligned}
			\phi_k^f=\begin{pmatrix}x_k-x_\star\\ f'(x_k)\end{pmatrix}^\top \left[\begin{pmatrix}
			a_k & c_k \\ c_k & b_k
			\end{pmatrix}\otimes I_d\right]\begin{pmatrix}x_k-x_\star\\ f'(x_k)\end{pmatrix} + d_k\, (f(x_k)-f_\star),
			\end{aligned}
			\end{equation}
			parametrized by $\{(a_k,b_k,c_k,d_k)\}_k$. The motivation for such a shape is simply to allow all the information available at $x_k$ to be used, and the Kronecker product with the identity ``$\cdot \otimes I_d$'' corresponds to requiring the potential function to be \emph{isotropic} in the spaces of coordinates and gradients; in other words, the potentials can be written
			\begin{align*}
			\phi_k^f= a_k\, \normsq{x_k-x_\star}+b_k\,\normsq{f'(x_k)}+2c_k\,\inner{f'(x_k)}{x_k-x_\star}+ d_k\, (f(x_k)-f_\star).
			\end{align*}		
			Now let us arbitrarily choose $\phi_0^f=L^2\normsq{x_0-x_\star}$ and $\phi_N^f=b_N\, \normsq{f'(x_N)}$. The main motivation underlying this choice is that this structure may result in $\normsq{f'(x_N)} \leq \tfrac{ L^2\normsq{x_0-x_\star}}{b_N}$ using similar arguments as in \secref{sec:Lyap}. For choosing an appropriate sequence $\{(a_k,b_k,c_k,d_k)\}_k$, we propose to solve the following problem:
			\begin{equation}\label{eq:GM_lyap}
			b_N^{\text{(opt)}}=\max_{\phi_1^f,\hdots,\phi_{N-1}^f,b_N} b_N \text{ subject to } (\phi_0^f,\phi_1^f)\in\V_0,\hdots,(\phi_{N-1}^f,\phi_N^f)\in\V_{N-1}.
			\end{equation}
			This problem can be formulated via semidefinite programming (SDP) with $N$ LMIs of size $3\times 3$ (see \appref{sec:GM_app})\edit{, and allows recovering the largest possible valid $b_N$ given structure~\eqref{eq:pot_GD}, $\phi_0^f$ and $\phi_N^f$}. Based on numerical inspection (details hereafter), one can find the following valid sequence of potentials: $$\phi_k^f(x_k) = (2k+1) L (f(x_k)-f_\star) + k (k+2) \normsq{f'(x_k)}+L^2\normsq{x_k-x_\star},$$
			which directly allows to prove both $f(x_k)-f_\star=\bO(k^{-1})$ and $\normsq{f'(x_k)}=\bO(k^{-2})$ simultaneously. Although simple, this is apparently the first time that convergence in gradient norm is proved using standard techniques that are usually used only for function values: the typical technique for obtaining this rate of convergence for gradient norm was presented by~\cite{nesterov2012make}. Let us briefly describe how such a potential can be obtained (the steps can be followed on Figure~\ref{fig:GM}). 
			\begin{enumerate}
				\item We started by numerically solving~\eqref{eq:GM_lyap} for a few values of $N$ using standard packages~\citep{Article:Yalmip,Article:Mosek}. Approximate numerical results are provided below:
				\begin{center}
					{\renewcommand{\arraystretch}{1.3}
						\begin{tabular}{@{}cccccccc@{}}
							\specialrule{2pt}{1pt}{1pt}
							$N=$	&  $1$  & $2$ 	&  $3$ 	& $4$ & $5$ & $\hdots$ 	& $100$ \\
							$b_N^{\text{(opt)}}=$ 	&  $4$ 	& $9$	&  $16$ & $25$& $36$& $\hdots$	& $10201$ \\
							\specialrule{2pt}{1pt}{1pt}
						\end{tabular}
					}
				\end{center}
				The solution to~\eqref{eq:GM_lyap} numerically appeared to match $b_N=(N+1)^2$. For completeness, note that the relative inaccuracy $b_N^{\text{(opt)}}/(N+1)^2-1$ observed on the numerical solutions appeared to be an increasing function of $N$ ranging from $10^{-8}$ for small values of $N$ to $10^{-5}$ for $N=100$. Readers interested in transforming those numerics into formal proofs may proceed with steps~2--4.
				\item Observe schedules of $\{(a_k,b_k,c_k,d_k)\}_k$ that are numerically obtained by solving~\eqref{eq:GM_lyap}. The result for $N=100$ is given in Figure~\ref{fig:GM} (plain brown).
				\item Try to simplify $\{\phi_k^f\}_k$ without loosing too much on $b_N^{\text{(opt)}}$ (i.e., keep $b_N^{\text{(opt)}}$ large enough). As examples, the first numerical schedules motivated trying to solve~\eqref{eq:GM_lyap} under the additional constraint $d_k=(2k+1)L$ (Figure~\ref{fig:GM}, dashed red), then we additionally tried $c_k=0$ and $a_k=L^2$ (Figure~\ref{fig:GM}, dashed blue). Those two simplifications turned out to be successful; an example of unsuccessful one is obtained by constraining $d_k=0$ in~\eqref{eq:GM_lyap} (Figure~\ref{fig:GM}, dashed purple), which does not allow achieving a large enough value of $b_N$ for proving $\bO(k^{-2})$ convergence in $\normsq{f'(x_k)}$.
				\item Using the numerical inspirations, study one step of the method, i.e., find a feasible point to~\eqref{eq:tt}.
				\begin{theorem}\label{thm:GM_Lyap}
					Let $x_k\in\Rd$, $f\in\FL$, and $x_{k+1}=x_k-\tfrac1L f'(x_k)$. The inequality $\phi_{k+1}^f(x_{k+1})\leq \phi_k^f(x_{k})$ holds with 
					\[\phi_k^f=d_k (f(x_k)-f_\star)+b_k\normsq{f'(x_k)}+ L^2\normsq{x_k-x_\star},\]
					and all values $b_k,d_k\geq0$ such that $b_{k+1}=2+\tfrac{d_{k}}{L}+b_k$ and $d_{k+1}= 2 L+d_k$.
				\end{theorem}
				The proof is relatively simple relies on finding a sequence of feasible points to~\eqref{eq:tt}. The proof is delayed to \appref{sec:GM_app}; in particular, the choice $d_k=(2k+1)L$ and $b_k=k(k+2)$ is valid for using Theorem~\ref{thm:GM_Lyap}, and allows recovering the previous claim. A corresponding result for the proximal gradient method is presented in~\appref{sec:pgm}, and similar results for accelerated variants are given in \secref{sec:opt_smooth}.
			\end{enumerate}
			\begin{remark}\label{rem:polynomialcoefficients}
				Instead of using $\{(a_k,b_k,c_k,d_k)\}_k$ as variables, one could explicitly require each of those to be a polynomial in $k$, and use the coefficients of those polynomials as variables in~\eqref{eq:GM_lyap}. Our choice of using $\{(a_k,b_k,c_k,d_k)\}_k$ allows knowing in advance that the optimal $b_N$ value obtained by solving~\eqref{eq:GM_lyap} is the best value of $b_N$ that can be certified given $\phi_0^f$ and the structure of $\phi_k^f$. It could nevertheless be more practical to use the polynomials' coefficients as variables in certain situations. 
			\end{remark}
			\begin{center}
				\begin{figure}[!ht]
					\begin{tabular}{r}
						\hspace{-.5cm}
						\begin{tikzpicture}
						\begin{axis}[plotOptions, title={$a_k$}, ymin=0, ymax=2,xmin=0,xmax=102,width=.285\linewidth]
						\addplot [colorP5] table [x=k,y=ak] {Data/GD_FG1.dat};
						\addplot [colorP2, dashed] table [x=k,y=ak] {Data/GD_FG2.dat};
						\addplot [colorP1, dashed] table [x=k,y=ak] {Data/GD_FG3.dat};
						\addplot [colorP3, dashed] table [x=k,y=ak] {Data/GD_FG4.dat};
						\end{axis}
						\end{tikzpicture} \hspace{-.3cm}
						\begin{tikzpicture}
						\begin{axis}[plotOptions, title={$b_k$}, ymin=-40, ymax=10000,xmin=0,xmax=102,width=.285\linewidth]
						\addplot [colorP5] table [x=k,y=bk] {Data/GD_FG1.dat};
						\addplot [colorP2, dashed] table [x=k,y=bk] {Data/GD_FG2.dat};
						\addplot [colorP1, dashed] table [x=k,y=bk] {Data/GD_FG3.dat};
						\addplot [colorP3, dashed] table [x=k,y=bk] {Data/GD_FG4.dat};
						\end{axis}
						\end{tikzpicture}\hspace{-.3cm}
						\begin{tikzpicture}
						\begin{axis}[plotOptions, title={$c_k$}, ymin=-10, ymax=10,xmin=0,xmax=102,width=.285\linewidth]
						\addplot [colorP5] table [x=k,y=ck] {Data/GD_FG1.dat};
						\addplot [colorP2, dashed] table [x=k,y=ck] {Data/GD_FG2.dat};
						\addplot [colorP1, dashed] table [x=k,y=ck] {Data/GD_FG3.dat};
						\addplot [colorP3, dashed] table [x=k,y=ck] {Data/GD_FG4.dat};
						\end{axis}
						\end{tikzpicture} \hspace{-.3cm}
						\begin{tikzpicture}
						\begin{axis}[plotOptions, title={$d_k$}, ymin=-10, ymax=200,xmin=0,xmax=102,width=.285\linewidth]
						\addplot [colorP5] table [x=k,y=dk] {Data/GD_FG1.dat};
						\addplot [colorP2, dashed] table [x=k,y=dk] {Data/GD_FG2.dat};
						\addplot [colorP1, dashed] table [x=k,y=dk] {Data/GD_FG3.dat};
						\addplot [colorP3, dashed] table [x=k,y=dk] {Data/GD_FG4.dat};
						\end{axis}
						\end{tikzpicture}
					\end{tabular}\vspace{-.35cm}
					\caption{Numerical solution to~\eqref{eq:GM_lyap} for $N=100$ and $L=1$ (plain brown), forced $d_k=(2k+1)L$ (dashed red), forced $d_k=(2k+1)L$, $c_k=0$ and $a_k=L^2$ (dashed  blue) and forced $d_k=0$ (dashed purple). \edit{Total time: $\sim35$ sec{.} on single core of Intel Core i$7$ $1.8$GHz~CPU.}}\label{fig:GM}
				\end{figure}
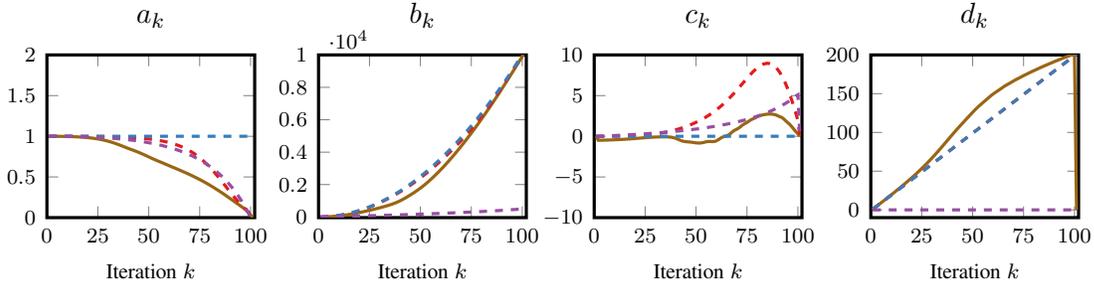
			\end{center}
			\subsubsection{Example II: stochastic smooth convex minimization, bounded variance}\label{sec:boundedVar}
			
			For studying stochastic methods, we heavily rely on the finite support assumption of the random variable $i_k\in I$. This is crucial, as our main tool is a reformulation of~\eqref{eq:tt} into a LMI where we perform an averaging over the $n=|I|$ possible scenarios.
			
			\sloppy Stochastic methods are commonly studied using a uniformly bounded variance assumption $\E_i\normsq{G(x;i)-f'(x)}$ over $x$~\citep{hu2009accelerated,xiao2010dual,lan2012optimal,devolder2011stochastic}. This assumption is quite restrictive, but analyses not relying on it often appear to be much more challenging and sometimes even out of reach so far. Nevertheless, this restrictive setting is used for the examples of \edit{this section}. Other setups are explored in \secref{sec:zeroVarOpt} and in appendix.
			
			In the previous section, the use of Proposition~\ref{prop:verify_potential} was exemplified for designing a potential function for vanilla gradient method. In the following lines, we provide two alternate ways of choosing sequences of potentials that can be used for stochastic first-order methods, the main additional difficulty being the appearance of a \emph{variance} term $e_k$ in the inequality~\eqref{eq:pot}. Final consequences of the results of this section are depicted in Table~\ref{tab:bv_best} for when (decreasing) step-sizes of the form $\delta_k=\left({L(1+k)^\alpha}\right)^{-1}$ are used in the stochastic algorithms. For SGD with and without averaging we discuss the differences with~\citet{moulines2011non} below.

			\paragraph{Stochastic gradient.}\label{sec:ex_sgd} Let us apply the methodology to the SGD iteration \begin{equation*}
			\begin{aligned}
			x_{k+1}^{({i_k})}&=x_k - \delta_k G(x_k;{i_k}),
			\end{aligned}
			\end{equation*} {\color{red}(with $\E_{i_k}G(x_k;i_k)=f'(x_k)$)} where we choose the following family of potentials:
			\begin{align*}
			\phi_k^f=\begin{pmatrix}x_k-x_\star\\ f'(x_k)\end{pmatrix}^\top \left[\begin{pmatrix}
			a_k & c_k \\ c_k & b_k
			\end{pmatrix}\otimes I_d\right]\begin{pmatrix}x_k-x_\star\\ f'(x_k)\end{pmatrix} + d_k\, (f(x_k)-f_\star).
			\end{align*}
			For choosing the sequence $\{(a_k,b_k,c_k,d_k,e_k)\}_k$, we arbitrarily start with $\phi^f_0=\tfrac{L}{2}\normsq{x_0-x_\star}$ and $\phi^f_N=d_N (f(x_N)-f_\star)$ as this may result in a guarantee of the form $\E[f(x_N)-f_\star] \leq \tfrac{L\, \normsq{x_0-x_\star}}{2d_N}+\sigma^2\tfrac{\sum_{k=0}^{N-1}e_k}{d_N}$. We proceed with a two-stage strategy:
			\begin{align}
			&d_N^{(\mathrm{opt})}=\max_{\substack{\phi_1^f,\hdots,\phi_{N-1}^f,d_N\\e_0,\hdots,e_{N-1}}} d_N\,\text{ s.t. }\, (\phi_0^f,\phi_1^f,e_0)\in\V_0,\hdots,(\phi_{N-1}^f,\phi_N^f,e_{N-1})\in\V_{N-1},\notag\\
			&\min_{\substack{\phi_1^f,\hdots,\phi_{N-1}^f,\\d_N,e_0,\hdots,e_{N-1}}} \sum_{k=1}^{N} e_k \,\text{ s.t. }\, d_N=d_N^{(\mathrm{opt})},\, (\phi_0^f,\phi_1^f,e_0)\in\V_0,\hdots,(\phi_{N-1}^f,\phi_N^f,e_{N-1})\in\V_{N-1},\label{eq:SGD_Lyap}
			\end{align}
			where the sequence is chosen as the optimal solution to~\eqref{eq:SGD_Lyap}\edit{, which is formulated using $N$ LMIs of sizes $(2n+1)\times(2n+1)$. Alternatively, one can choose two weights $R^2$ and $\sigma^2$ and solve}
			\begin{align*}
			\min_{V_1,\hdots,V_{N-1},d_N} \tfrac{R^2}{d_N}+\tfrac{\sigma^2}{d_N}\sum_{k=1}^Ne_k \text{ subject to } (\phi_0^f,\phi_1^f,e_0)\in\V_0,\hdots,(\phi_{N-1}^f,\phi_N^f,e_{N-1})\in\V_{N-1},
			\end{align*}
			which is also convex. As in the case of the gradient method, let us briefly describe the steps.
			\begin{enumerate}
				\item Solve~\eqref{eq:SGD_Lyap} for the fixed step-size policy $\delta_k=\tfrac1L$ and for a few values of $N$ (number of iterations) and $n$ (cardinality of $\{G(x;1),\hdots,G(x;n)\}$). Approximate numerical results are as follow:
				\begin{center}
					{\renewcommand{\arraystretch}{1.3}
						\begin{tabular}{@{}c|ccccccc|ccccccc@{}}
							\specialrule{2pt}{1pt}{1pt}
							$n=$	& \multicolumn{7}{c}{$2$}										& \multicolumn{7}{c}{$10$} 			\\
							$N=$	&  $1$  & $2$ 	&  $3$ 		& $4$	&  $5$	& $\hdots$ 	&$100$ 	&  $1$  & $2$ 	&  $3$ 		&  	$4$	&  $5$	& $\hdots$ 	& $100$\\
							$d_N^{\text{(opt)}}=$ 	&  $3$ 	& $5$	&  $7$ 	 	& $9$	& $11$ 	& $\hdots$	&$201$  & $3$  	& $5$ 	&$7$   	 	& $9$ 	&$11$  	& $\hdots$	& $201$\\
							$e_N^{\text{(opt)}}=$ 	&  $1.5$& $2.5$	&  $3.5$	& $4.5$	& $5.5$	& $\hdots$	&$100.5$& $1.5$	&$2.5$	&$3.5$ 	 	& $4.5$	&$5.5$ 	& $\hdots$	& $100.5$\\
							\specialrule{2pt}{1pt}{1pt}
						\end{tabular}
					}
				\end{center}
				\item Observe schedules $\{(a_k,b_k,c_k,d_k,e_k)\}_k$ that are numerically obtained by solving~\eqref{eq:SGD_Lyap}. The result for $N=100$ and $n=2$ is given in Figure~\ref{fig:sgd_BV_example} (plain brown).
				\item Simplify $\{\phi_k^f\}_k$ as much as possible without loosing too much (i.e., keep $d_N$ large and $e_N$ small). For example, enforcing $c_k=0$ in~\eqref{eq:SGD_Lyap} (Figure~\ref{fig:sgd_BV_example}, dashed red), enforcing $b_k=c_k=0$ in~\eqref{eq:SGD_Lyap} (Figure~\ref{fig:sgd_BV_example}, dashed blue) and finally enforcing $b_k=c_k=0$ and $a_k=\tfrac{L}{2}$ in ~\eqref{eq:SGD_Lyap} (Figure~\ref{fig:sgd_BV_example}, dashed purple). 
				\item Using numerical inspiration, study one step of the method, i.e., find a feasible point to~\eqref{eq:tt}.
			\end{enumerate}
			\begin{theorem}\label{thm:sgd_BV} Let $x_k\in\Rd$, $f\in\FL$ and $x_{k+1}^{({i_k})}=x_k - \delta_k G(x_k;{i_k})$ with $\E_{i_k} \normsq{G(x_k;{i_k})}\leq \sigma^2+\normsq{f'(x_k)}$ {\color{red}(and $\E_{i_k} G(x_k;i_k)=f'(x_k)$)}. The inequality $\E_{i_k} [\phi^f_{k+1}(x_{k+1}^{({i_k})})]\leq \phi_k^f(x_k)+e_k \sigma^2$ holds with
				\[\phi^f_k(x_k)= d_{k} (f(x_k)-f_\star)+\tfrac{L}{2}\normsq{x_k-x_\star},\]
				for all values $d_k, \delta_k\geq0$ such that $d_{k+1}=d_k+\delta_k L$, $e_k= \tfrac{\delta_k^2 L}{2}(1+d_{k+1})$, and $\delta_k d_{k+1} \geq e_k$ (i.e., when step-size $\delta_k$ is small enough; in particular, the choice $0\leq \delta_k\leq \tfrac1L$ is valid).
			\end{theorem}
			The proof is provided in \appref{sec:boundedvariances}. The choice $\delta_k=\tfrac1L$, $d_k=k$ and $e_k=\tfrac{1}{L}\left(\tfrac{k}2+1\right)$ leads to \begin{equation*}
			\begin{aligned}
			(k+1)\E_{i_k}[f(x_{k+1}^{({i_k})})-f_\star]+&\tfrac{L}{2}\E_{i_k}[\normsq{x_{k+1}^{({i_k})}-x_\star}]\leq k \left[f(x_{k})-f_\star\right]+\tfrac{L}{2}[\normsq{x_{k}-x_\star}] + (\tfrac{k}{2}+1)\tfrac{\sigma^2}L,
			\end{aligned}
			\end{equation*}
			the results are shown on Figure~\ref{fig:sgd_BV_example}, whose dashed purple curves correspond to
			\[N\E[f(x_{N})-f_\star]+\tfrac{L}{2}\E[\normsq{x_{N}-x_\star}]\leq \tfrac{L}{2}\normsq{x_{0}-x_\star} + \tfrac{N}{4} (N+3)\tfrac{\sigma^2}L.\]
			The choice $\delta_k=(L(1+k)^{\alpha})^{-1} $ leads to the results provided in Table~\ref{tab:bv_best} (details in \appref{sec:BV_rates}). Compared to~\citet{moulines2011non} ---which deals with the slightly different setting $f_i\in\FL$ and bounded variance at $x_\star$: $\E_i\normsq{f_i'(x_\star)}\leq\sigma^2_\star$--- we obtained bounds that are valid for all values of~$\alpha$, compared to $\alpha\in(1/2,1)$, and the same optimal value $\alpha=2/3$. Although the scope of this new result is more limited through the assumptions, the proof is considerably simpler.
			\begin{center}
				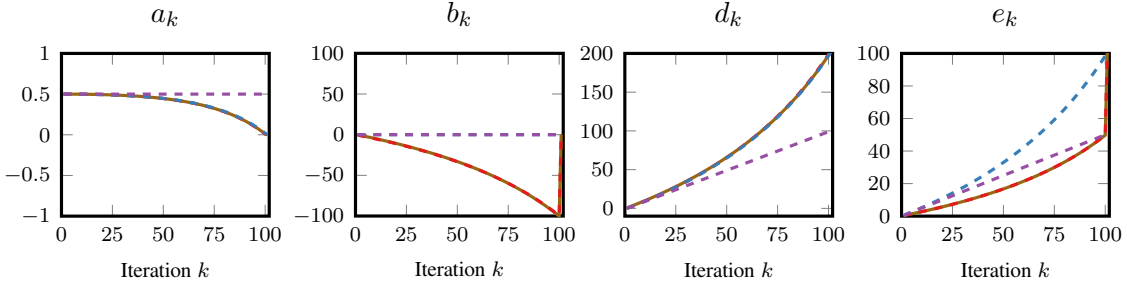
\begin{figure}[!t]
					\begin{tabular}{r}
						\hspace{-.5cm}
						\begin{tikzpicture}
						\begin{axis}[plotOptions, title={$a_k$}, ymin=-1, ymax=1,xmin=0,xmax=102,width=.285\linewidth]
						\addplot [colorP5] table [x=k,y=ak] {Data/SGD_F0.dat};
						\addplot [colorP2, dashed] table [x=k,y=ak] {Data/SGD_F1.dat};
						\addplot [colorP1, dashed] table [x=k,y=ak] {Data/SGD_F2.dat};
						\addplot [colorP3, dashed] table [x=k,y=ak] {Data/SGD_F3.dat};
						\end{axis}
						\end{tikzpicture} \hspace{-.3cm}
						\begin{tikzpicture}
						\begin{axis}[plotOptions, title={$b_k$}, ymin=-100, ymax=100,xmin=0,xmax=102,width=.285\linewidth]
						\addplot [colorP5] table [x=k,y=bk] {Data/SGD_F0.dat};
						\addplot [colorP2, dashed] table [x=k,y=bk] {Data/SGD_F1.dat};
						\addplot [colorP1, dashed] table [x=k,y=bk] {Data/SGD_F2.dat};
						\addplot [colorP3, dashed] table [x=k,y=bk] {Data/SGD_F3.dat};
						\end{axis}
						\end{tikzpicture}\hspace{-.3cm}
						\begin{tikzpicture}
						\begin{axis}[plotOptions, title={$d_k$}, ymin=-10, ymax=200,xmin=0,xmax=102,width=.285\linewidth]
						\addplot [colorP5] table [x=k,y=dk] {Data/SGD_F0.dat};
						\addplot [colorP2, dashed] table [x=k,y=dk] {Data/SGD_F1.dat};
						\addplot [colorP1, dashed] table [x=k,y=dk] {Data/SGD_F2.dat};
						\addplot [colorP3, dashed] table [x=k,y=dk] {Data/SGD_F3.dat};
						\end{axis}
						\end{tikzpicture} \hspace{-.3cm}
						\begin{tikzpicture}
						\begin{axis}[plotOptions, title={$e_k$}, ymin=0, ymax=100,xmin=0,xmax=102,width=.285\linewidth]
						\addplot [colorP5] table [x=k,y=ek] {Data/SGD_F0.dat};
						\addplot [colorP2, dashed] table [x=k,y=ek] {Data/SGD_F1.dat};
						\addplot [colorP1, dashed] table [x=k,y=ek] {Data/SGD_F2.dat};
						\addplot [colorP3, dashed] table [x=k,y=ek] {Data/SGD_F3.dat};
						\end{axis}
						\end{tikzpicture} 
					\end{tabular}
					\caption{Numerical solution to~\eqref{eq:SGD_Lyap} for $n=2$, $N=100$ and $L=1$ (plain brown), forced $c_k=0$ (dashed red), forced $c_k=0$, $b_k=0$ (dashed blue) and forced $c_k=0$, $b_k=0$ and $a_k=\tfrac{L}{2}$ (dashed purple). \edit{Total time: $\sim60$ sec{.} on single core of Intel Core i$7$ $1.8$GHz CPU.}}\label{fig:sgd_BV_example}
				\end{figure}
			\end{center}
			\paragraph{Stochastic gradient with averaging.}\label{sec:ex_avg_sgd} A standard way to improve theoretical guarantees of stochastic gradient methods is to embed them within an averaging process, as in~\citet{ruppert1988efficient,polyak1992acceleration,moulines2011non}, leading to the following iterative process:
			\begin{equation*}
			\begin{aligned}
			x_{k+1}^{({i_k})}&=x_k - \delta_k G(x_k;{i_k}),\\
			z_{k+1}^{({i_k})}&=\tfrac{1}{k+1}\,x_{k+1}^{({i_k})}+\tfrac{k}{k+1}\, z_k.
			\end{aligned}
			\end{equation*} 
			The following potential was found using a procedure similar to the ones presented for gradient and stochastic gradient methods. In contrast with previous examples, this bound can only be propagated for decreasing step-sizes (i.e., $\delta_{k+1}\leq \delta_k$). In particular, the choice $\delta_k=(L(1+k)^{\alpha})^{-1} $ leads to the results provided in Table~\ref{tab:bv_best} (details in \appref{sec:BV_rates}). Compared to~\citet{moulines2011non}, as for SGD, we obtain the same optimal value $\alpha=1/2$ for averaging, and bounds that are valid for all values of $\alpha$ (compared to only $\alpha\in(1/2,1)$ before). The proof is delayed to \appref{sec:avgsgd_pot}\edit{, and relied on using $N$ LMIs of sizes $(3+2n)\times (3+2n)$.}
			\begin{theorem}\label{thm:avg_BV}
				Consider the following iterative scheme
				\begin{equation*}
				\begin{aligned}
				x_{k+1}^{({i_k})}&=x_k-\delta_k G(x_k;{i_k}),\\
				z^{({i_k})}_{k+1}&=\tfrac{1}{d_k+1}\,x_{k+1}^{({i_k})}+\tfrac{d_k}{d_k+1}\, z_k,
				\end{aligned}
				\end{equation*}
				for some $d_k, \delta_k\geq0$. Assuming $f\in\FL$ and $\E_{i_k} \normsq{G(x_k;{i_k})}\leq \sigma^2+\normsq{f'(x_k)}$ {\color{red}(and $\E_{i_k} G(x_k;i_k)=f'(x_k)$)}, the following inequality holds
				\[\delta_k d_{k+1}L\E_{i_k}[f(z^{({i_k})}_{k+1})-f_\star]+\tfrac{L}{2}\E_{i_k}\normsq{x_{k+1}^{({i_k})}-x_\star}\leq \delta_k d_{k}L (f(z_k)-f_\star)+\tfrac{L}{2}\normsq{x_k-x_\star}+e_{k}\sigma^2,\]
				with $d_{k+1}= d_k+1$, $e_{k}=\tfrac{\delta_k^2}{2} \tfrac{L(1+d_k+L\delta_k)}{1+d_k}$ and $\delta_k \leq \tfrac{1+\sqrt{5}}{2L}$. 
			\end{theorem}
			
			\paragraph{Stochastic gradient with primal averaging.}\label{sec:ex_pravg_sgd} Inspired by the numerical step-size selection tool provided in \appref{sec:paramSelec}, we propose an alternative to averaging---sometimes referred to as \emph{primal averaging}~\citep{tao2018primal}---corresponding to evaluating the stochastic gradient at the averaged iterate, in the particular case of a fixed-step policy $\delta_k=\tfrac1L$:
			\begin{equation*}
			\begin{aligned}
			y_{k+1}&=\tfrac{k}{k+1}\,y_k+\tfrac{1}{k+1}\, x_k,\\
			x_{k+1}^{({i_k})}&=x_k - \tfrac1L G(y_{k+1};{i_k}).
			\end{aligned}
			\end{equation*}
			\edit{The following theorem was obtained through the use of $N$ LMIs of sizes $(3+n)\times (3+n)$.}
			\begin{theorem}\label{thm:pavg_BV}
				Consider the following iterative scheme
				\begin{equation*}
				\begin{aligned}
				y_{k+1}&=\tfrac{d_k}{d_k+\delta_k L} y_k+\tfrac{\delta_k L}{d_k+\delta_k L}x_k\\
				x_{k+1}^{({i_k})}&=x_k-\delta_k G(y_{k+1};{i_k})
				\end{aligned}
				\end{equation*}
				for some $d_k, \delta_k\geq0$. Assuming $f\in\FL$ and $\E_{i_k} \normsq{G(y_{k+1};{i_k})}\leq \sigma^2+\normsq{f'(y_{k+1})}$ {\color{red}(and $\E_{i_k} G(y_{k+1};i_k)=f'(y_{k+1})$)}, the following inequality holds
				\[d_{k+1}(f(y_{k+1})-f_\star)+\tfrac{L}{2}\E_{i_k}\normsq{x_{k+1}^{({i_k})}-x_\star}\leq d_k (f(y_{k})-f_\star)+\tfrac{L}{2}\normsq{x_k-x_\star}+e_k\sigma^2,\]
				with $d_{k+1}=d_k+\delta_k L$ and $e_k=\tfrac{L\delta_k^2}{2}$ when $\delta_k\leq\tfrac1L$.
			\end{theorem}
			The proof is provided in \appref{sec:primalavg_pot}. In particular, the choice $\delta_k=(L(1+k)^{\alpha})^{-1} $ leads to the results provided in Table~\ref{tab:bv_best} (details in \appref{sec:BV_rates}), and an alternate version where we always evaluate the stochastic gradient at the averaged iterate for any step-size policy $\delta_k$ is provided in \appref{sec:alt_pavging}.
			\begin{table}[!ht]
				\begin{center}
					{\renewcommand{\arraystretch}{1.3}
						\begin{tabular}{@{}ccccc@{}}
							\specialrule{2pt}{1pt}{1pt}
							& & $\normsq{x_0-x_\star}$ & $\sigma^2$ & Optimal $\alpha$\\
							Vanilla SGD&$\E f(x_k)-f_\star\leq $ & { $\bO(k^{\alpha-1})$} & { $\bO(k^{1-2\alpha})$} & {$2/3$} \\
							Polyak-Ruppert averaging&$\E f(z_k)-f_\star\leq $ & {$\bO(k^{\alpha-1})$} & {$\bO(k^{-\alpha})$} & {$1/2$}\\
							Primal averaging&$\E f(y_k)-f_\star\leq $ & {$\bO(k^{\alpha-1})$} & { $\bO(k^{-\alpha})$} & {$1/2$} \\
							\specialrule{2pt}{1pt}{1pt}
						\end{tabular}}\vspace{-.6cm}
				\end{center}
				\caption{Asymptotic rates for SGD, Polyak-Ruppert averaging, and primal averaging under uniformly bounded variance $\E_i\normsq{G(x;i)}\leq \sigma^2+\normsq{f'(x)}$ and step-sizes $\delta_k\sim k^{-\alpha}$. \edit{A factor $\log k$ was neglected for optimal $\alpha$'s. Details in~\appref{sec:BV_rates}, and momentum in \appref{sec:sgd_fast}}.}
				\label{tab:bv_best}
			\end{table}
			
			\section{Application to stochastic convex minimization for over-parameterized models}\label{sec:zeroVarOpt}		
			In many modern machine learning settings, models are over-parametrized and allow interpolating the data. This is discussed by~\citet{schmidt2013fast,belkin18,vaswani2018fast} and sometimes analyzed through the use of \emph{growth conditions} (which we discuss in \appref{sec:weakgrowth_proofs}). Alternatively, we model this scenario through the setup
			\[ \min_{x\in\Rd} \{ f(x)\equiv \ \E_i f_i(x) \},  \]
			where we assume $f_i\in\FL$ and that there exists an optimal point $x_\star$ such that $f_i'(x_\star)=0$ for all $i\in I$. Using the previous methodology, the best worst-case guarantees we could reach for vanilla SGD (without averaging) was achieved by using a decreasing step-size policy, resulting only in a disappointing $\bO(k^{-1/2})$ guarantee. On the other hand, the following method (inspired by our step-size selection tool in appendix) turned out to be considerably simpler to analyze, while enjoying better worst-case guarantees. As in the previous section, the main idea is to evaluate the stochastic gradient at the averaged iterate instead of the last one (primal averaging). The proof is delayed to \appref{sec:overparameterized}\edit{, and relied on using our step-size selection technique. The computational cost (of designing $N$ iterations of this method) was that of solving $N$ LMIs of sizes $(3 + 3n + n^2)\times(3 + 3n + n^2)$.}
			\begin{theorem}\label{thm:overparameterized} Let $x_k\in\Rd$, $f_i\in\FL$ and an optimal point $x_\star$ such that $f_i'(x_\star)=0$ for all $i\in I$. Then the iterative scheme
				\begin{equation*}
				\begin{aligned}
				y_{k+1}&=\tfrac{d_k}{d_k+\delta_k L}y_k+\tfrac{\delta_k L}{d_k+\delta_k L}x_k,\\
				x_{k+1}^{({i_k})}&=x_k-\delta_k f_{{i_k}}'(y_{k+1}),
				\end{aligned}
				\end{equation*}
				satisfies
				\[d_{k+1} (f(y_{k+1})-f_\star)+\tfrac{L}{2}\E_{i_k}\normsq{x_{k+1}^{({i_k})}-x_\star}\leq d_{k} (f(y_k)-f_\star)+\tfrac{L}{2}\normsq{x_k-x_\star},\]
				for all values of $d_k, \delta_k\geq0$ and
				\begin{equation*}
				d_{k+1}= \left\{\begin{array}{ll}
				d_k+\delta_k L & \text{if } \delta_k\leq \tfrac1L,\\
				d_k+2 \delta_k L - \delta_k^2 L^2 & \text{otherwise.}
				\end{array}\right.
				\end{equation*} 
			\end{theorem}
			Using $\delta_k=\tfrac1L$ (choice that maximizes $d_{k+1}$) and $d_0=0$ leads to $d_k=k$ and to the algorithm
			\begin{equation*}
			\begin{aligned}
			y_{k+1}&=\tfrac{k}{k+1} y_k+\tfrac{1}{k+1}x_k,\\
			x_{k+1}^{({i_k})}&=x_k-\tfrac{1}{L}f_{{i_k}}'(y_{k+1}),
			\end{aligned}
			\end{equation*}	
			for which the bound $\E f(y_{k})-f_\star\leq \tfrac{L\normsq{x_0-x_\star}}{2k}$ holds for all $k\geq 0$.
			
			\section{Conclusion}\label{sec:ccl}
			In this work, we showed how to adapt the performance estimation approach to obtain potential-based proofs. Given a first-order methods and a class of (quadratic) potential functions \edit{and predefined numbers of iterations}, the methodology allows obtaining the \emph{best} worst-case guarantees that can be obtained by a potential-based approach with a given structure---choosing an appropriate structure is therefore the critical point. Hence, if the methodology fails to provide the user a satisfactory worst-case bound, the only possible alternatives for improving the results are to either (i) enrich the class of potential functions, or (ii)~add assumptions, or change the problem class. This methodology has the advantage of quickly allowing to assess feasibility of new ideas and to develop simple algorithms for new settings. 
			
			Although provided only for unconstrained minimization, the methodology allows dealing with many other settings such as projection, linear optimization operators (a.k.a., Frank-Wolfe or conditional gradient oracles), proximal terms, deterministic noise (bias) and so on. For using the framework, the only requirement is the ability to formulate the verifications of potential inequalities of type ``$\E\phi_{k+1}^f\leq \phi_k^f+e_k$'' (or sufficient conditions for satisfying it) in a tractable way---and this can be done for many optimization settings~\citep{taylor2017exact} and standard operator classes~\citep{ryu2018operator} (e.g., for studying fixed-point iterations for monotone inclusion problems). The current work focuses on smooth problems without strong convexity, but the same tools can be used when strong convexity (or related notions for obtaining linear convergence results, see e.g., \citet{necoara2018linear,karimi2016linear}) is involved\edit{, as in e.g.,~\cite{moulines2011non,nguyen2018new}, where the norms of the stochastic gradients are not assumed to be uniformly bounded}. Finally, minor adaptations allow studying algorithms specifically designed for finite sums problems (see e.g.,~\citet{roux2012stochastic,johnson2013accelerating,defazio2014saga,schmidt2017minimizing,allen2017katyusha,zhou19c}).

			\paragraph{Acceleration and algorithmic design.} In \appref{sec:paramSelec}, \appref{sec:opt_smooth}, \appref{sec:sgd_fast}, \appref{sec:overp_param_selec}, we discuss techniques for automatic step-size selection in different settings. This is done by adapting the constructive approach to efficient first-order methods by~\citet{drori2018efficient} to deal with potential functions. In a few words, the idea is to study methods with \emph{unrealistic line-search procedures} and to deduce, from the analysis, step-size policies for methods of type~\eqref{eq:algo} that enjoys the same worst-case guarantees. The technique is also inspired by historical developments related to accelerated first-order methods~\citep{nemirovski1982orth,Nesterov:1983wy}.		
			
			\paragraph{Application to proximal/projected methods.} The methodology extends to projected and proximal settings, as previously used in the performance estimation literature~\citep{drori2014contributions,taylor2017exact,taylor2018exact}. As an example, we provide a corresponding potential function for the proximal gradient method in \appref{sec:pgm}.
			
			\paragraph{Application to coordinate descent.} We illustrate the application of technique to coordinate descent-type schemes in \appref{sec:coordinatedescent}. The assumptions used here differ from standard ones~\citep{nesterov2012efficiency}, but allows a unified treatment of this kind of methods. We also use this example for illustrating the incorporation of strong convexity within the framework.
			
			\section*{Codes}
			{The codes used to generate and validate the results are available at\\ {\centering \href{https://github.com/AdrienTaylor/Potential-functions-for-first-order-methods}{github.com/AdrienTaylor/Potential-functions-for-first-order-methods}.}}
			
			\section*{Acknowledgments}\label{s:acks}
			{The authors would like to thank Nicolas Flammarion, Fran\c{c}ois Glineur and Yurii Nesterov for insightful discussions related on the one hand to stochastic gradient methods and on the other one to potential functions and estimate sequences. \edit{The authors also thank the three anonymous referees for their constructive remarks which helped improving the quality of this paper.} The authors acknowledge support from the European Research Council (grant SEQUOIA 724063).} \correc{The authors are also extremely grateful to \href{https://konstmish.github.io/}{Konstantin Mishchenko} and \href{https://scholar.google.pt/citations?user=811XPMsAAAAJ&hl=en&oi=sra}{Jo\~{a}o Miguel Oliveira Domingos} for reporting a few typos and for suggesting a few corrections and improvements to the text and to the appendix.}

			\bibliographystyle{plainnat}
			\bibliography{bib}

			\clearpage 
			
			\appendix
			
			\clearpage 
			
			\section{How to read the appendix}\label{sec:gen}

			In this section, we provide a few keys for going through the appendix.
			
			\paragraph{How to read the appendix?} 	 Those additional sections provide proofs that were not presented in the core of the paper, and complementary examples of applications. The full content of the appendix is listed in Table~\ref{tab:appendix_summary}.
			
			The appendix is divided in a few sections: each of them focuses on a single optimization setup. For example, \appref{sec:boundedvariances} focuses on stochastic methods under a bounded variance assumption $\E_i\normsq{G(x;i)-f'(x)}\leq \sigma^2$. In each section, we start by presenting the proofs that were not done in the core part of the text (see next paragraphs for discussions on how those proofs were found). Then, for the first few settings, we provide the  derivations of the corresponding linear matrix inequalities and the parameter selection technique.

			\paragraph{Going through the proofs.} The proofs presented in the sequel where \emph{computer-generated}, by numerically solving~\eqref{eq:tt}. They all consists in the exact same ideas: reformulating weighted sums of inequalities. In order to generate the proofs, we mostly used specific inequalities; the so-called \emph{interpolation inequalities}~\citep{taylor2017exact,taylor2017smooth}; for any $L$-smooth $\mu$-strongly convex function $f$ (notation $f\in\FmuL$), those inequalities can be written as
			\begin{equation*}
			\begin{aligned}
			f(x)\geq f(y)&+\inner{f'(y)}{x-y}\\ &+\tfrac{1}{2\left(1-\tfrac\mu L\right)}\left(\tfrac1L\normsq{f'(x)-f'(y)}+\mu\normsq{x-y}-2\tfrac\mu L\inner{f'(x)-f'(y)}{x-y}\right),
			\end{aligned}
			\end{equation*}
			for all $x,y\in\Rd$; whereas in the $L$-smooth convex case they simplify to
			\begin{equation*}
			\begin{aligned}
			f(x)\geq f(y)&+\inner{f'(y)}{x-y}+\tfrac{1}{2L}\normsq{f'(x)-f'(y)}.
			\end{aligned}
			\end{equation*}		
			This choice is essentially motivated by the fact those inequalities are key for reformulating~\eqref{eq:tt} in a tractable way. This is explained in e.g.,~\appref{sec:Vk_pgm} were we used them for formulating the linear matrix inequalities for the gradient method.
			
			In order to simplify most proofs, we could often directly replace some of those \emph{interpolation inequalities} encoding smoothness and convexity by appropriate uses of either simple convexity inequalities, or with the descent lemma (which are both \emph{weaker} than interpolation conditions):
			\begin{equation*}
			\begin{aligned}
			f(x)&\geq f(y)+\inner{f'(y)}{x-y},\\
			f(x)&\geq f(y)-\inner{f'(x)}{y-x}-\tfrac{L}2\normsq{y-x}.
			\end{aligned}
			\end{equation*}
			
			\paragraph{Obtaining and verifying the proofs.} The proofs that were \emph{computer-aided} may seem quite mysterious. However, they can be verified in a systematic manner (essentially verifying that the claimed result can be rewritten as the given weighted sum of inequalities). The weights used in those weighted sums essentially correspond to dual variables used in our reformulation of the problem~\eqref{eq:tt}, and can be either guessed based on numerical solutions (see e.g., \appref{sec:Vk_pgm} for the gradient method; for all interpolation inequalities, the weights in the weighted sum is equal to a feasible choice of the corresponding dual variable $\lambda_{i,j}$), or obtained through symbolic computations. Finally, it is possible to \emph{validate} them numerically e.g., by formulating~\eqref{eq:tt} via the performance estimation toolbox~\citep{taylor2017performance}.

			\edit{\paragraph{Computational cost of the approach} The computational complexity of the approach can be deduced from that of semidefinite programming, see e.g., discussions in~\citep{vandenberghe2005interior}. The resulting complexities depend on (i) the structure of the potentials, (ii) the stochastic setting, (iii) the specific method under consideration, and (iv) whether we use the specific structure of the SDP for solving it. Therefore, we do not provide those complexities in the discussions, and rather give the computational time required to execute the different examples using one of the current state of the art solver~\citep{Article:Mosek} on a laptop, along with the sizes of the LMIs at hand.}						
			\begin{table}[!ht]
				\begin{center}
					{\renewcommand{\arraystretch}{1.3}
						\begin{tabular}{@{}ll@{}}
							\specialrule{2pt}{1pt}{1pt}
							Section & Content \\
							\specialrule{2pt}{1pt}{1pt}
							\edit{\appref{app:rel_works}} & \edit{Existing methodologies for computer-assisted worst-case analyses.}\\
							\hline
							\appref{sec:paramSelec} & High-level explanation of our proposed parameter selection technique.\\
							\hline
							\multirow{5}{.24\linewidth}{\appref{sec:GM_app} (no noise)} &  Proof for the potential for gradient descent (\appref{sec:GM_pot}).\\
							& Potential for the proximal gradient method (\appref{sec:pgm}).\\
							& Automated design of accelerated methods, parameter selection (\appref{sec:opt_smooth}).\\
							& Linear matrix inequalities for gradient method (\appref{sec:Vk_pgm}).\\
							& Linear matrix inequalities for parameter selection (\appref{sec:LMI_design}).\\
							\hline
							\multirow{7}{.24\linewidth}{\appref{sec:boundedvariances}\\(bounded variance)} & Potential for stochastic gradient descent (\appref{sec:sgd_pot}).\\
							& Potential for stochastic gradient descent with averaging (\appref{sec:avgsgd_pot}).\\
							& Potential for stochastic gradient descent with primal averaging (\appref{sec:primalavg_pot}).\\
							& Stochastic gradient evaluated at averaged iterate (\appref{sec:alt_pavging}).\\
							& Convergence rates (\appref{sec:BV_rates}).\\
							& Momentum, dual averaging and line-searches (\appref{sec:sgd_fast}).\\
							& Linear matrix inequalities (\appref{sec:SDP_formulation_bv}).\\
							\hline
							\multirow{3}{.23\linewidth}{\appref{sec:overparameterized}\\(over-parametrization)} & Potential for primal averaging (\appref{sec:paving_pot_overp}). \\
							& Parameter selection (\appref{sec:overp_param_selec}).\\
							& Linear matrix inequalities (\appref{sec:overp_SDP}).\\
							\hline 
							\multirow{2}{.2\linewidth}{\appref{sec:weakgrowth_proofs}\\(weak growth)} & Primal averaging under weak growth conditions (\appref{sec:wcg_pavging}).\\
							& \\
							\hline 
							\multirow{2}{.2\linewidth}{\appref{sec:boundedVaratOpt}\\(variance at $x_\star$)} & Primal averaging under bounded variance at $x_\star$ (\appref{sec:boundedVaratOpt}).\\
							& \\
							\hline
							\multirow{2}{.2\linewidth}{\appref{sec:coordinatedescent}\\(block-coordinate)} & Potential for randomized block-coordinate descent (\appref{sec:CD_sublin}).\\
							& Potential with strong convexity (\appref{sec:CD_lin}).\\
							\specialrule{2pt}{1pt}{1pt}
						\end{tabular}
					}
				\end{center}
				\caption{Organization of the appendix.}\label{tab:appendix_summary}
			\end{table}

			\clearpage 
			
			\edit{\section{Methodologies for computer-assisted worst-case analyses}\label{app:rel_works}
			
			As introduced in~\secref{sec:prior}, two competing strategies rely on using semidefinite programming for studying worst-case performances of first-order methods.
			\begin{itemize}
				\item First, performance estimation problems (PEPs) were introduced by~\cite{drori2014performance}. This methodology relies on formulating the worst-case performance of $N$ iterations of a given first-order method as the solution to an optimization problem. One of the key advantage of this methodology is that it is guaranteed to provide \emph{non-improvable} (or tight) results, due to lossless semidefinite reformulations~\citep{taylor2017smooth}, while being applicable to a wide range of settings~\citep{drori2014contributions,taylor2017exact}. This methodology was initially tailored for studying methods with sublinear convergence rates, but linear rates can be obtained as well, through smaller SDPs~\citep{taylor2018exact,ryu2018operator}. The methodology is available through the performance estimation toolbox~\citep{taylor2017performance}, and was used to develop optimized methods~\citep{drori2014performance,drori2016optimal,kim2016optimized,kim2018another,kim2018generalizing,kim2018optimizing,drori2018efficient} and lower bounds~\citep{drori2017exact}.
				
				The main attractive features of this framework are that (i) feasible points to primal PEPs correspond to lower bounds (i.e., functions) on which the given algorithms behave badly, whereas (ii) feasible points to the dual PEPs correspond to upper bounds on the worst-case performance of the given methods. The main inherent difficulty is to convert numerics into analytical proofs. This is mostly due to the fact all iterations are treated at once, which typically implies playing with large semidefinite matrices. Those matrices may scale particularly badly in complicated optimization settings, such as in stochastic setups. 
				
				\item The second approach is based on \emph{integral quadratic constraints} (IQCs). Their uses for studying optimization methods is due to~\cite{lessard2016analysis}. This framework was initially tailored for studying settings with linear convergence rates, through the use of smaller SDPs.
				Recent works formally linked IQCs with performance estimation~\citep{taylor2018lyapunov}, which can be seen as feasible points to PEPs specifically designed to look for Lyapunov functions with the smallest possible linear convergence rate. 
				
				Recent works established that IQCs could also be used for sublinear rates~\citep{hu2017dissipativity,fazlyab2018analysis}, to study stochastic methods~\citep{hu2017analysis,hu2017unified} and to design new first-order algorithms~\citep{van2018fastest,cyrus2018robust}.
			\end{itemize}
				In this work, we use the exact same technique as in PEPs for performing the lossless verification of the ``potential inequality''~\eqref{eq:pot}. The resulting proofs are much simpler, while keeping some \emph{a priori} guarantees, as we know \emph{a priori} that there is no way of ending up with better worst-case guarantees with the same potential-based proof structure. This allows using the methodology with e.g., randomness and stochasticity.
				
				Compared to IQCs, the approach taken here allows (i) studying sublinear rates with more general types of potential functions, beyond standard guarantees on $f(x_N)-f_\star$ and $\norm{x_0-x_\star}$, while only having to specify the content of the potential, (ii) obtaining theoretical guarantees of ending up with the best possible worst-case bounds (for the given class of potentials) for the chosen number of iterations, and (iii) allows dealing with many models involving randomness and stochasticity, even without strong convexity. The lossless reformulation is a key feature of our approach, as it allows guaranteeing that the methodology cannot fail verifying a potential function that is true. Hence, failure is still informative in understanding the behavior of the algorithm at hand.}
			
			\clearpage 
			\section{A parameter selection technique}\label{sec:paramSelec}
			
			In this section, we provide a high-level overview of a technique we used for performing automatic step-size selection, whereas the details are delayed to the following sections where the technique is used. The main idea is to try to optimize algorithmic parameters for improving worst-case performance guarantees. However, from the derivations of the linear matrix inequalities of next sections, finding at the same time a valid sequence of potentials and an optimized sequence of steps actually requires solving a set of bilinear matrix inequalities (BMIs), which are intractable in general~\citep{toker1995np}. One way to work around this difficulty is to study a variant of algorithm~\eqref{eq:algo} where the parameters $\{(\alpha_k,\alpha_k',\beta_k,\beta_k',\delta_k,\gamma_k,\gamma_k',\epsilon_k) \}_k$ are chosen by appropriate span-search procedures. One can then make use of the technique developed by~\citet{drori2018efficient} for formulating~\eqref{eq:tt} (most of the time \emph{relaxed} versions of it) into a LMI (feasibility problem), with the particularity that for any feasible point to this LMI, one can reconstruct an algorithm of the form~\eqref{eq:algo} (without span-searches) that achieves the same performances.	
			
			The technique relies on two elements:
			\begin{itemize}
				\item choice of an \emph{idealized} algorithm, typically using (possibly unrealistic) line-searches,
				\item choice of a family of potentials that easily allows optimizing the algorithmic parameters while looking for a sequence of valid potentials. The parameters that can not be optimized are replaced by (possibly unrealistic, see below) line-search procedures.
			\end{itemize}
			The technique is inspired by~\citet{drori2018efficient} and original developments related to accelerated methods~\citet{nemirovski1982orth,Nesterov:1983wy}. The main thing to keep in mind is that we would ideally want to optimize the algorithmic parameters for improving its worst-case guarantees. For explaining the strategy, let us consider the following example---which we carry out in \secref{sec:opt_smooth}---: consider the first-order method given by
			\begin{equation*}
			\begin{aligned}
			y_{k+1}&= (1-\tau_k) x_k+\tau_k z_k,\\
			x_{k+1}&=y_{k+1}-\alpha_k f'(y_{k+1}),\\
			z_{k+1}&=(1-\delta_k) y_{k+1}+\delta_k z_k -\gamma_k f'(y_{k+1}),\\
			\end{aligned}
			\end{equation*}
			for which we wish to optimize parameters $\{(\tau_k,\alpha_k,\delta_k,\gamma_k)\}_k$. We also consider a specific family of potentials (discussed hereafter):
			\begin{equation}\label{eq:FGM_pot_choice}
			\phi_k^f=\begin{pmatrix}x_k-x_\star\\ f'(x_k)\end{pmatrix}^\top \left[Q_k\otimes I_d\right]\begin{pmatrix}x_k-x_\star\\ f'(x_k)\end{pmatrix} + d_k\, (f(x_k)-f_\star)+a_k \normsq{z_k-x_\star},
			\end{equation}
			with $Q_k\in\Sb^2$ (space of $2\times 2$ symmetric matrices), and picking $\phi_0^f=\tfrac{L}2 \normsq{x_0-x_\star}$ and $\phi_N^f=d_N \,(f(x_N)-f_\star)$. As our goal is to optimize the parameter schedule $\{(\tau_k,\alpha_k,\delta_k,\gamma_k)\}_k$, a natural thing to try is to solve
			\begin{equation}\label{eq:FGM_lyap_opt_coefs}
			\max_{\{(\tau_k,\alpha_k,\delta_k,\gamma_k)\}_k}\quad \max_{\phi_1^f,\hdots,\phi_{N-1}^f,d_N} d_N \st (\phi_0^f,\phi_1^f)\in\V_0,\hdots,(\phi_{N-1}^f,\phi_N^f)\in\V_{N-1}.
			\end{equation}
			However, although there might be other workarounds, this problem turns out to have $N$ BMIs. Instead of solving this problem, the workaround we propose is to study the algorithm
			\begin{equation}\label{eq:FGM_ELS}
			\begin{aligned}
			y_{k+1}&=\argm{x}{f(x)\, \st\, x\in x_k+\sspan\{z_k-x_k\}},\\
			x_{k+1}&=\argm{x}{f(x)\, \st\, x\in y_{k+1}+\sspan \{f'(y_{k+1})\}},\\
			z_{k+1}&=(1-\delta_k) y_{k+1}+\delta_k z_k -\gamma_k f'(y_{k+1}),\\
			\end{aligned}
			\end{equation}
			for which one can formulate relaxed versions of~\eqref{eq:tt} (i.e., sufficient conditions for verifying a potential) using ideas developed below. By denoting $\tilde{\V}_k$ the set of pairs $(\phi_k^f,\phi_{k+1}^f)$ of potentials that can be verified for~\eqref{eq:FGM_ELS} with our sufficient conditions (see below), we propose to solve the following alternative to~\eqref{eq:FGM_lyap_opt_coefs}:
			\begin{equation}\label{eq:FGM_lyap_ELS}
			d_N^{(\textrm{LSearch})}=\max_{\{(\delta_k,\gamma_k)\}_k}\quad \max_{\phi_1^f,\hdots,\phi_{N-1}^f,d_N} d_N \st (\phi_0^f,\phi_1^f)\in\tilde{\V}_0,\hdots,(\phi_{N-1}^f,\phi_N^f)\in\tilde{\V}_{N-1},
			\end{equation}
			from which one can recover a policy $\{(\tau_k,\alpha_k,\delta_k,\gamma_k)\}_k$ with the same $d_N^{(\textrm{LSearch})}$ is attained, as illustrated below. All steps involved in the analysis of~\eqref{eq:FGM_ELS}---similar in spirit with those presented for vanilla gradient in \secref{sec:GM}--- are presented in \secref{sec:opt_smooth}. Before going to the next section, note that it is straightforward to verify $\tilde{\V}_k\subseteq\V_k$, as for any pair $(\phi_k^f,\phi_{k+1}^f)$ we have
			\[ (\phi_k^f,\phi_{k+1}^f) \in\tilde{\V}_k \Rightarrow\, (\phi_k^f,\phi_{k+1}^f) \in {\V}_k\]
			(in other words, all potentials that can be verified are potentials). However, in general we do not have $\tilde{\V}_k={\V}_k$, meaning that the analysis can fail even though a good sequence of potentials with the desired structure might exist---given a chosen structure of potentials, this problem is not present in the analysis framework presented in the core of the paper. Another reason why it might fall is when the method with line-search does not have nice worst-case guarantees.
			\paragraph{Transforming line-search procedures to fixed-step policies.} Let us provide an example of the use of a method for designing a gradient method with optimal step-size for smooth strongly convex minimization, as provided by~\citet{de2017worst}. That is, consider the problem of minimizing a smooth strongly convex function $f\in\FmuL$. We show how to let the computer choose an appropriate step-size $\delta_k$ in a gradient descent scheme, by studying the line-search variant
			\[x_{k+1}=\argm{x}{f(x)\, \st \, x\in x_k+\sspan\{f'(x_{k}) \}} .\] For keeping things as simple as possible, let us proceed with the potential: $\phi_k^f=d_k (f(x_k)-f_\star)$; in the following lines, we illustrate the technique for choosing the step-size $\delta_k$ achieving $d_{k+1}=\rho^{-1}d_k$ (with $0<\rho<1$ being the convergence rate) with the smallest possible $\rho$. Let us first note that the rate of convergence $\rho$ of the line-search procedure for the family of potentials that we chose satisfies by definition
			\begin{equation*}
			\begin{aligned}\rho \defeq \max_{\substack{x_k,x_{k+1},\\ f\in\FmuL}}&\frac{f(x_{k+1})-f_\star}{f(x_k)-f_\star} \\ &\st x_{k+1} = x_k-\delta_k f'(x_k),\, \delta_k=\argm{\delta_k\in\R}{f(x_k-\delta_k f'(x_k))},
			\end{aligned}
			\end{equation*}
			\correc{(note that we optimize over $f\in\FmuL$ irrespective of the ambient dimension of worst-case problems (see notations in  \secref{s:prelim}), that is, the maximization problem could be formulated more explicitly in the variables $x_{k},x_{k+1}\in\Rd$, $f\in\FmuL(\Rd)$ and $d\in\mathbb{N}$)} which can be upper-bounded by  (using optimality conditions of the line-search procedure)
			\begin{equation}\label{eq:ELS}
			\rho \leq \max_{\substack{x_k,x_{k+1},\\ f\in\FmuL}}\frac{f(x_{k+1})-f_\star}{f(x_k)-f_\star} \st \inner{f'(x_{k+1})}{f'(x_k)}= 0,\ \inner{f'(x_{k+1})}{x_{k+1}-x_k} = 0.
			\end{equation}
			It turns out that~\eqref{eq:ELS} often holds with equality, motivating the following developments. {As a next step, one can then get an upper bound from the use of a Lagrangian relaxation with $\lambda_1,\lambda_2\in\R$:
				\[\rho\, \leq\, \bar{\rho}(\lambda_1,\lambda_2) \defeq \max_{\substack{x_k, x_{k+1},\\ f\in\FmuL}} \left\{\frac{f(x_{k+1})-f_\star}{f(x_k)-f_\star} +\lambda_1 \inner{ f'(x_{k+1})}{ f'(x_k)}+\lambda_2 \inner{f'(x_{k+1})}{x_{k+1}-x_k}\right\},\]}
			{and from the same pair $(\lambda_1,\lambda_2)$, one can create an intermediary problem
				\begin{equation*}
				\begin{aligned}
				\rho &\leq{\max_{\substack{x_k,x_{k+1},\\ f\in\FmuL}}\left\{\frac{f(x_{k+1})-f_\star}{f(x_k)-f_\star} \st \lambda_1\inner{f'(x_{k+1})}{f'(x_k)} +\lambda_2\inner{f'(x_{k+1})}{x_{k+1}-x_k} = 0\right\}}\\ &\leq \bar{\rho}(\lambda_1,\lambda_2).
				\end{aligned}
				\end{equation*}
				\[\]}
			So, for any pair $\lambda_1,\lambda_2\in\R$ we get:
			\begin{itemize}
				\item an upper bound $\bar{\rho}(\lambda_1,\lambda_2)$ (possibly $+\infty$ if the choice for $\lambda_1,\lambda_2$ was not appropriate) on $\rho$,
				\item as a consequence, all methods satisfying $\inner{ f'(x_{k+1})}{ \lambda_1 f'(x_k)+ \lambda_2  (x_{k+1}-x_k)} = 0$ benefits from  convergence rate at most $\bar{\rho}(\lambda_1,\lambda_2)$. In particular, the method $x_{k+1}=x_k-\tfrac{\lambda_2}{\lambda_1}f'(x_k)$ satisfies the previous equality.
				\item Finally, if there exists a choice $\lambda_1^\star,\lambda_2^\star$ such that	\[\rho=\bar{\rho}(\lambda_1^\star,\lambda_2^\star),\]
				then, assuming $\lambda_2^\star\neq 0$ (this is reasonable as otherwise $\inner{f'(x_{k+1})}{x_{k+1}-x_k} = 0$ would not be used in the analysis), the method $x_{k+1}=x_k-\tfrac{\lambda_1^\star}{\lambda_2^\star}f'(x_k)$ benefits from the same worst-case guarantee as the line-search procedure. This phenomenon actually (maybe surprisingly) occurs at least in standard smooth and non-smooth convex optimization settings~\citep{drori2018efficient}. In the following sections, we use the same strategy in slightly more complicated settings---where we do not check whether formulations corresponding to~\eqref{eq:ELS} should hold with equality or not. 
			\end{itemize}
			
			\paragraph{Choice of the potential family.} The choice of the family of potentials into consideration plays a crucial role in the parameter selection process. For example, choice~\eqref{eq:FGM_pot_choice} for method~\eqref{eq:FGM_ELS} allows to easily optimize the coefficients $\delta_k$ and $\gamma_k$. The reason is technical (everything lies in the LMI formulation, which we delay to later sections), but the consequence is relatively simple: the fact $z_{k}-x_\star$ appears only in a norm allows writing the LMI reformulation of the potential inequality ``$\phi_{k+1}^f\leq\phi_k^f$'' in a way that easily permits to optimize both $\delta_k$ and $\gamma_k$ (see \appref{sec:LMI_design}). 
			
			\paragraph{Details and extensions.} We provide detailed developments relying on this technique in later sections. 		In its simplest form, the technique is illustrated in the deterministic smooth convex minimization setting in \secref{sec:opt_smooth}, among others on algorithm~\eqref{eq:FGM_lyap_opt_coefs}.
			
			The technique can also be adapted to other settings, such as stochastic methods, coordinate descent, or finite sums. Examples that are not provided here include the application to noise models satisfying strong growth conditions; in the latter case, the parameter selection technique allows obtaining accelerated methods similar to those of~\citet{vaswani2018fast}.

			\clearpage 
			
			\section{Gradient method}\label{sec:GM_app}
			In this section, we present the proof of Theorem~\ref{thm:GM_Lyap} for the gradient method. For simplicity, we present the result for vanilla gradient method in the core of the paper; the proximal case is presented in \appref{sec:pgm} and the semidefinite reformulation of~\eqref{prop:verify_potential} in \appref{sec:Vk_pgm}. Finally, we present the parameter selection technique to devise variants of accelerated gradient methods in \appref{sec:opt_smooth}. The corresponding LMIs are presented in \secref{sec:LMI_design}.
			
			The codes implementing the LMI formulations and numerics presented hereafter are provided in \secref{sec:ccl}.
			\subsection{Proof of Theorem~\ref{thm:GM_Lyap}}\label{sec:GM_pot}
			The proof follows the same lines as previous works on performance estimation problems (see e.g.,~\citet[Section 4]{de2017worst}), and only consists in reformulating a linear combination of inequalities.
			\begin{proof} Combine the following inequalities with their corresponding weights.
				\begin{itemize}
					\item Convexity and smoothness between $x_k$ and $x_\star$ with weight $\lambda_1=2 L$
					\[f_\star\geq f(x_k) + \inner{f'(x_k)}{x_\star-x_k}+\tfrac{1}{2L}\normsq{f'(x_k)}, \]
					\item convexity and smoothness between $x_{k+1}$ and $x_k$ with weight $\lambda_2=2d_k+2L(2+b_k)$
					\[f(x_k)\geq f(x_{k+1})+\inner{f'(x_{k+1})}{x_k-x_{k+1}}+\tfrac{1}{2L}\normsq{f'(x_{k+1})-f'(x_k)}, \]
					\item convexity between $x_{k}$ and $x_{k+1}$ with weight $\lambda_3=2L(1+b_k)+d_k$
					\[f(x_{k+1}) \geq f(x_k) +\inner{f'(x_k)}{x_{k+1}-x_k}. \]
				\end{itemize}
				By substituting $x_{k+1}=x_k-\tfrac1L f'(x_k)$, one can easily verify that the corresponding weighted sum can be reformulated exactly as the desired result:
				\begin{equation*}
				\begin{aligned}
				0\geq &\lambda_1 \left[f(x_k)-f_\star+ \inner{f'(x_k)}{x_\star-x_k}+\tfrac{1}{2L}\normsq{f'(x_k)}\right]\\&+\lambda_2 \left[f(x_{k+1})-f(x_k)+\inner{f'(x_{k+1})}{x_k-x_{k+1}}+\tfrac{1}{2L}\normsq{f'(x_{k+1})-f'(x_k)}\right]\\ & +\lambda_3 \left[f(x_k)-f(x_{k+1}) +\inner{f'(x_k)}{x_{k+1}-x_k} \right] \\
				=& (d_k+2L)(f(x_{k+1})-f_\star)+L^2\normsq{x_{k+1}-x_\star}+(2+\tfrac{d_k}{L}+b_k)\normsq{f'(x_{k+1})}\\&-d_k(f(x_k)-f_\star)-L^2\normsq{x_k-x_\star}-b_k\normsq{f'(x_k)},
				\end{aligned}
				\end{equation*}
				leading to
				\begin{equation*}
				\begin{aligned}
				(d_k+2L)&(f(x_{k+1})-f_\star)+L^2\normsq{x_{k+1}-x_\star}+(2+\tfrac{d_k}{L}+b_k)\normsq{f'(x_{k+1})}\\
				&\leq d_k(f(x_k)-f_\star)+L^2\normsq{x_k-x_\star}+b_k\normsq{f'(x_k)}.
				\end{aligned}
				\end{equation*}
			\end{proof}
			\subsection{Proximal gradient method}\label{sec:pgm}
			In this section, we consider the problem
			\[\min_{x\in\Rd} \left\{F(x)\equiv f(x)+h(x)\right\}, \]
			where $f:\Rd\rightarrow\R $ is convex and smooth ($f\in\FL$) and $h:\Rd\rightarrow\R\cup\{\infty\}$ is closed, proper and convex (notation $h\in\Fccp$), where the proximal operator of $h$ is readily available:
			\[\prox{\gamma h}{y}=\argm{x\in\Rd}{\gamma h(x)+\tfrac12\normsq{y-x}}.\]
			Nearly the same potential as that for the gradient method holds for the proximal gradient method (note that we have assumed $\dom\,f=\Rd$ for simplicity; other inequalities have to be used if it is not the case~\citep{de2017worst2,drori2018properties}). 
			\begin{theorem}{{\bf\color{red}[Updated: cleaner statement of the theorem; proof does not change]}}\label{thm:PGM_Lyap}
				Let $f\in\FL$, $h\in\Fccp$ (class of closed, proper and convex functions), \edittwo{let $x_k\in\Rd$ be satisfying $\partial h(x_k)\neq\emptyset$}, and $x_{k+1}=\prox{h/L}{x_k-\tfrac1L f'(x_k)}$. {\color{red} For all $b_k,d_k\geq0$ and $F'(x_{\color{red}k})\in\partial F(x_{\color{red}k})$, there exists $F'(x_{k+1})\in\partial F(x_{k+1})$} such that inequality $\phi_{k+1}^F(x_{k+1})\leq \phi_k^F(x_{k})$ holds with 
				\[\phi_k^F=d_k (F(x_k)-F_\star)+b_k\normsq{F'(x_k)}+ L^2\normsq{x_k-x_\star},\]
				{\color{red} and }$b_{k+1}=1+\tfrac{d_{k}}{L}+b_k$ and $d_{k+1}= 2 L+d_k$. In particular, the above inequality holds when choosing $F'(x_{k+1})=L(x_{k}-x_{k+1})-f'(x_{k})+f'(x_{k+1})$ (this choice is natural as it corresponds to using the particular subgradient of $h$ that was used in the proximal operation).
			\end{theorem}
			In particular, the previous theorem establishes that {{\bf\color{red}[Updated: missing $L$]}}
			\[ \phi_k^F = (2k+1){\color{red}L} (F(x_k)-F(x_\star))+k(k+1)\normsq{F'(x_k)}+L^2\normsq{x_k-x_\star},\]
			is a potential for the proximal gradient method with step-size $1/L$.
			\begin{proof} Combine the following inequalities with their corresponding weights. We denote by $s_{k+1}\in\partial h(x_{k+1})$ the specific subgradient used in the proximal operation, i.e., such that $x_{k+1}=x_k-\tfrac1L(f'(x_k)+s_{k+1})$, and $s_k\in\partial h(x_k)$ any subgradient of $h$ at $x_k$. We also specifically choose $F'(x_{k+1})=f'(x_{k+1})+s_{k+1}\in\partial F(x_{k+1})$ and $F'(x_k)=f'(x_k)+s_k\in \partial F(x_k)$.
				\begin{itemize}
					\item Convexity of $f$ between $x_k$ and $x_\star$ with weight $\lambda_1=2 L$
					\[f_\star\geq f(x_k) + \inner{f'(x_k)}{x_\star-x_k}, \]
					\item convexity and smoothness of $f$ between $x_{k+1}$ and $x_k$ with weight $\lambda_2=2d_k+2L(1+b_k)$
					\[f(x_k)\geq f(x_{k+1})+\inner{f'(x_{k+1})}{x_k-x_{k+1}}+\tfrac{1}{2L}\normsq{f'(x_{k+1})-f'(x_k)}, \]
					\item convexity of $f$ between $x_{k}$ and $x_{k+1}$ with weight $\lambda_3=2L b_k+d_k$
					\[f(x_{k+1}) \geq f(x_k) +\inner{f'(x_k)}{x_{k+1}-x_k}, \]
					\item convexity of $h$ between $x_{k+1}$ and $x_\star$ with weight $\lambda_4=2L$
					\[h(x_\star)\geq h(x_{k+1})+\inner{s_{k+1}}{x_\star-x_{k+1}},\]
					\item convexity of $h$ between $x_k$ and $x_{k+1}$ with weight $\lambda_5=2L b_k+d_{\color{red}k}$
					\[h(x_{k})\geq h(x_{k+1})+\inner{s_{k+1}}{x_k-x_{k+1}},\]
					\item convexity of $h$ between $x_{k+1}$ and $x_k$ with weight $\lambda_6=2 L b_k$
					\[h(x_{k+1})\geq h(x_k)+\inner{s_k}{x_{k+1}-x_k}.\]
				\end{itemize}
				By substituting $x_{k+1}=x_k-\tfrac1L(f'(x_k)+s_{k+1})$, one can easily verify that the corresponding weighted sum can be reformulated exactly as the desired result plus a positive term:
				\begin{equation*}
				\begin{aligned}
				0\geq &\lambda_1 \left[f(x_k) - f_\star + \inner{f'(x_k)}{x_\star-x_k}\right]\\&+\lambda_2 \left[f(x_{k+1})-f(x_k)+\inner{f'(x_{k+1})}{x_k-x_{k+1}}+\tfrac{1}{2L}\normsq{f'(x_{k+1})-f'(x_k)}\right]\\ & +\lambda_3 \left[f(x_k)-f(x_{k+1}) +\inner{f'(x_k)}{x_{k+1}-x_k} \right] \\
				&+\lambda_4 \left[ h(x_{k+1})-h(x_\star)+\inner{s_{k+1}}{x_\star-x_{k+1}} \right] \\
				&+\lambda_5 \left[h(x_{k+1})-h(x_{k})+\inner{s_{k+1}}{x_k-x_{k+1}}\right] \\
				&+\lambda_6 \left[h(x_k)-h(x_{k+1})+\inner{s_k}{x_{k+1}-x_k}\right] \\
				=& (d_k+2L)(F(x_{k+1})-F_\star)+L^2\normsq{x_{k+1}-x_\star}+(1+\tfrac{d_k}{L}+b_k)\normsq{F'(x_{k+1})}\\&-d_k(F(x_k)-F_\star)-L^2\normsq{x_k-x_\star}-b_k\normsq{F'(x_k)}+b_k \normsq{s_{k+1}-s_k},
				\end{aligned}
				\end{equation*}
				leading to the desired result
				\begin{equation*}
				\begin{aligned}
				(d_k+2L)&(F(x_{k+1})-F_\star)+L^2\normsq{x_{k+1}-x_\star}+(1+\tfrac{d_k}{L}+b_k)\normsq{F'(x_{k+1})}\\&\leq d_k(F(x_k)-F_\star)+L^2\normsq{x_k-x_\star}+b_k\normsq{F'(x_k)}-b_k \normsq{s_{k+1}-s_k},\\
				&\leq   d_k(F(x_k)-F_\star)+L^2\normsq{x_k-x_\star}+b_k\normsq{F'(x_k)}.
				\end{aligned}
				\end{equation*}
			\end{proof}
		
			\subsection{Design of accelerated methods} \label{sec:opt_smooth}
			There are many different variants of accelerated gradient methods (see e.g.,~\citet{tseng2008accelerated}), particularly for smooth unconstrained optimization in the Euclidean setting. Here are two examples that can be obtained through the line-search strategy presented in \secref{sec:paramSelec}.
			
			The codes implementing the LMI formulations and numerics below are provided in \secref{sec:ccl}.
			\paragraph{Design of a first accelerated method.} Let us first apply our step-size selection technique to the following first-order method involving three sequences:
			\begin{equation}\label{eq:fgm_3seqs}
			\begin{aligned}
			y_{k+1}&= (1-\tau_k) x_k+\tau_k z_k,\\
			x_{k+1}&=y_{k+1}-\alpha_k f'(y_{k+1}),\\
			z_{k+1}&=(1-\delta_k) y_{k+1}+\delta_k z_k -\gamma_k f'(y_{k+1}),\\
			\end{aligned}
			\end{equation}
			relying on the alternate version involving line-searches and optimization of the last step:
			\begin{equation}\label{eq:fgm_3seqs_LS}
			\begin{aligned}
			y_{k+1}&=\argm{x}{f(x)\, \st\, x\in x_k+\sspan\{z_k-x_k\}},\\
			x_{k+1}&=\argm{x}{f(x)\, \st\, x\in y_{k+1}+\sspan \{f'(y_{k+1})\}},\\
			z_{k+1}&=(1-\delta_k) y_{k+1}+\delta_k z_k -\gamma_k f'(y_{k+1}).
			\end{aligned}
			\end{equation}
			We provide our LMI encoding sufficient conditions for verifying ``$\phi_{k+1}^f\leq \phi_k^f$'' in the next section (\secref{sec:LMI_design}) for the potential			
			\begin{equation}\label{eq:pot_fgm1}
			\begin{aligned}
			\phi_k^f= &\begin{pmatrix}x_k-x_\star\\ f'(x_k) \end{pmatrix}^\top \left[Q_k\otimes I_d\right]\begin{pmatrix}x_k-x_\star\\ f'(x_k) \end{pmatrix}+a_k \normsq{z_k-x_\star} + d_k\,
			(f(x_k)-f_\star),
			\end{aligned}	
			\end{equation}
			with $Q_k\in\Sb^2$ and the choice $\phi_0^f=\tfrac{L}2\normsq{{\color{red}z_0}-x_\star}$, $\phi_N^f=d_N\, (f(x_N)-f_\star)$. Note that one could add arbitrary sequences to~\eqref{eq:fgm_3seqs}, and other states to $\phi_k^f$ and still use the same tricks (see below for another example).
			
			Let us denote $\tilde{\V}_k$ the set of pairs $(\phi^f_{k+1},\phi_k^f)$ for which we can verify the inequality $\phi^f_{k+1}\leq \phi_k^f$ holds for algorithm~\eqref{eq:fgm_3seqs_LS} (see \secref{sec:LMI_design}).
			Similar in spirit with the approach of \secref{sec:GM} for designing a potential for vanilla gradient descent, we can now simulteneously design potential and a sequence of parameters for~\eqref{eq:fgm_3seqs}, for example by solving
			\begin{equation}\label{eq:AGM_lyap_design}
			\edit{\max_{\{(\delta_k,\gamma_k)\}_k}}\max_{\phi_1^f,\hdots,\phi_{N-1}^f,d_N} d_N\, \st\, (\phi_0^f,\phi_1^f)\in \tilde{\V}_0,\hdots,(\phi_{N-1}^f,\phi_N^f)\in\tilde{\V}_{N-1}.
			\end{equation}
			As an example, we provide the results obtained by solving~\eqref{eq:AGM_lyap_design} for $N=100$ on Figure \ref{fig:design}. Carrying a few simplifications in a similar manner to those from \secref{sec:GM}, one can arrive easily arrive to Theorem~\ref{thm:fgm_smooth} (simpler expressions below).
			\begin{center}
				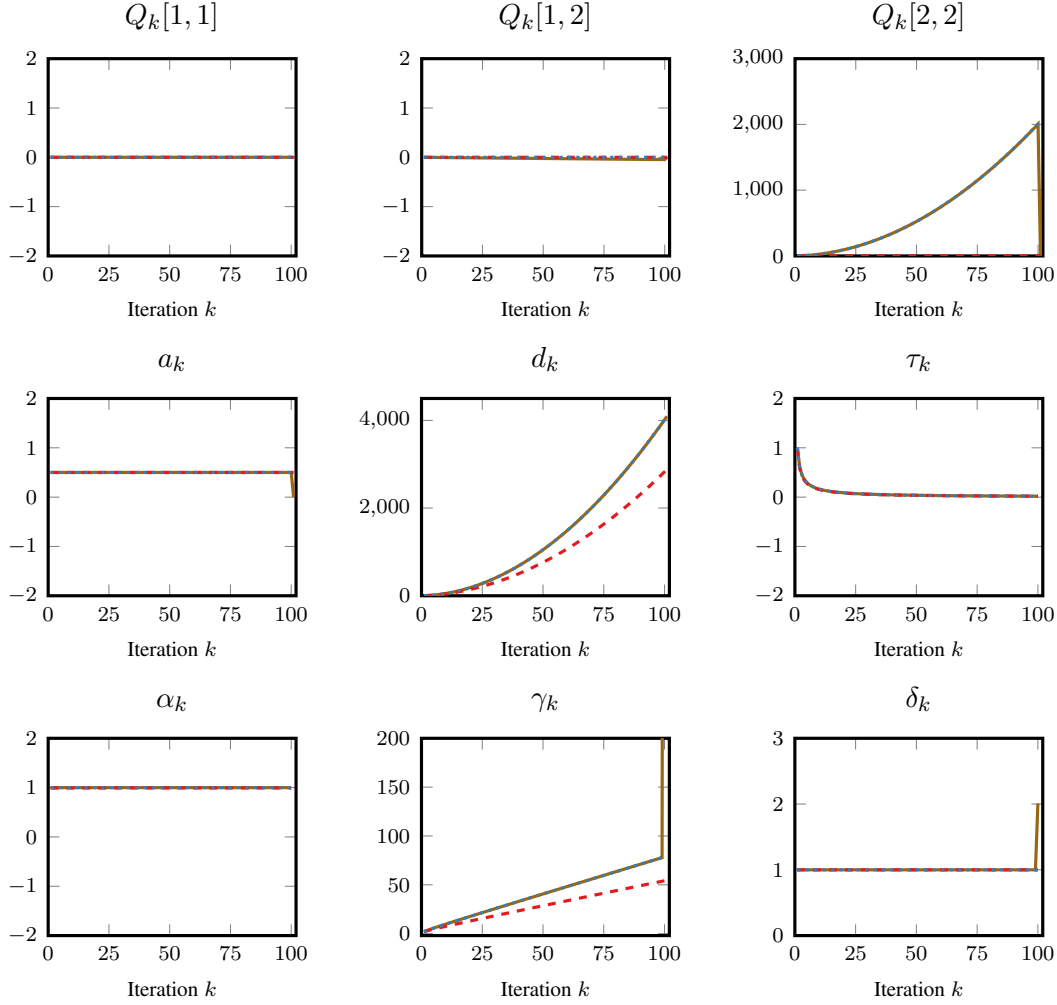
\begin{figure}[!ht]
					\begin{tabular}{rrr}
						\begin{tikzpicture}
						\begin{axis}[plotOptions, title={$Q_k[1,1]$}, ymin=-2, ymax=2,xmin=0,xmax=102,width=.32\linewidth]
						\addplot [colorP5] table [x=k,y=ak] {Data/FGM1_raw.dat};
						\addplot [colorP2, dashed] table [x=k,y=ak] {Data/FGM1_pure.dat};
						\addplot [colorP1, dotted] table [x=k,y=ak] {Data/FGM1_inter.dat};
						\end{axis}
						\end{tikzpicture} &\begin{tikzpicture}
						\begin{axis}[plotOptions, title={$Q_k[1,2]$}, ymin=-2, ymax=2,xmin=0,xmax=102,width=.32\linewidth]
						\addplot [colorP5] table [x=k,y=ck] {Data/FGM1_raw.dat};
						\addplot [colorP2, dashed] table [x=k,y=ck] {Data/FGM1_pure.dat};
						\addplot [colorP1, dotted] table [x=k,y=ck] {Data/FGM1_inter.dat};
						\end{axis}
						\end{tikzpicture} &
						\begin{tikzpicture}
						\begin{axis}[plotOptions, title={$Q_k[2,2]$}, ymin=-2, ymax=3000,xmin=0,xmax=102,width=.32\linewidth]
						\addplot [colorP5] table [x=k,y=bk] {Data/FGM1_raw.dat};
						\addplot [colorP2, dashed] table [x=k,y=bk]{Data/FGM1_pure.dat};
						\addplot [colorP1, dotted] table [x=k,y=bk]{Data/FGM1_inter.dat};
						\end{axis}
						\end{tikzpicture} \\
						\hspace{-.5cm}
						\begin{tikzpicture}
						\begin{axis}[plotOptions, title={$a_k$}, ymin=-2, ymax=2,xmin=0,xmax=102,width=.32\linewidth]
						\addplot [colorP5] table [x=k,y=apk] {Data/FGM1_raw.dat};
						\addplot [colorP2, dashed] table [x=k,y=apk] {Data/FGM1_pure.dat};
						\addplot [colorP1, dotted] table [x=k,y=apk] {Data/FGM1_inter.dat};
						\end{axis}
						\end{tikzpicture}&
						\begin{tikzpicture}
						\begin{axis}[plotOptions, title={$d_k$}, ymin=-2, ymax=4500,xmin=0,xmax=102,width=.32\linewidth]
						\addplot [colorP5] table [x=k,y=dk] {Data/FGM1_raw.dat};
						\addplot [colorP2, dashed] table [x=k,y=dk] {Data/FGM1_pure.dat};
						\addplot [colorP1, dotted] table [x=k,y=dk] {Data/FGM1_inter.dat};
						\end{axis}
						\end{tikzpicture} &
						\begin{tikzpicture}
						\begin{axis}[plotOptions, title={$\tau_k$}, ymin=-2, ymax=2,xmin=0,xmax=102,width=.32\linewidth]
						\addplot [colorP5] table [x=k,y=tauk, skip coords between index={100}{101}] {Data/FGM1_raw.dat};
						\addplot [colorP2, dashed] table [x=k,y=tauk, skip coords between index={100}{101}] {Data/FGM1_pure.dat};
						\addplot [colorP1, dotted] table [x=k,y=tauk, skip coords between index={100}{101}] {Data/FGM1_inter.dat};
						\end{axis}
						\end{tikzpicture}\\
						\begin{tikzpicture}
						\begin{axis}[plotOptions, title={$\alpha_k$}, ymin=-2, ymax=2,xmin=0,xmax=102,width=.32\linewidth]
						\addplot [colorP5] table [x=k,y=alphak, skip coords between index={100}{101}] {Data/FGM1_raw.dat};
						\addplot [colorP2, dashed] table [x=k,y=alphak, skip coords between index={100}{101}] {Data/FGM1_pure.dat};
						\addplot [colorP1, dotted] table [x=k,y=alphak, skip coords between index={100}{101}] {Data/FGM1_inter.dat};
						\end{axis}
						\end{tikzpicture}&
						\begin{tikzpicture}
						\begin{axis}[plotOptions, title={$\gamma_k$}, ymin=-2, ymax=200,xmin=0,xmax=102,width=.32\linewidth]
						\addplot [colorP5] table [x=k,y=gammak, skip coords between index={100}{101}] {Data/FGM1_raw.dat};
						\addplot [colorP2, dashed] table [x=k,y=gammak, skip coords between index={100}{101}] {Data/FGM1_pure.dat};
						\addplot [colorP1, dotted] table [x=k,y=gammak, skip coords between index={100}{101}] {Data/FGM1_inter.dat};
						\end{axis}
						\end{tikzpicture}&
						\begin{tikzpicture}
						\begin{axis}[plotOptions, title={$\delta_k$}, ymin=0, ymax=3,xmin=0,xmax=102,width=.32\linewidth]
						\addplot [colorP5] table [x=k,y=deltak, skip coords between index={100}{101}] {Data/FGM1_raw.dat};
						\addplot [colorP2, dashed] table [x=k,y=deltak,skip coords between index={100}{101}] {Data/FGM1_pure.dat};
						\addplot [colorP1, dotted] table [x=k,y=deltak,skip coords between index={100}{101}] {Data/FGM1_inter.dat};
						\end{axis}
						\end{tikzpicture}
					\end{tabular}\vspace{-.35cm}
					\caption{Numerical solution to~\eqref{eq:AGM_lyap_design} for $N=100$ and $L=1$ (plain brown, large values for $\delta_{100}$ and $\gamma_{100}$ were capped for readability purposes; they are due to the fact we impose no control on $z_N$ with our initial choice $\phi_{N}^f$), forced $a_k=\tfrac{L}2$ and $Q_k=0$ (dashed red), forced $a_k=\tfrac{L}2$, $Q_k[1,1]=Q_k[1,2]=0$ and $Q_{k}[2,2]=\tfrac{d_k}{2L}$ (dotted blue). \edit{Total time: $\sim35$ sec{.} on single core of Intel Core i$7$ $1.8$GHz CPU.} }\label{fig:design}
				\end{figure}
			\end{center}
			\begin{theorem}\label{thm:fgm_smooth} Let $f\in\FL$. For all values of $b_k,d_k\geq0$, the iterates of algorithm~\eqref{eq:fgm_3seqs_LS} with $\delta_k=1$ and $\gamma_k=\tfrac{d_{k+1}-d_k}{L}$ satisfy
			\begin{equation*}
			\begin{aligned}d_{k+1} (f(x_{k+1})-f_\star)&+\tfrac{b_{k+1}}{2L}\normsq{f'(x_{k+1})}+\tfrac{L}{2}\normsq{z_{k+1}-x_\star}\\ &\leq d_k (f(x_k)-f_\star)+\tfrac{b_{k}}{2L}\normsq{f'(x_k)}+\tfrac{L}{2}\normsq{z_k-x_\star},
			\end{aligned}
			\end{equation*}
				for all $d_{k+1},b_{k+1}\in\R$ satisfying $d_{k+1}^2-2 (d_{k}+1) d_{k+1}+\tfrac{d_{k}^2}{b_{k}+d_{k}}+d_{k}^2\leq0$ (reducing to $d_{k+1}\in [1+d_k-\sqrt{1+d_k},\, 1+d_k+\sqrt{1+d_k}]$ when $b_k=0$) and $b_{k+1}\leq d_{k+1}$. In addition, the iterates produced by algorithm~\eqref{eq:fgm_3seqs} with  $\alpha_k=\tfrac1L$, and $\tau_k=\tfrac{d_{k+1}-d_k}{d_{k+1}}$ satisfy the same inequality.
			\end{theorem}
			Before proceeding with the proof, let us note two simple scenarios that are valid for Theorem~\ref{thm:fgm_smooth}:
			\begin{itemize}
				\item the choice $d_0=0$ along with $b_k=0$ and $d_{k+1}=1+d_k+\sqrt{1+d_k}=\bO(k^2)$, reaching acceleration (analytical version numerically matching the red curves on Figure \ref{fig:design}),
				\item the choice $b_0=d_0=0$ along with $b_{k+1}=d_{k+1}=1+d_k + \sqrt{1+\tfrac{3}{2} d_k}=\bO(k^2)$, reaching acceleration (analytical version numerically matching the blue curves on Figure \ref{fig:design}).
			\end{itemize}
			\begin{proof}
				Combine the following inequalities with corresponding weights:
				\begin{itemize}
					\item smoothness and convexity between $x_\star$ and $y_{k+1}$ with weight $\lambda_1=d_{k+1}-d_k$
					\[ f_\star \geq f(y_{k+1})+\inner{f'(y_{k+1})}{x_\star-y_{k+1}}+\tfrac1{2L}\normsq{f'(y_{k+1})}, \]
					\item smoothness and convexity between $x_k$ and $y_{k+1}$ with weight $\lambda_2=d_k$
					\[ f(x_k)\geq f(y_{k+1})+\inner{f'(y_{k+1})}{x_k-y_{k+1}}+\tfrac{1}{2L}\normsq{f'(x_k)-f'(y_{k+1})},\]
					\item smoothness and convexity between $y_{k+1}$ and $x_{k+1}$ with weight $\lambda_3=d_{k+1}$
					\[f(y_{k+1})\geq f(x_{k+1})+\inner{f'(x_{k+1})}{y_{k+1}-x_{k+1}}+\tfrac{1}{2L}\normsq{f'(x_{k+1})-f'(y_{k+1})}, \]
					\item first line-search optimality condition for $y_{k+1}$ with weight $\lambda_4=d_{k+1}$
					\[\inner{f'(y_{k+1})}{y_{k+1}-x_k}\leq 0, \]
					\item second line-search optimality condition for $y_{k+1}$ with weight $\lambda_5=d_{k+1}-d_k$
					\[\inner{f'(y_{k+1})}{x_k-z_k}\leq 0, \]
					\item first line-search optimality condition for $x_{k+1}$ with weight $\lambda_6=d_{k+1}$
					\[\inner{f'(x_{k+1})}{x_{k+1}-y_{k+1}}\leq 0,\]
					\item second line-search optimality condition for $x_{k+1}$ with weight $\lambda_7=\tfrac{d_{k+1}}{L}$
					\[\inner{f'(x_{k+1})}{f'(y_{k+1})}\leq 0.\]
					
				\end{itemize}
				In the case $b_k+d_k>0$, the weighted sum gives:
				\begin{equation*}
				\begin{aligned}
				0\geq& \lambda_1 \left[f(y_{k+1})-f_\star+\inner{f'(y_{k+1})}{x_\star-y_{k+1}}+\tfrac1{2L}\normsq{f'(y_{k+1})}\right] \\ 
				&+ \lambda_2 \left[f(y_{k+1})-f(x_k)+\inner{f'(y_{k+1})}{x_k-y_{k+1}}+\tfrac{1}{2L}\normsq{f'(x_k)-f'(y_{k+1})}\right]\\
				&+ \lambda_3 \left[f(x_{k+1})-f(y_{k+1})+\inner{f'(x_{k+1})}{y_{k+1}-x_{k+1}}+\tfrac{1}{2L}\normsq{f'(x_{k+1})-f'(y_{k+1})}\right]\\
				&+ \lambda_4 \left[\inner{f'(y_{k+1})}{y_{k+1}-x_k}\right]\\
				&+ \lambda_5 \left[\inner{f'(y_{k+1})}{x_k-z_k}\right]\\
				&+ \lambda_6 \left[\inner{f'(x_{k+1})}{x_{k+1}-y_{k+1}}\right]\\
				&+ \lambda_7 \left[\inner{f'(x_{k+1})}{f'(y_{k+1})}\right]\\
				=& d_{k+1} (f(x_{k+1})-f_\star)+\tfrac{b_{k+1}}{2 L} \normsq{f'(x_{k+1})}+\tfrac{L}{2}\normsq{z_{k+1}-x_\star}\\&-d_k (f(x_k)-f_\star)-\tfrac{b_k}{2 L}\normsq{f'(x_k)}-\tfrac{L}{2}  \normsq{z_k-x_\star}\\
				 &+ \tfrac{d_{k+1}-b_{k+1}}{2 L}\normsq{f'(x_{k+1})}  + \tfrac{b_k+d_k }{2 L} \normsq{f'(x_k)-\tfrac{d_k}{b_k+d_k} f'(y_{k+1})}\\
				& +\tfrac{-d_{k+1}^2+2 (d_{k}+1) d_{k+1}-\tfrac{d_{k}^2}{b_{k}+d_{k}}-d_{k}^2}{2 L} \normsq{f'(y_{k+1})},
				\end{aligned}
				\end{equation*}
				which can be reformulated as
				\begin{equation*}
				\begin{aligned}
				d_{k+1} (f(x_{k+1})-f_\star)&+\tfrac{b_{k+1}}{2 L} \normsq{f'(x_{k+1})}+\tfrac{L}{2}\normsq{z_{k+1}-x_\star}\\
				\leq&\, d_k (f(x_k)-f_\star)+\tfrac{b_k}{2 L}\normsq{f'(x_k)}+\tfrac{L}{2}  \normsq{z_k-x_\star}\\
				& - \tfrac{d_{k+1}-b_{k+1}}{2 L}\normsq{f'(x_{k+1})}- \tfrac{b_k+d_k }{2 L} \normsq{f'(x_k)-\tfrac{d_k}{b_k+d_k} f'(y_{k+1})}\\
				& - \tfrac{-d_{k+1}^2+2 (d_{k}+1) d_{k+1}-\tfrac{d_{k}^2}{b_{k}+d_{k}}-d_{k}^2}{2 L} \normsq{f'(y_{k+1})}\\
				\leq&\, d_k (f(x_k)-f_\star)+\tfrac{b_k}{2 L}\normsq{f'(x_k)}+\tfrac{L}{2}  \normsq{z_k-x_\star},
				\end{aligned}
				\end{equation*}
				where the last inequality is valid as soon as $b_k+d_k>0$ (i.e., at least $b_k>0$ or $d_k>0$ hold), $d_{k+1}\geq b_{k+1}$ and $-d_{k+1}^2+2 (d_{k}+1) d_{k+1}-\tfrac{d_{k}^2}{b_{k}+d_{k}}-d_{k}^2\geq0$, that is, when $d_{k+1}$ lies in the interval
				\[ \left[\tfrac{b_{k} d_{k}+b_{k}+d_{k}^2+d_{k}-\sqrt{2 b_{k}^2 d_{k}+b_{k}^2+3 b_{k} d_{k}^2+2 b_{k} d_{k}+d_{k}^3+d_{k}^2}}{b_{k}+d_{k}},\, \tfrac{b_{k} d_{k}+b_{k}+d_{k}^2+d_{k}+\sqrt{2 b_{k}^2 d_{k}+b_{k}^2+3 b_{k} d_{k}^2+2 b_{k} d_{k}+d_{k}^3+d_{k}^2}}{b_{k}+d_{k}}\right]. \]				
				For the case $b_k=0$ and $d_k\geq 0$ the weighted sum can be written as
				\begin{equation*}
				\begin{aligned}
				0\geq& d_{k+1} (f(x_{k+1})-f_\star)+\tfrac{b_{k+1}}{2 L} \normsq{f'(x_{k+1})}+\tfrac{L}{2}\normsq{z_{k+1}-x_\star}-d_k (f(x_k)-f_\star)-\tfrac{L}{2}  \normsq{z_k-x_\star}\\
					&+ \tfrac{d_{k+1}-b_{k+1}}{2 L}\normsq{f'(x_{k+1})}+\tfrac{d_k}{2 L}\normsq{f'(x_k)-f'(y_{k+1})}\\ & + \tfrac{-d_{k}^2+2 d_{k} d_{k+1}-d_{k}-(d_{k+1}-2) d_{k+1}}{2 L} \normsq{f'(y_{k+1})},
				\end{aligned}
				\end{equation*}
				and the same simplifications can be done again, as soon as $d_{k+1}\in[1+d_k-\sqrt{1+d_k},\ 1+d_k+\sqrt{1+d_k} ]$:
				\[ d_{k+1} (f(x_{k+1})-f_\star)+\tfrac{b_{k+1}}{2 L} \normsq{f'(x_{k+1})}+\tfrac{L}{2}\normsq{z_{k+1}-x_\star}\leq d_k (f(x_k)-f_\star)+\tfrac{L}{2}  \normsq{z_k-x_\star}.\]
				Finally, note that the same proofs are valid for all methods of the form~\eqref{eq:fgm_3seqs} satisfying
				\begin{equation*}
				\begin{aligned}
				&\inner{f'(y_{k+1})}{y_{k+1}-x_k+\tfrac{d_{k+1}-d_k}{d_{k+1}}(x_k-z_k)}\leq 0,\\
				&\inner{f'(x_{k+1})}{x_{k+1}-y_{k+1}+\tfrac1L f'(y_{k+1})}\leq 0,				
				\end{aligned}
				\end{equation*}
				which is true in particular when $\alpha_k=\tfrac1L$, and $\tau_k=\tfrac{d_{k+1}-d_k}{d_{k+1}}$.
			\end{proof}

			\paragraph{Yet another accelerated method.} For the sake of illustration, we want to point out again that some of the choices we made for the design of the previous method were quite arbitrary, and that it was actually not the only way of ending up with an accelerated first-order method. For example, one could start from a method involving only two sequences
			\begin{equation}\label{eq:fgm_2seqs}
			\begin{aligned}
			y_{k+1}&= (1-\tau_k) y_k+\tau_k z_k-\alpha_k f'(y_{k}),\\
			z_{k+1}&=(1-\delta_k) y_{k+1}+\delta_k z_k -\gamma_k f'(y_{k})-\gamma_k' f'(y_{k+1}),\\
			\end{aligned}
			\end{equation}
			along with the corresponding version involving a span-search
			\begin{equation}\label{eq:fgm_2seqs_LS}
			\begin{aligned}
			y_{k+1}&= \argm{x} {f(x)\, \st\, x\in y_k+\sspan\{(z_k-y_k),\, f'(y_k)\}} ,\\
			z_{k+1}&=  (1-\delta_k)y_{k+1}+\delta_k z_k-\gamma_k f'(y_k)-\gamma_k' f'(y_{k+1}).
			\end{aligned}
			\end{equation} For the potential we choose
			\begin{equation*}
			\begin{aligned}
			\phi_k^f= &\begin{pmatrix}\edit{y_k}-x_\star\\ f'(\edit{y_k}) \end{pmatrix}^\top \left[Q_k\otimes I_d\right]\begin{pmatrix}\edit{y_k}-x_\star\\ f'(\edit{y_k}) \end{pmatrix}+a_k \normsq{z_k-x_\star} + d_k\,
			(f(\edit{y_k})-f_\star),
			\end{aligned}	
			\end{equation*}
			with $Q_k\in\Sb^2$ and start with the choice $\phi_0^f=\tfrac{L}2\normsq{{\color{red}z_0}-x_\star}$, $\phi_N^f=d_N\, (f({\color{red}y_N})-f_\star)$. 
			Let us denote $\tilde{\V}_k$ the set of pairs $(\phi^f_{k+1},\phi_k^f)$ for which we can verify the inequality $\phi^f_{k+1}\leq \phi_k^f$ holds for algorithm~\eqref{eq:fgm_2seqs_LS}. As before, we can now design at the same time a potential and a sequence of parameters for~\eqref{eq:fgm_2seqs}, for example by solving
			\begin{equation}\label{eq:AGM_lyap_design2}
			\edit{\max_{\{(\delta_k,\gamma_k,\gamma_k')\}_k}}\max_{\phi_1^f,\hdots,\phi_{N-1}^f,d_N} d_N\, \st\, (\phi_0^f,\phi_1^f)\in \tilde{\V}_0,\hdots,(\phi_{N-1}^f,\phi_N^f)\in\tilde{\V}_{N-1}.
			\end{equation}
			As an example, we provide the results obtained by solving~\eqref{eq:AGM_lyap_design2} for $N=100$ on Figure \ref{fig:design2}. As for the previous cases, it presents a few possible simplifications, leading to the following theorem.

			\begin{center}
				\begin{figure}[!ht]
					\begin{tabular}{rrr}
						\begin{tikzpicture}
						\begin{axis}[plotOptions, title={$Q_k[1,1]$}, ymin=-2, ymax=2,xmin=0,xmax=102,width=.32\linewidth]
						\addplot [colorP5] table [x=k,y=ak] {Data/FGM2_raw.dat};
						\addplot [colorP2, dashed] table [x=k,y=ak] {Data/FGM2_pure.dat};
						\addplot [colorP1, dotted] table [x=k,y=ak] {Data/FGM2_inter.dat};
						\end{axis}
						\end{tikzpicture} &\begin{tikzpicture}
						\begin{axis}[plotOptions, title={$Q_k[1,2]$}, ymin=-2, ymax=30,xmin=0,xmax=102,width=.32\linewidth]
						\addplot [colorP5] table [x=k,y=ck] {Data/FGM2_raw.dat};
						\addplot [colorP2, dashed] table [x=k,y=ck] {Data/FGM2_pure.dat};
						\addplot [colorP1, dotted] table [x=k,y=ck] {Data/FGM2_inter.dat};
						\end{axis}
						\end{tikzpicture} &
						\begin{tikzpicture}
						\begin{axis}[plotOptions, title={$Q_k[2,2]$}, ymin=-3000, ymax=2,xmin=0,xmax=102,width=.32\linewidth]
						\addplot [colorP5] table [x=k,y=bk] {Data/FGM2_raw.dat};
						\addplot [colorP2, dashed] table [x=k,y=bk]{Data/FGM2_pure.dat};
						\addplot [colorP1, dotted] table [x=k,y=bk]{Data/FGM2_inter.dat};
						\end{axis}
						\end{tikzpicture} \\
						\hspace{-.5cm}
						\begin{tikzpicture}
						\begin{axis}[plotOptions, title={$a_k$}, ymin=-2, ymax=2,xmin=0,xmax=102,width=.32\linewidth]
						\addplot [colorP5] table [x=k,y=apk] {Data/FGM2_raw.dat};
						\addplot [colorP2, dashed] table [x=k,y=apk] {Data/FGM2_pure.dat};
						\addplot [colorP1, dotted] table [x=k,y=apk] {Data/FGM2_inter.dat};
						\end{axis}
						\end{tikzpicture}&
						\begin{tikzpicture}
						\begin{axis}[plotOptions, title={$d_k$}, ymin=-2, ymax=6000,xmin=0,xmax=102,width=.32\linewidth]
						\addplot [colorP5] table [x=k,y=dk] {Data/FGM2_raw.dat};
						\addplot [colorP2, dashed] table [x=k,y=dk] {Data/FGM2_pure.dat};
						\addplot [colorP1, dotted] table [x=k,y=dk] {Data/FGM2_inter.dat};
						\end{axis}
						\end{tikzpicture} &
						\begin{tikzpicture}
						\begin{axis}[plotOptions, title={$\tau_k$}, ymin=-2, ymax=2,xmin=0,xmax=102,width=.32\linewidth]
						\addplot [colorP5] table [x=k,y=tauk, skip coords between index={100}{101}] {Data/FGM2_raw.dat};
						\addplot [colorP2, dashed] table [x=k,y=tauk, skip coords between index={100}{101}] {Data/FGM2_pure.dat};
						\addplot [colorP1, dotted] table [x=k,y=tauk, skip coords between index={100}{101}] {Data/FGM2_inter.dat};
						\end{axis}
						\end{tikzpicture}\\
						\begin{tikzpicture}
						\begin{axis}[plotOptions, title={$\alpha_k$}, ymin=-2, ymax=2,xmin=0,xmax=102,width=.32\linewidth]
						\addplot [colorP5] table [x=k,y=alphak, skip coords between index={100}{101}] {Data/FGM2_raw.dat};
						\addplot [colorP2, dashed] table [x=k,y=alphak, skip coords between index={100}{101}] {Data/FGM2_pure.dat};
						\addplot [colorP1, dotted] table [x=k,y=alphak, skip coords between index={100}{101}] {Data/FGM2_inter.dat};
						\end{axis}
						\end{tikzpicture}&
						\begin{tikzpicture}
						\begin{axis}[plotOptions, title={$\gamma_k'$}, ymin=-2, ymax=200,xmin=0,xmax=102,width=.32\linewidth]
						\addplot [colorP5] table [x=k,y=gammakp, skip coords between index={100}{101}] {Data/FGM2_raw.dat};
						\addplot [colorP2, dashed] table [x=k,y=gammakp, skip coords between index={100}{101}] {Data/FGM2_pure.dat};
						\addplot [colorP1, dotted] table [x=k,y=gammakp, skip coords between index={100}{101}] {Data/FGM2_inter.dat};
						\end{axis}
						\end{tikzpicture}&
						\begin{tikzpicture}
						\begin{axis}[plotOptions, title={$\delta_k$}, ymin=0, ymax=3,xmin=0,xmax=102,width=.32\linewidth]
						\addplot [colorP5] table [x=k,y=deltak, skip coords between index={100}{101}] {Data/FGM2_raw.dat};
						\addplot [colorP2, dashed] table [x=k,y=deltak,skip coords between index={100}{101}] {Data/FGM2_pure.dat};
						\addplot [colorP1, dotted] table [x=k,y=deltak,skip coords between index={100}{101}] {Data/FGM2_inter.dat};
						\end{axis}
						\end{tikzpicture}
					\end{tabular}\vspace{-.35cm}
					\caption{Numerical solution to~\eqref{eq:AGM_lyap_design2} for $N=100$ and $L=1$ (plain brown, large values for $\delta_{100}$ and $\gamma_{100}$ were capped for readability purposes; they are due to the fact we impose no control on $z_{N}$ with our initial choice $\phi_{N}^f$), forced $a_k=\tfrac{L}2$ and $Q_k=0$ (dashed red), and forced $a_k=\tfrac{L}2$, $Q_k[1,1]=Q_k[1,2]=0$ and $Q_k[2,2]=-\tfrac{d_k}{2L}$ (dotted blue). For convenience, we did not plot $\gamma_k$ which numerically appeared to be negligible compared to other variables (about $10^{-7}$). \edit{Total time: $\sim30$ sec{.} on single core of Intel Core i$7$ $1.8$GHz CPU.}}\label{fig:design2}
				\end{figure}
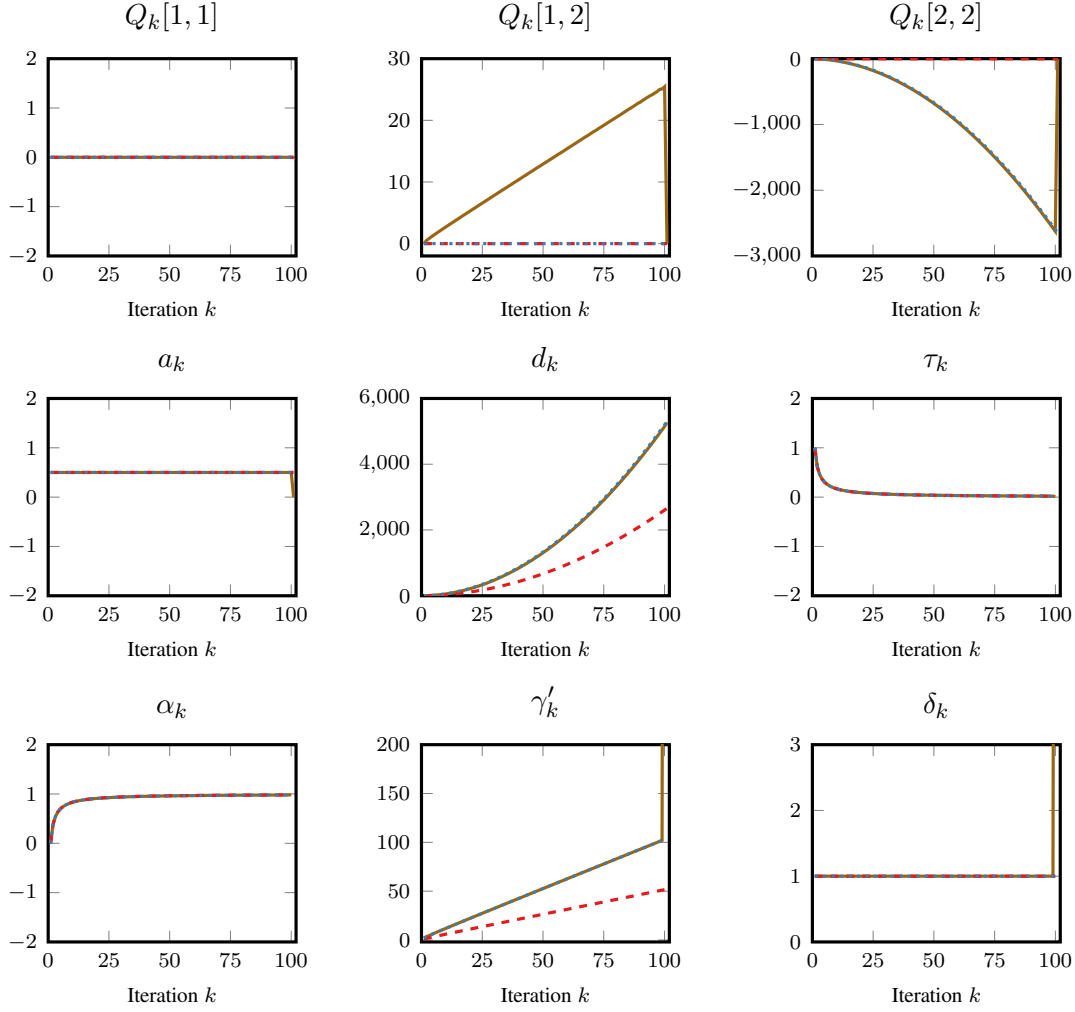
			\end{center}

			The following theorem presents one of the possible outcome of the approach (simple choices are presented just after the theorem).
			\begin{theorem}\label{thm:fgm2_smooth} Let $f\in\FL$. For all values of $d_k\geq b_k \geq0$, the iterates of algorithm~\eqref{eq:fgm_2seqs} with $\delta_k=1$, $\gamma_k=0$, and $\gamma_k'=\tfrac{d_{k+1}-d_k}{L}$ satisfy
			\begin{equation*}
			\begin{aligned}
			d_{k+1} (f(y_{k+1})-f_\star)-\tfrac{b_{k+1}}{2L}\normsq{f'(y_{k+1})}&+\tfrac{L}{2}\normsq{z_{k+1}-x_\star} \\ &\leq d_k (f(y_{k})-f_\star)-\tfrac{b_k}{2L}\normsq{f'(y_{k})}+\tfrac{L}{2}\normsq{z_k-x_\star}
			\end{aligned}
			\end{equation*}
			for all $d_{k+1},b_{k+1}\geq0$ satisfying $b_{k+1}+d_{k+1}-(d_k-d_{k+1})^2\geq 0$. In addition, the iterates of algorithm~\eqref{eq:fgm_2seqs_LS} with $\alpha_k=\tfrac{d_{k}}{d_{k+1} L}$ and $\tau_k=\tfrac{d_{k+1}-d_k}{d_{k+1}}$ satisfy the same inequality.
			\end{theorem}
			Before proceeding with the proof, let us consider those two simpler cases:
			\begin{itemize}
				\item the choice $d_0=0$ along with $b_k=0$ and $d_{k+1}=\tfrac12+d_k+\sqrt{\tfrac14+d_k}=\bO(k^2)$, reaching acceleration (analytical version numerically matching the red curves on Figure \ref{fig:design}).
				\item The choice $b_0=d_0=0$ along with $b_{k+1}=d_{k+1}=1+d_k + \sqrt{1+2 d_k}=\bO(k^2)$, reaching acceleration (analytical version numerically matching the blue curves on Figure \ref{fig:design}), for example for the point $x_k=y_{k}-\tfrac1L f'(y_{k})$ for which we get $f(x_k)-f_\star\leq f(y_{k})-f_\star -\tfrac1{2L}\normsq{f'(y_{k})}\leq \tfrac{L}{2 d_k}\normsq{z_0-x_\star}$.
			\end{itemize}
			\begin{proof}
				Combine the following inequalities with corresponding weights:
				\begin{itemize}
					\item smoothness and convexity between $x_\star$ and $y_{k+1}$ with weight $\lambda_1=d_{k+1}-d_k$:
					\[f_\star\geq f(y_{k+1})+\inner{f'(y_{k+1})}{x_\star-y_{k+1}}+\tfrac{1}{2L}\normsq{f'(y_{k+1})}, \]
					\item smoothness and convexity between $y_k$ and $y_{k+1}$ with weight $\lambda_2=d_k$:
					\[f(y_k)\geq f(y_{k+1})+\inner{f'(y_{k+1})}{y_k-y_{k+1}}+\tfrac{1}{2L}\normsq{f'(y_k)-f'(y_{k+1})}, \]
					\item first line-search optimality condition for $y_{k+1}$ with weight $\lambda_3=d_{k+1}$
					\[\inner{f'(y_{k+1})}{y_{k+1}-y_k}\leq 0, \]
					\item second line-search optimality condition for $y_{k+1}$ with weight $\lambda_4=d_{k+1}-d_k$
					\[\inner{f'(y_{k+1})}{y_{k}-z_k}\leq 0, \]
					\item third line-search optimality condition for $y_{k+1}$ with weight $\lambda_5=\tfrac{d_{k}}{L}$
					\[\inner{f'(y_{k+1})}{f(y_k)}\leq 0.\]
				\end{itemize}
			The weighted sum can be rewritten as
			\begin{equation*}
			\begin{aligned}
			0\geq& \lambda_1 \left[ f(y_{k+1})-f_\star+\inner{f'(y_{k+1})}{x_\star-y_{k+1}}+\tfrac{1}{2L}\normsq{f'(y_{k+1})} \right]\\
			&+ \lambda_2 \left[ f(y_{k+1})-f(y_k)+\inner{f'(y_{k+1})}{y_k-y_{k+1}}+\tfrac{1}{2L}\normsq{f'(y_k)-f'(y_{k+1})} \right]\\
			&+ \lambda_3 \left[ \inner{f'(y_{k+1})}{y_{k+1}-y_k} \right]\\
			&+ \lambda_4 \left[ \inner{f'(y_{k+1})}{y_{k}-z_k} \right]\\
			&+ \lambda_5 \left[ \inner{f'(y_{k+1})}{f(y_k)} \right]\\
			=&d_{k+1} (f(y_{k+1})-f_\star) -\tfrac{b_{k+1}}{2L}\normsq{f'(y_{k+1})}+\tfrac{L}2 \normsq{z_{k+1}-x_\star}\\
			&- d_k (f(y_k)-f_\star) +\tfrac{b_k}{2L}\normsq{f'(y_k)}-\tfrac{L}2\normsq{z_k-x_\star}\\
			&+ \tfrac{ (d_{k}-b_{k})}{2 L} \normsq{f'(y_k)}+ \tfrac{ \left(b_{k+1}-(d_{k}-d_{k+1})^2+d_{k+1}\right)}{2 L} \normsq{f'(y_{k+1})},
			\end{aligned}
			\end{equation*}
			which, in turn, gives:
			\begin{equation*}
			\begin{aligned}
			d_{k+1} &(f(y_{k+1})-f_\star) -\tfrac{b_{k+1}}{2L}\normsq{f'(y_{k+1})}+\tfrac{L}2 \normsq{z_{k+1}-x_\star}\\\leq  &
			 d_k (f(y_k)-f_\star) -\tfrac{b_k}{2L}\normsq{f'(y_k)}+\tfrac{L}2\normsq{z_k-x_\star}\\
			&- \tfrac{ (d_{k}-b_{k})}{2 L} \normsq{f'(y_k)}- \tfrac{ b_{k+1}-(d_{k}-d_{k+1})^2+d_{k+1}}{2 L} \normsq{f'(y_{k+1})}\\ \leq  &
			d_k (f(y_k)-f_\star) -\tfrac{b_k}{2L}\normsq{f'(y_k)}+\tfrac{L}2\normsq{z_k-x_\star},
			\end{aligned}
			\end{equation*}
			where the last inequality is valid as soon as $d_k\geq b_k$, and $b_{k+1}-(d_{k}-d_{k+1})^2+d_{k+1}\geq 0$. The same bound is valid for all methods satisfying
			\[\inner{f'(y_{k+1})}{d_{k+1} (y_{k+1}-y_k)+(d_{k+1}-d_k)(y_k-z_k)+\tfrac{d_k}{L}f'(y_k)}\leq 0, \]
			that include algorithm~\eqref{eq:fgm_2seqs} when $\alpha_k=\tfrac{d_{k}}{d_{k+1} L}$ and $\tau_k=\tfrac{d_{k+1}-d_k}{d_{k+1}}$.
			\end{proof}
			
				\subsection{Linear matrix inequalities for the gradient method}\label{sec:Vk_pgm}
				In this section, we provide the LMI that was used in \secref{sec:GM}, i.e., for vanilla gradient descent. One can extend this LMI to the projected/proximal case using the tools by~\citet{taylor2017exact}, resulting in proofs such as in previous section. The code implementing the LMI formulation below is provided in \secref{sec:ccl}.
				
				 Let us recall that the target is to reformulate~\eqref{eq:tt}. In this case, it corresponds to
				\begin{equation*}\begin{aligned}
				0\geq \max_{f,x_k,x_{k+1},x_\star} \ & \phi^f_{k+1}(x_{k+1})-\phi^f_{k}(x_k)\\
				\text{s.t. } &x_{k+1}=x_k-\tfrac1L f'(x_k),\\
				& f\in\FL \text{ and } f'(x_\star)=0,
				\end{aligned}
				\end{equation*}
				\correc{(note again that maximization over the dimension of the problem is implicit here, by considering $f\in\FL$---see notations in \secref{s:prelim}---equivalently, one could formulate the problem by maximizing over $d\in\mathbb{N}$ and requiring $f\in\FL(\Rd)$)} for when
				\begin{align} \label{eq:pot_GM}
				\phi_k^f=\begin{pmatrix}x_k-x_\star\\ f'(x_k)\end{pmatrix}^\top \left[\begin{pmatrix}
				a_k & c_k \\ c_k & b_k
				\end{pmatrix}\otimes I_d\right]\begin{pmatrix}x_k-x_\star\\ f'(x_k)\end{pmatrix} + d_k\, (f(x_k)-f_\star).
				\end{align}
				They first key step in the reformulation is standard from the performance estimation literature and consists in using a discrete version of the variable $f$, as follows
				\begin{equation}\label{eq:pep_gm}\begin{aligned}
				0\geq \max_{\substack{f_k,f_{k+1},f_\star,\\g_k,g_{k+1},g_\star,\\x_k,x_{k+1},x_\star}} \ & \phi^f_{k+1}(x_{k+1})-\phi^f_{k}(x_k),\\
				\text{s.t. } &x_{k+1}=x_k-\tfrac1L g_k,\\
				& \exists f\in\FL \text{ such that } f(x_i)=f_i \text{ and } f'(x_i)=g_i \text{ for all } i\in\{k,k+1,\star\},\\
				& g_\star=0.
				\end{aligned}
				\end{equation}
				The existence constraint is often referred to as an \emph{interpolation constraint} and can be reformulated using appropriate quadratic inequalities~\citep[Theorem 4]{taylor2017smooth} --- this is also sometimes referred to as \emph{extensions}~\citep{azagra2017extension,daniilidis2018explicit}. For smooth convex functions, this can be reformulated as
				\begin{equation*}\begin{aligned}
				\exists f\in\FL & \text{ such that } f(x_i)=f_i \text{ and } f'(x_i)=g_i \text{ for all } i\in\{k,k+1,\star\} \\&\Leftrightarrow
				f_i\geq f_j+\inner{g_j}{x_i-x_j}+\tfrac1{2L}\normsq{g_i-g_j} \text{ for all } i,j\in\{k,k+1,\star\}.
				\end{aligned} 
				\end{equation*}
				This allows reformulating~\eqref{eq:pep_gm} as a linear matrix inequality. For doing that, let us choose without loss of generality $x_\star=g_\star=0$ and $f_\star=0$ and introduce two matrices $P$ and $F$:
				\begin{equation*}
				P= [ \ x_k \quad g_k \quad g_{k+1}\ ], \quad F=[\ f_k \quad f_{k+1} \ ],
				\end{equation*}
				and the corresponding Gram matrix $G=P^{\top\!}P\succeq 0$, that is,
				\[G=\begin{pmatrix}
				\normsq{x_k} & \inner{g_k}{x_k} & \inner{g_{k+1}}{x_k} \\
				\inner{g_k}{x_k} & \normsq{g_k} & \inner{g_{k}}{g_{k+1}}\\
				\inner{g_{k+1}}{x_k} & \inner{g_k}{g_{k+1}} & \normsq{g_{k+1}}
				\end{pmatrix}.\]
				Let us introduce the following shorthand notations for picking up elements in the Gram matrix $G$ and in $F$: we choose $\bx_k,\bx_{k+1},\bx_\star$, $\bg_{k},\bg_{k+1},\bg_\star \in\R^3$ and $\bfu_k,\bfu_{k+1},\bfu_\star \in\R^2$ such that
				\begin{equation*}
				\begin{aligned}
				x_k=P\bx_k,\quad x_{k+1}=P\bx_{k+1},\quad x_\star=P\bx_\star,\\ g_k=P\bg_k, \quad g_{k+1}=P\bg_{k+1},\quad g_\star=P\bg_\star, \\
				f_k=F\bfu_k, \quad f_{k+1}=F\bfu_{k+1},\quad f_\star=F\bfu_\star.
				\end{aligned}
				\end{equation*}
				In other words, by letting $e_i$ be the unit vector with $1$ as its i$^{\text{th}}$ component we have $\bx_k:=e_1\in\R^3$, $\bg_k:=e_2\in\R^3$, $\bg_{k+1}:=e_3\in\R^3$ along with
				\[\bx_{k+1}:=\bx_k-\tfrac1L \bg_k\in\R^3,\quad \bx_\star=\bg_\star:=0\in\R^3, \]
				and $\bfu_k:=e_1\in\R^2$, $\bfu_{k+1}:=e_2\in\R^2$ and $\bfu_\star:=0\in\R^2$. Those notations allow conveniently writing scalar products by picking up  elements in $G$. For example, $\inner{g_{k+1}}{x_{k+1}-x_\star}$ can be written as
				\[\inner{g_{k+1}}{x_{k+1}-x_\star}=(P\bg_{k+1})^{\top\!} P(\bx_{k+1}-\bx_\star)=\tra\left(G \bg_{k+1}(\bx_{k+1}-\bx_\star)^\top \right). \]
				Also, one can equivalently rewrite the inequality
				\[ f_i- f_j-\inner{g_j}{x_i-x_j}-\tfrac1{2L}\normsq{g_i-g_j} \geq 0,\]
				as
				\[  F (\bfu_i - \bfu_j) +  \tra\left(G\begin{pmatrix}
				\bx_i & \bx_j & \bg_i & \bg_j
				\end{pmatrix} M \begin{pmatrix}
				\bx_i & \bx_j & \bg_i & \bg_j
				\end{pmatrix}^\top\right)\geq 0,\]
				where
				\begin{equation}\label{eq:MM}
				\begin{aligned}
				M := \tfrac1 2\begin{pmatrix}
				0  & 0 & 0  & -1 \\
				0  & 0 & 0  & 1 \\
				0  & 0 & -1/L & 1/L \\
				-1 & 1 & 1/L & -1/L \end{pmatrix}.
				\end{aligned}\end{equation}
				We can also rewrite
				\begin{equation*}
				\begin{aligned}
				\phi_k^f(x_k)&=d_k F(\bfu_k-\bfu_\star)+\tra\left(G \begin{pmatrix}
				\bx_k-\bx_\star & \bg_k 
				\end{pmatrix}\begin{pmatrix}
				a_k & c_k \\ c_k & b_k
				\end{pmatrix}\begin{pmatrix}
				\bx_k^\top-\bx_\star^\top \\ \bg_k^\top 
				\end{pmatrix}\right),\\
				\phi_{k+1}^f(x_{k+1})&=d_{k+1} F(\bfu_{k+1}-\bfu_\star) \\&\quad +\tra\left(G \begin{pmatrix}
				\bx_{k+1}-\bx_\star & \bg_{k+1} 
				\end{pmatrix}\begin{pmatrix}
				a_{k+1} & c_{k+1} \\ c_{k+1} & b_{k+1}
				\end{pmatrix}\begin{pmatrix}
				\bx_{k+1}^\top-\bx_\star^\top \\ \bg_{k+1}^\top 
				\end{pmatrix}\right),
				\end{aligned}
				\end{equation*}
				leading to the following reformulation of~\eqref{eq:pep_gm} given by \correc{(note the absence of a rank constraint, due to the fact maximize over $f$ irrespective of the dimension of the ambient space; this is further discussed in~\citet[Section 3.3]{taylor2017smooth})}
				\begin{equation*}\begin{aligned}
				0\geq \max_{G,\ F} \ & \phi^f_{k+1}(x_{k+1})-\phi^f_{k}(x_k),\\
				\text{s.t. } &G\succeq 0,\\
				&F (\bfu_i - \bfu_j) +  \tra\left(G\begin{pmatrix}
				\bx_i & \bx_j & \bg_i & \bg_j
				\end{pmatrix} M \begin{pmatrix}
				\bx_i & \bx_j & \bg_i & \bg_j
				\end{pmatrix}^\top\right)\geq 0\\ &\quad \quad \text{ for all $i,j\in\{k,k+1,\star \}$},
				\end{aligned}
				\end{equation*}
				whose feasibility (for given parameters $\{(a_k,b_k,c_k,d_k),(a_{k+1},b_{k+1},c_{k+1},d_{k+1})\}$ and $L$) can be verified with standard semidefinite packages~\citep{Article:Sedumi,Article:Mosek}. Let us provide the corresponding linear matrix inequality, which is simply obtained by dualizing the previous formulation. By associating one multiplier for each \emph{interpolation constraint}
				\begin{equation*}
				\begin{aligned}
				F (\bfu_i - \bfu_j) +  \tra\left(G\begin{pmatrix}
				\bx_i & \bx_j & \bg_i & \bg_j
				\end{pmatrix} M \begin{pmatrix}
				\bx_i & \bx_j & \bg_i & \bg_j
				\end{pmatrix}^\top\right)\geq 0 && :\lambda_{i,j},
				\end{aligned}
				\end{equation*}
				one can arrive to the final LMI for gradient descent (we use the notation $\Sb^3$ to denote the set of $3\times 3$ symmetric matrices). \correc{The equivalence with the primal problem can be formally established through strong duality; see for example~\citet[Theorem 6]{taylor2017smooth}.}
				\begin{oframed}
					\noindent The inequality $\phi^f_{k+1}(x_{k+1})\leq \phi^f_{k}(x_k)$ ($\phi_k^f$ defined in~\eqref{eq:pot_GM}) holds for all $x_{k}\in\Rd$, all $f\in\FL(\Rd)$, all $d\in\N$ and $x_{k+1}=x_k-\tfrac1L f'(x_k)$ if and only if there exists $\{\lambda_{i,j}\}_{i,j\in I_k}$, with $I_k:=\{k,k+1,\star\}$, such that 	\begin{equation*}
					\begin{aligned}
					& \lambda_{i,j} \geq 0 \text{ for all } i,j\in I_k,\\
					& d_{k+1} (\bfu_{k+1}-\bfu_\star)-d_k (\bfu_k-\bfu_\star) + \sum_{i,j\in I_k} \lambda_{i,j} (\bfu_{i}-\bfu_j) = 0 \text{ (linear constraint in $\R^2$)},\\
					& V_{k+1} - V_k + \sum_{i,j\in I_k} \lambda_{i,j} M_{i,j} \preceq 0 \text{ (linear matrix inequality in $\Sb^3$)},
					\end{aligned}
					\end{equation*}
					with 
					\begin{equation*}
					\begin{aligned}
					V_{k}&:=\begin{pmatrix}
					\bx_{k}-\bx_\star & \bg_{k} 
					\end{pmatrix}\begin{pmatrix}
					a_{k} & c_{k} \\ c_{k} & b_{k}
					\end{pmatrix}\begin{pmatrix}
					\bx_{k}^\top-\bx_\star^\top \\ \bg_{k}^\top 
					\end{pmatrix}\in\Sb^{3} ,\\
					V_{k+1}&:=\begin{pmatrix}
					\bx_{k+1}-\bx_\star & \bg_{k+1} 
					\end{pmatrix}\begin{pmatrix}
					a_{k+1} & c_{k+1} \\ c_{k+1} & b_{k+1}
					\end{pmatrix}\begin{pmatrix}
					\bx_{k+1}^\top-\bx_\star^\top \\ \bg_{k+1}^\top 
					\end{pmatrix}\in\Sb^{3},\\
					M_{i,j}&:=\begin{pmatrix}
					\bx_i & \bx_j & \bg_i & \bg_j
					\end{pmatrix} M \begin{pmatrix}
					\bx_i & \bx_j & \bg_i & \bg_j
					\end{pmatrix}^\top\in\Sb^{3}.
					\end{aligned}.
					\end{equation*}
				\end{oframed}
			
			\subsection{Linear matrix inequalities for the design procedure}\label{sec:LMI_design}
			
			We provide the LMI formulation of $\tilde{\V}_k$ for algorithm~\eqref{eq:fgm_3seqs}. The code implementing the LMI formulation below is provided in \secref{sec:ccl}.

			 We proceed mostly like in \secref{sec:Vk_pgm} for reformulating~\eqref{eq:tt}.  That is, we now reformulate
			\begin{equation}\label{eq:pep_fgm}\begin{aligned}
			0\geq \max_{\substack{f_{y_k},f_{x_{k}},f_{x_{k+1}},f_\star,\\g_{y_k},g_{x_{k}},g_{x_{k+1}},g_\star,\\z_k,y_k,x_k,x_{k+1},x_\star}} \ & \phi^f_{k+1}(x_{k+1},z_{k+1})-\phi^f_{k}(x_k,z_k),\\
			\text{s.t. } &\inner{g_{y_{k+1}}}{y_{k+1}-x_k}=0, \quad \inner{g_{y_{k+1}}}{z_k-x_k}=0,\\ &\inner{g_{x_{k+1}}}{y_{k+1}-x_{k+1}}=0,\quad \inner{g_{x_{k+1}}}{g_{y_{k+1}}}=0,\\
			&			z_{k+1}=(1-\delta_k) y_{k+1}+\delta_k z_k -\gamma_k f'(y_{k+1}),\\
			& \exists f\in\FL \text{ such that } f(x_i)=f_i \text{ and } f'(x_i)=g_i \text{ for all } (x_i,g_i,f_i)\in S,\\
			& g_\star=0,
			\end{aligned}
			\end{equation}
			where $S$ is the discrete version of $f$: \[S:=\{(x_k,g_{x_k},f_{x_k}),(x_{k+1},g_{x_{k+1}},f_{x_{k+1}}),(y_{k+1},g_{y_{k+1}},f_{y_{k+1}}),(x_\star,g_\star,f_\star)\}.\]
			Note again that~\eqref{eq:pep_fgm} is not an exact reformulation of~\eqref{eq:tt} for the method~\eqref{eq:fgm_3seqs_LS}: it is only a sufficient condition for~\eqref{eq:pot} to be satisfied for the method~\eqref{eq:fgm_3seqs_LS} and all $d\in\N$, $f\in\FL$ and $(x_k,z_k)\in\Rd\times\Rd$.
			
			\paragraph{SDP reformulation.}
			As in \secref{sec:Vk_pgm}, we use a vector $F$ and a Gram matrix $G=P^\top P\succeq 0$ for encoding the problem, with $P$ and $F$ defined as
			\begin{equation*}
			P:= [ \ z_k \quad x_k \quad x_{k+1 } \quad y_{k+1} \quad g_{x_{k}}  \quad g_{x_{k+1}} \quad g_{y_{k+1}}\ ], \quad F:=[\ f_{x_k}   \quad f_{x_{k+1}} \quad f_{y_{k+1}} \ ].
			\end{equation*}
			For denoting all points whose gradient and functions values are needed within the formulation, we re-organize them and denote
			\begin{equation*}
			\begin{aligned}
			&(w_\star,g_\star,f_\star):=(x_\star,g_\star,f_\star),\, &(w_1,g_1,f_1):=(x_k,g_{x_k},f_{x_k}),\\ &(w_2,g_2,f_2):=(x_{k+1},g_{x_{k+1}},f_{x_{k+1}}),\, &(w_3,g_3,f_3):=(y_{k+1},g_{y_{k+1}},f_{y_{k+1}}).
			\end{aligned}
			\end{equation*}			
			Let us denote $\bw_\star$, $\bw_1$, $\bw_2$, $\bw_{3}$, $\bg_\star$, $\bg_1$, $\bg_2$, $\bg_3$, $\bz_k\in\R^7$ and $\bfu_\star,\bfu_1,\bfu_2, \bfu_{3}\in\R^3$ the vectors such that
			\begin{equation*}
			\begin{aligned}
			&x_k=P\bw_1,\quad 
			x_{k+1}=P\bw_2,\quad
			y_{k+1}=P\bw_3,\\
			&g_{x_k}=P\bg_1,\quad
			g_{x_{k+1}}=P\bg_2,\quad
			g_{y_{k+1}}=P\bg_3,\\
			&f_{x_k}=P\bfu_{1},\quad f_{x_{k+1}}=P\bfu_{2}, \quad f_{y_{k+1}}=P\bfu_{3}.
			\end{aligned}
			\end{equation*}
			along with $z_k=P\bz_k$, $\bw_\star=\bg_\star=0$ and $\bfu_\star=0$. More precisely, we define $\bz_k:=e_1\in\R^7$
			\begin{equation*}
			\begin{aligned}
			&\bw_1:=e_2\in\R^7,\quad \bw_2:=e_3\in\R^7,\quad \bw_3:=e_4\in\R^7,\\
			&\bg_1:=e_5\in\R^7,\quad \bg_{2}:=e_6\in\R^7, \quad \bg_{3}:=e_7\in\R^7,\\
			&\bfu_{1}:=e_1\in\R^3, \quad \bfu_{2}:=e_2\in\R^3,\quad \bfu_{3}:=e_3\in\R^3.
			\end{aligned}
			\end{equation*}
			We encode \emph{interpolation constraints} and the corresponding \emph{multiplier} $\lambda_{i,j}$ as before, for all points for which we needed to use the gradient and function values, i.e., for $x_\star$, $x_k$, $x_{k+1}$ and $y_{k+1}$:
			\begin{equation*}
			\begin{aligned}
			F (\bfu_i - \bfu_j) +  \tra\left(G\begin{pmatrix}
			\bw_i & \bw_j & \bg_i & \bg_j
			\end{pmatrix} M \begin{pmatrix}
			\bw_i & \bw_j & \bg_i & \bg_j
			\end{pmatrix}^\top\right)\geq 0 && :\lambda_{i,j},
			\end{aligned}
			\end{equation*}
			where $i,j\in\{\star,1,2,3\}$. Next, the four line-search conditions can be encoded as $\tra (G\ A_i)=0$ (for $i=1,\hdots,4$) with
			\begin{equation*}
			\begin{aligned}
			A_1&=\tfrac12 \left(\bg_{3} (\bw_{3}-\bw_1)^\top+ (\bw_{3}-\bw_1) \bg_{3}^\top \right)\text{, for $\inner{f'(y_{k+1})}{y_{k+1}-x_k}=\tra(A_1G)=0$,}\\
			A_2&=\tfrac12 \left(\bg_{3} (\bz_{k}-\bw_1)^\top+ (\bz_{k}-\bw_1) \bg_{3}^\top \right)\text{, for $\inner{f'(y_{k+1})}{z_{k}-x_k}=\tra(A_2G)=0$,}\\
			A_3&=\tfrac12 \left(\bg_{2}(\bw_{3}-\bw_2)^\top+ (\bw_{3}-\bw_2) \bg_{2}^\top \right)\text{, for $\inner{f'(x_{k+1})}{y_{k+1}-x_{k+1}}=\tra(A_3G)=0$,}\\
			A_4&=\tfrac12 \left(\bg_{2} \bg_{3}^\top + \bg_3 \bg_2^\top \right)\text{, for $\inner{f'(x_{k+1})}{f'(y_{k+1})}=\tra(A_4G)=0$,}\\
			\end{aligned}
			\end{equation*}
			and we denote the corresponding multipliers by $\mu_{1},\hdots,\mu_{4}$. Finally, we define
			\[ \bz_{k+1}:=(1-\delta_k) \bw_{3}+\delta_k\bz_k-\gamma_k \bg_3.\]
			The LMI formulation is, again, a dual to the SDP reformulation of~\eqref{eq:pep_fgm} (see below):
			\begin{equation*}\begin{aligned}
			0\geq \max_{G,\, F} \ & \phi^f_{k+1}(x_{k+1},z_{k+1})-\phi^f_{k}(x_k,z_k),\\
			\text{s.t. } &\tra (G\ A_1)=0,\, \tra (G\ A_2)=0,\,\tra (G\ A_3)=0,\, \tra (G\ A_4)=0,\\
			& F (\bfu_i - \bfu_j) +  \tra\left(G\begin{pmatrix}
			\bw_i & \bw_j & \bg_i & \bg_j
			\end{pmatrix} M \begin{pmatrix}
			\bw_i & \bw_j & \bg_i & \bg_j
			\end{pmatrix}^\top\right)\geq 0 \\ &\quad \quad \text{ for all } i,j\in \{\star,1,2,3\}.
			\end{aligned}
			\end{equation*}
			We use $\Sb^7$ to denote the set of $7\times 7$ symmetric matrices, and use $M$ as defined in~\eqref{eq:MM}. \correc{In this case, a strong duality argument can be obtained using similar lines as~\citet[Lemma 6]{drori2018efficient}.}
			\begin{oframed}
				\noindent Given two triplets $(a_k,d_k,Q_k)$, $(a_{k+1},d_{k+1},Q_{k+1})$ and a pair $(\delta_k,\gamma_k)$, if there exists $\{\lambda_{i,j}\}_{i,j\in I_k}$, with $I_k:=\{\star,1,2,3\}$ and $\{\mu_i\}_{i\in\{1,2,3,4\}}$ such that 	\begin{equation*}
				\begin{aligned}
				& \lambda_{i,j} \geq 0 \text{ for all } i,j\in I_k, \\
				& d_{k+1} (\bfu_{k+1}-\bfu_\star)-d_k (\bfu_k-\bfu_\star) + \sum_{i,j\in I_k} \lambda_{i,j} (\bfu_{i}-\bfu_j) = 0 \text{ (linear constraint in $\R^3$)},\\
				&  V_{k+1}-V_k + \sum_{i,j\in I_k} \lambda_{i,j} M_{i,j}+\sum_{i\in\{1,2,3,4\}} \mu_{i}\ A_i \preceq 0 \text{ (linear matrix inequality in $\Sb^7$)},
				\end{aligned}
				\end{equation*}
				with 
				\begin{equation*}
				\begin{aligned}
				V_{k}&:=a_{k}(\bz_k-\bx_\star) (\bz_k-\bx_\star)^\top + \begin{pmatrix}
				\bw_{1}-\bw_\star & \bg_{1} 
				\end{pmatrix}Q_k\begin{pmatrix}
				\bw_{1}^\top-\bw_\star^\top \\ \bg_{1}^\top 
				\end{pmatrix}\in\Sb^{7} ,\\
				V_{k+1}&:=a_{k+1}(\bz_{k+1}-\bx_\star) (\bz_{k+1}-\bx_\star)^\top +\begin{pmatrix}
				\bw_{2}-\bw_\star & \bg_{2} 
				\end{pmatrix}Q_{k+1}\begin{pmatrix}
				\bw_{2}^\top-\bw_\star^\top \\ \bg_{2}^\top 
				\end{pmatrix}\in\Sb^{7},\\
				M_{i,j}&:=\begin{pmatrix}
				\bw_i & \bw_j & \bg_i & \bg_j
				\end{pmatrix} M \begin{pmatrix}
				\bw_i & \bw_j & \bg_i & \bg_j
				\end{pmatrix}^\top\in\Sb^{7},
				\end{aligned}
				\end{equation*}
				then the inequality $\phi^f_{k+1}(x_{k+1})\leq \phi^f_{k}(x_k)$ ($\phi_k^f$ defined in~\eqref{eq:pot_fgm1}) holds for all $f\in\FL(\Rd)$, all $d\in\N$, all $x_k,z_k\in\Rd$, and $x_{k+1},z_{k+1}$ generated by method~\eqref{eq:fgm_3seqs_LS}. 
			\end{oframed}
			There are several ways of optimizing the parameters $(\delta_k,\gamma_k)$ in the process (an alternative to what we propose below is to use an appropriate Schur complement),  for example
			\begin{equation*}
			\begin{aligned}
			&a_{k+1}(\bz_{k+1}-\bx_\star)(\bz_{k+1}-\bx_\star)^\top=a_k\begin{pmatrix}
			\bw_3 \\ \bz_k-\bw_3 \\ \bg_3
			\end{pmatrix}^\top \begin{pmatrix}
			1 \\ \delta_k \\ -\gamma_k
			\end{pmatrix}\begin{pmatrix}
			1 \\ \delta_k \\ -\gamma_k
			\end{pmatrix}^\top
			\begin{pmatrix}
			\bw_3 \\ \bz_k-\bw_3 \\ \bg_3
			\end{pmatrix},
			\\
			& = \begin{pmatrix}
			\bw_3 \\ \bz_k-\bw_3 \\ \bg_3
			\end{pmatrix}^\top \underbrace{\begin{pmatrix}
			a_k & a_k\delta_k & -a_k\gamma_k \\
			a_k\delta_k & a_k\delta_k^2 & -a_k\delta_k\gamma_k\\
			-a_k\gamma_k & -a_k\delta_k\gamma_k & a_k\gamma_k^2
			\end{pmatrix}}_{S_k:=}\begin{pmatrix}
			\bw_3 \\ \bz_k-\bw_3 \\ \bg_3
			\end{pmatrix}= \begin{pmatrix}
			\bw_3 \\ \bz_k-\bw_3 \\ \bg_3
			\end{pmatrix}^\top S_k \begin{pmatrix}
			\bw_3 \\ \bz_k-\bw_3 \\ \bg_3
			\end{pmatrix},
			\end{aligned}
			\end{equation*}
			where $\mathrm{Rank}\, S_k = 1$, $S_k\succeq 0$ and $S_k[1,1]=a_k$. Now, one can pick $S_k\succeq 0$ as new variable, drop the rank constraint, and keep constraining $S_k[1,1]=a_k$. There is always a feasible solution to the LMI with  $\mathrm{Rank}\, S_k= 1$ if the LMI is feasible ($S_k$ intervenes only on the side $.\preceq 0$), leading to
			\[ V_{k+1}=S_k +\begin{pmatrix}
			\bw_{2}-\bw_\star & \bg_{2} 
			\end{pmatrix}Q_k\begin{pmatrix}
			\bw_{2}^\top-\bw_\star^\top \\ \bg_{2}^\top 
			\end{pmatrix}\]
			in the previous LMI, under the additional constraints $S_k\succeq 0$ and $S_k[1,1]=a_k$. One can then recover $\delta_k=\tfrac{S_k[1,2]}{S_k[1,1]}$ and $\gamma_k=-\tfrac{S_k[1,3]}{S_k[1,1]}$, along with $\tau_k=-\tfrac{\mu_2}{\mu_1}$ and $\alpha_k=\tfrac{\mu_4}{\mu_3}$.
			\clearpage 
			\section{Stochastic gradients under bounded variance}\label{sec:boundedvariances}
			The proofs of this section follow the same ideas as that of \appref{sec:GM_app}: we only reformulate weighted sums of inequalities.
			\subsection{Proof of Theorem~\ref{thm:sgd_BV}}\label{sec:sgd_pot}
			\begin{proof}
				Combine the following inequalities with their corresponding weights.
				\begin{itemize}
					\item Convexity between $x_\star$ and $x_k$ with weight $\lambda_1=\delta_k L$
					\[f_\star \geq f(x_k)+\inner{f'(x_k)}{x_\star-x_k}, \]
					\item averaged smoothness between $x_k$ and $x_{k+1}^{{\color{red}({i_k})}}$ with weight $\lambda_2=d_{k+1}=d_k+\delta_k L$
					\[\E_{i_k} f(x_{k+1}^{({i_k})}) \leq f(x_k) + \E_{i_k}\inner{f'(x_k)}{x_{k+1}^{({i_k})}-x_k}+\tfrac{L}{2}\E_{i_k}\normsq{x_{k+1}^{({i_k})}-x_k},\]
					\item bounded variance at $x_k$ with weight $\lambda_3=e_{k}=\tfrac{\delta_k^2 L}{2}(1+d_{k+1})$
					\[\E_{i_k} \normsq{G(x_k;{i_k})-f'(x_k)}\leq \sigma^2.\]
				\end{itemize}
				The weighted sum can be reformulated as
				\begin{equation*}
				\begin{aligned}
				0\geq & \lambda_1\left[f(x_k)-f_\star+\inner{f'(x_k)}{x_\star-x_k}\right] \\ &\correc{-} \lambda_2 \left[ f(x_k)-\E_{i_k} f(x_{k+1}^{({i_k})})  + \E_{i_k}\inner{f'(x_k)}{x_{k+1}^{({i_k})}-x_k}+\tfrac{L}{2}\E_{i_k}\normsq{x_{k+1}^{({i_k})}-x_k}\right]\\ &+\lambda_3\left[ \E_{i_k} \normsq{G(x_k;{i_k})-f'(x_k)}- \sigma^2\right]\\
				=&d_{k+1}\E_{i_k}[f(x_{k+1}^{({i_k})})-f_\star]+\tfrac{L}{2}\E_{i_k}\normsq{x_{k+1}^{({i_k})}-x_\star} \\
				&-d_{k} (f(x_k)-f_\star)-\tfrac{L}{2}\normsq{x_k-x_\star}-e_{k}\sigma^2+ (\delta_k d_{k+1}-e_{k}) \normsq{f'(x_k)},
				\end{aligned}
				\end{equation*}
				where rearrangment of the last inequality leads to
				\begin{equation*}
				\begin{aligned}
				d_{k+1}&\E_{i_k}[f(x_{k+1}^{({i_k})})-f_\star]+\tfrac{L}{2}\E_{i_k}\normsq{x_{k+1}^{({i_k})}-x_\star}\\&\leq d_{k} (f(x_k)-f_\star)+\tfrac{L}{2}\normsq{x_k-x_\star}+e_{k}\sigma^2- (\delta_k d_{k+1}-e_{k}) \normsq{f'(x_k)}\\
				&\leq d_{k} (f(x_k)-f_\star)+\tfrac{L}{2}\normsq{x_k-x_\star}+e_{k}\sigma^2,
				\end{aligned}
				\end{equation*}
				as soon as $\delta_k d_{k+1}-e_{k}\geq 0$.
			\end{proof}
			\subsection{Proof of Theorem~\ref{thm:avg_BV}}\label{sec:avgsgd_pot}
			\begin{proof}
				Combine the following inequalities with their corresponding weights.
				\begin{itemize}
					\item Smoothness and convexity between $x_k$ and $x_\star$ with weight $\lambda_1=\delta_k L$
					\[f_\star\geq f(x_k)+\inner{f'(x_k)}{x_\star-x_k}+\tfrac1{2L}\normsq{f'(x_k)}, \]
					\item averaged smoothness and convexity between $x_{k}$ and $z_{k+1}^{({i_k})}$ with weight $\lambda_2=\delta_k L$
					\[f(x_k)\geq \E_{i_k} f(z^{({i_k})}_{k+1}) +\E_{i_k}\inner{f'(z_{k+1}^{({i_k})})}{x_k-z_{k+1}^{({i_k})}} +\tfrac1{2L}\E_{i_k}\normsq{f'(x_k)-f'(z_{k+1}^{({i_k})})},\]
					\item averaged smoothness and convexity between $z_{k}$ and $z_{k+1}^{({i_k})}$ with weight $\lambda_3=d_k\delta_k L$
					\[f(z_k)\geq \E_{i_k} f(z^{({i_k})}_{k+1}) +\E_{i_k}\inner{f'(z_{k+1}^{({i_k})})}{z_k-z_{k+1}^{({i_k})}} +\tfrac1{2L}\E_{i_k}\normsq{f'(z_k)-f'(z_{k+1}^{({i_k})})},\]
					\item bounded variance at $x_k$ with weight $\lambda_4=e_{k} =\tfrac{\delta_k^2 L}{2}\tfrac{1+d_k+L\delta_k}{1+d_k}$
					\[\E_{i_k} \normsq{G(x_k;{i_k})-f'(x_k)}\leq \sigma^2.\]
				\end{itemize}
				The weighted sum can be reformulated as
				\begin{equation*}
				\begin{aligned}
				0\geq & \lambda_1 \left[ f(x_k)-f_\star+\inner{f'(x_k)}{x_\star-x_k}+\tfrac1{2L}\normsq{f'(x_k)} \right]\\
				&+\lambda_2 \left[\E_{i_k} f(z^{({i_k})}_{k+1})-f(x_k) +\E_{i_k}\inner{f'(z_{k+1}^{({i_k})})}{x_k-z_{k+1}^{({i_k})}} +\tfrac1{2L}\E_{i_k}\normsq{f'(x_k)-f'(z_{k+1}^{({i_k})})}\right]\\
				&+\lambda_3 \left[\E_{i_k} f(z^{({i_k})}_{k+1})-f(z_k) +\E_{i_k}\inner{f'(z_{k+1}^{({i_k})})}{z_k-z_{k+1}^{({i_k})}} +\tfrac1{2L}\E_{i_k}\normsq{f'(z_k)-f'(z_{k+1}^{({i_k})})}\right]\\
				&+\lambda_4 \left[\E_{i_k} \normsq{G(x_k;{i_k})-f'(x_k)}- \sigma^2\right]\\
				=& \delta_k d_{k+1}L\E_{i_k}[f(z^{({i_k})}_{k+1})-f_\star]+\tfrac{L}{2}\E_{i_k}\normsq{x_{k+1}^{({i_k})}-x_\star}-\delta_k d_{k}L (f(z_k)-f_\star)-\tfrac{L}{2}\normsq{x_k-x_\star}-e_k\sigma^2\\
				&+\tfrac{(d_k+1) \delta_k}{2} \E_{i_k} \normsq{(\tfrac{1}{d_k+1}-1) f'(z_k)-\tfrac{1}{d_k+1}f'(x_k)+\tfrac{\delta_k   L}{d_k+1}G(x_k;{i_k})+f'(z_{k+1}^{({i_k})})}\\
				&+\tfrac{d_k \delta_k }{2 (d_k+1)}\normsq{f'(z_k)+ (\delta_k  L-1) f'(x_k)}\\
				&+\tfrac{ \delta_k (1+\delta_k  L (1-\delta_k  L))}{2} \normsq{f'(x_k)}.
				\end{aligned}
				\end{equation*}
				Again, rearranging the terms leads to:		
				\begin{equation*}
				\begin{aligned}
				\delta_k d_{k+1} L&\E_{i_k}[f(z^{({i_k})}_{k+1})-f_\star]+\tfrac{L}{2}\E_{i_k}\normsq{x_{k+1}^{({i_k})}-x_\star}\\ &\leq \delta_k d_{k}L (f(z_k)-f_\star)+\tfrac{L}{2}\normsq{x_k-x_\star}+e_k\sigma^2\\ &\quad -\tfrac{(d_k+1) \delta_k}{2} \E_{i_k} \normsq{(\tfrac{1}{d_k+1}-1) f'(z_k)-\tfrac{1}{d_k+1}f'(x_k)+\tfrac{\delta_k   L}{d_k+1}G(x_k;{i_k})+f'(z_{k+1}^{({i_k})})}\\
				&\quad -\tfrac{d_k \delta_k }{2 (d_k+1)}\normsq{f'(z_k)+ (\delta_k  L-1) f'(x_k)}-\tfrac{ \delta_k (1+\delta_k  L (1-\delta_k  L))}{2} \normsq{f'(x_k)}\\
				&\leq\delta_k d_{k} L (f(z_k)-f_\star)+\tfrac{L}{2}\normsq{x_k-x_\star}+e_k\sigma^2,
				\end{aligned}\end{equation*}
				where the last inequality follows from $\delta_k  L (1-\delta_k L)+1\geq 0$ (i.e., $\delta_k \leq \tfrac{1+\sqrt{5}}{2L}$).
			\end{proof}
			\subsection{Proof of Theorem~\ref{thm:pavg_BV}}\label{sec:primalavg_pot}
			\begin{proof} Combine the following inequalities with their corresponding weights.
				\begin{itemize}
					\item Convexity and smoothness between $y_{k+1}$ and $x_\star$ with weight $\lambda_1=\delta_k L$
					\[f_\star\geq f(y_{k+1})+\inner{f'(y_{k+1})}{x_\star-y_{k+1}}+\tfrac{1}{2L}\normsq{f'(y_{k+1})}, \]
					\item convexity and smoothness between $y_{k+1}$ and $y_k$ with weight $\lambda_2=d_k$
					\[f(y_k)\geq f(y_{k+1})+\inner{f'(y_{k+1})}{y_k-y_{k+1}}+\tfrac{1}{2L}\normsq{f'(y_{k+1})-f'(y_k)}, \]
					\item bounded variance at $y_{k+1}$ with weight $\lambda_3=e_k=\tfrac{\delta_k^2 L}{2}$
					\[\E_{i_k}\normsq{G(y_{k+1};{i_k})-f'(y_{k+1})}\leq \sigma^2. \]
				\end{itemize}
				The weighted sum can be reformulated as
				\begin{equation*}
				\begin{aligned}
				0\geq & \lambda_1\left[f(y_{k+1})-f_\star+\inner{f'(y_{k+1})}{x_\star-y_{k+1}}+\tfrac{1}{2L}\normsq{f'(y_{k+1})}\right]\\
				&+\lambda_2\left[f(y_{k+1})-f(y_k)+\inner{f'(y_{k+1})}{y_k-y_{k+1}}+\tfrac{1}{2L}\normsq{f'(y_{k+1})-f'(y_k)}\right]\\
				&+\lambda_3\left[\E_{i_k}\normsq{G(y_{k+1};{i_k})-f'(y_{k+1})}- \sigma^2\right]\\
				=& 	(d_{k}+ \delta_k L)[f(y_{k+1})-f_\star]+\tfrac{L}{2}\E_{i_k}\normsq{x_{k+1}^{({i_k})}-x_\star}- d_{k} (f(y_k)-f_\star)-\tfrac{L}{2}\normsq{x_k-x_\star}-e_k\sigma^2\\
				&+\tfrac{\delta_k}{2}(1-\delta_k L)\normsq{f'(y_{k+1})}+\tfrac{d_k}{2L}\normsq{f'(y_{k+1})-f'(y_k))}.
				\end{aligned}
				\end{equation*}
				Rearranging the terms allows obtaining:
				\begin{equation*}
				\begin{aligned}
				(d_{k}+ \delta_k L)&[f(y_{k+1})-f_\star]+\tfrac{L}{2}\E_{i_k}\normsq{x_{k+1}^{({i_k})}-x_\star}\\&\leq d_{k} (f(y_k)-f_\star)+\tfrac{L}{2}\normsq{x_k-x_\star}+e_k\sigma^2\\&\quad-\tfrac{\delta_k}{2}(1-\delta_k L)\normsq{f'(y_{k+1})}-\tfrac{d_k}{2L}\normsq{f'(y_{k+1})-f'(y_k))},
				\\&\leq d_{k} (f(y_k)-f_\star)+\tfrac{L}{2}\normsq{x_k-x_\star}+e_k\sigma^2,
				\end{aligned}
				\end{equation*}
				where the last inequality follows from $\delta_k\leq \tfrac1L$.
			\end{proof}
			
			\subsection{Evaluation of the stochastic gradient at the true averaged iterate}\label{sec:alt_pavging}
			The target algorithm we study here is as follows
			\begin{equation*}
			\begin{aligned}
			y_{k+1}&=\tfrac{k}{k+1}y_k+\tfrac{1}{k+1}x_k,\\
			x_{k+1}^{({i_k})}&=x_{{\color{red}k}}-\delta_k G(y_{k+1};{i_k}).
			\end{aligned}
			\end{equation*}
			\begin{theorem}\label{thm:pavgbis_BV}
				Consider the following iterative scheme
				\begin{equation*}
				\begin{aligned}
				y_{k+1}&=\tfrac{d_k}{d_k+1}y_k+\tfrac{1}{d_k+1}x_k,\\
				x_{k+1}^{({i_k})}&=x_{{\color{red}k}}-\delta_k G(y_{k+1};{i_k}),
				\end{aligned}
				\end{equation*}
				for some $d_k\geq0$ and $0\leq\delta_k\leq \tfrac1L$. Assuming $f\in\FL$ the following inequality holds
				\[d_{k+1}\delta_k L[f(y_{k+1})-f_\star]+\tfrac{L}{2}\E_{i_k}\normsq{x_{k+1}^{({i_k})}-x_\star}\leq d_{k}\delta_k L (f(y_k)-f_\star)+\tfrac{L}{2}\normsq{x_k-x_\star}+e_k\sigma^2,\]
				with $d_{k+1}=d_k+1$ and $e_k=\tfrac{L\delta_k^2}{2}$.
			\end{theorem}
			\begin{proof}
				Combine the following inequalities with their corresponding weights.
				\begin{itemize}
					\item Convexity and smoothness between $y_{k+1}$ and $x_\star$ with weight $\lambda_1=\delta_k L$
					\[f_\star\geq f(y_{k+1})+\inner{f'(y_{k+1})}{x_\star-y_{k+1}}+\tfrac{1}{2L}\normsq{f'(y_{k+1})}, \]
					\item convexity between $y_{k+1}$ and $y_k$ with weight $\lambda_2=d_k \delta_k L$ {{\bf\color{red}[Update (improvement): smoothness part was not necessary]}}
					\[f(y_k)\geq f(y_{k+1})+\inner{f'(y_{k+1})}{y_k-y_{k+1}}, \]
					\item bounded variance at $y_{k+1}$ with weight $\lambda_3=e_k=\tfrac{ \delta_k^2 L}{2}$
					\[\E_{i_k}\normsq{G(y_{k+1};{i_k})-f'(y_{k+1})}\leq \sigma^2. \]
				\end{itemize}
				The corresponding weighted sum can be reformulated as {{\bf\color{red}[Update (improvement): removed effect of smoothness through the second inequality]}}
				\begin{equation*}
				\begin{aligned}
				0\geq & \lambda_1\left[f(y_{k+1})-f_\star+\inner{f'(y_{k+1})}{x_\star-y_{k+1}}+\tfrac{1}{2L}\normsq{f'(y_{k+1})}\right]\\
				&+\lambda_2 \left[f(y_{k+1})-f(y_k)+\inner{f'(y_{k+1})}{y_k-y_{k+1}}\right]\\
				&+\lambda_3 \left[\E_{i_k}\normsq{G(y_{k+1};{i_k})-f'(y_{k+1})}- \sigma^2\right]\\
				=&(d_k+1)\delta_k L (f(y_{k+1})-f_\star)+\tfrac{L}2\E_{i_k}\normsq{x_{k+1}^{({i_k})}-x_\star}- d_k\delta_k L(f(y_{k})-f_\star)-\tfrac{L}2\normsq{x_k-x_\star}\\
				&\quad-e_k \sigma^2+\tfrac{\delta_k}{2}(1-\delta_k L)\normsq{f'(y_{k+1})}.
				\end{aligned}
				\end{equation*}
				After rearranging the terms, we reach: {{\bf\color{red}[Update: (typo) there was a unnecessary ``$- d_k\delta_k$'' term in the reformulation, and (improvement) removed effect of smoothness through the second inequality]}}
				\begin{equation*}
				\begin{aligned}
				(d_k+1)\delta_k L &(f(y_{k+1})-f_\star)+\tfrac{L}2\E_{i_k}\normsq{x_{k+1}^{({i_k})}-x_\star}\\
				&\leq  d_k\delta_k L(f(y_{k})-f_\star)+\tfrac{L}2\normsq{x_k-x_\star}+e_k \sigma^2 -\tfrac{\delta_k}{2}(1-\delta_k L)\normsq{f'(y_{k+1})}\\
				&\leq  d_k\delta_k L(f(y_{k})-f_\star)+\tfrac{L}2\normsq{x_k-x_\star}+e_k \sigma^2
				\end{aligned}
				\end{equation*}
				where the last inequality follows from $\delta_k\leq \tfrac1L$.
			\end{proof}
			\subsection{Convergence rates}\label{sec:BV_rates}
			In this section, we use the following two facts (they can easily be recovered by upper and lower bounding the sums with appropriate integrals)
			\begin{enumerate}
				\item $\sum_{t=1}^N t^{-\alpha}=O(N^{1-\alpha})$ for $\alpha \neq 1$,
				\item $\sum_{t=1}^N t^{-1}=O(\log N)$.
			\end{enumerate}
			for obtaining asymptotic rates for the previous methods when $\delta_k=(L(1+k)^\alpha)^{-1}$.
			
			\paragraph{Stochastic gradient descent.} From Theorem~\ref{thm:sgd_BV}, we have
			\begin{equation*}
			\begin{aligned}
			(d_k+\delta_k L)\E_{i_k}[f(x_{k+1}^{({i_k})})-f_\star]+&\tfrac{L}{2}\E_{i_k}\normsq{x_{k+1}^{({i_k})}-x_\star}\\&\leq d_{k} (f(x_k)-f_\star)+\tfrac{L}{2}\normsq{x_k-x_\star}+\tfrac{\delta_k^2 L}{2}(1+d_k+\delta_k L)\sigma^2.
			\end{aligned}
			\end{equation*}
			The choice $d_0=0$ leads to
			\begin{align*}
			\left(\sum_{t=0}^{N-1}L\delta_t\right) \E(f(x_{N})-f_\star)\leq \frac{L}{2}\normsq{x_0-x_\star}+\frac{\sigma^2}{2}\sum_{k=0}^{N-1} \left[{L^2\delta_k^2}\left(1+\sum_{t=0}^k L\delta_t\right)\right].
			\end{align*}
			For the choice $\delta_k=(L(1+k)^\alpha)^{-1}$, the different terms behave as follows:
			\begin{itemize}
				\item $\sum_{k=0}^{N-1}L\delta_k\sim N^{1-\alpha}$ when $\alpha\neq 1$,
				\item $\sum_{k=0}^{N-1} {L^2\delta_k^2} \sim N^{1-2\alpha}$ for $\alpha\neq 1/2$,
				\item $\sum_{k=0}^{N-1} \left[{L^2\delta_k^2}\sum_{t=0}^k L\delta_t\right]\sim \sum_{k=0}^{N-1} k^{1-3\alpha}\sim N^{2-3\alpha}$ when $\alpha\neq 1$ and $\alpha\neq 2/3$,
			\end{itemize}
			Under the same restrictions on $\alpha$, we also get:
			\[ \left(\sum_{k=0}^{N-1}L\delta_k\right)^{-1}\sim N^{ \alpha-1},\quad \frac{\sum_{k=0}^{N-1} {L^2\delta_k^2}}{\sum_{t=0}^{N-1} L\delta_t}\sim N^{-\alpha},\quad \frac{\sum_{k=0}^{N-1} \left[{L^2\delta_k^2}\sum_{t=0}^k L\delta_t\right]}{\sum_{k=0}^{N-1}L\delta_k}\sim N^{1-2\alpha},\]
			and hence
			\begin{align*}
			\E f(x_{N})-f_\star\leq  \frac{L}{2}\normsq{x_0-x_\star}\, \bO(N^{\alpha-1}) +\frac{\sigma^2}{2} \left[\bO\left(N^{-\alpha}\right)+\bO\left(N^{1-2\alpha}\right)\right].
			\end{align*}
			\paragraph{Stochastic gradient descent with averaging.} From Theorem~\ref{thm:avg_BV}, we have
			\begin{equation*}
			\begin{aligned}
			(d_k+1)L\delta_k\E_{i_k}[f(z^{({i_k})}_{k+1})-f_\star]+&\tfrac{L}{2}\E_{i_k}\normsq{x_{k+1}^{({i_k})}-x_\star}\\&\leq d_{k}L\delta_k (f(z_k)-f_\star)+\tfrac{L}{2}\normsq{x_k-x_\star}+\tfrac{\delta_k^2}{2} \tfrac{L(1+d_k+L\delta_k)}{1+d_k}\sigma^2.
			\end{aligned}
			\end{equation*}
			Let us choose $d_0=0$ and define $d_{k+1}:=\tfrac{\delta_k}{\delta_{k+1}}(d_k+1)$, $D_k:=L\delta_{k-1}d_{k-1}+L\delta_{k-1}$ (with $D_0:=0$); one can write
			\[D_{k+1}\E_{i_k}[f(z^{({i_k})}_{k+1})-f_\star]+\tfrac{L}{2}\E_{i_k}\normsq{x_{k+1}^{({i_k})}-x_\star}\leq D_{k} (f(z_k)-f_\star)+\tfrac{L}{2}\normsq{x_k-x_\star}+\tfrac{\delta_k^2}{2} \tfrac{L(D_{k+1}+L^2\delta_k^2)}{D_{k+1}}\sigma^2.\]
			In addition, we get $D_{k+1}=L\delta_{k}d_{k}+L\delta_{k}=L\delta_{k}\tfrac{\delta_{k-1}}{\delta_k}(d_{k-1}+1)+L\delta_{k}=D_k+L\delta_k=\sum_{t=0}^{k}L\delta_t$, and arrive to the final guarantee
			\[D_N\E[f(z_{N})-f_\star]\leq \tfrac{L}{2}\normsq{x_0-x_\star} + \tfrac{\sigma^2}{2L}\sum_{k=1}^{N}\delta_{k-1}^2L^2 \tfrac{D_{k}+L^2\delta_{k-1}^2}{D_{k}},\]
			from which one can arrive to the following rates:
			\[\E f(z_{N})-f_\star\leq \tfrac{L}{2}\normsq{x_0-x_\star} \,\bO(N^{\alpha-1})+ \tfrac{\sigma^2}{2L} [\bO(N^{-\alpha})+\bO(N^{-1-2\alpha})],\]
			where we used the following estimates of the rates for the different terms:
			\begin{itemize}
				\item $D_N=\sum_{k=0}^{N-1} L\delta_k\sim N^{1-\alpha}$ when $\alpha\neq 1$,
				\item $\sum_{k=1}^{N} L^2\delta_{k-1}^2\sim N^{1-2\alpha}$ when $\alpha\neq 1/2$,
				\item $\sum_{k=1}^{N} \frac{L^4\delta_{k-1}^4}{D_k}\sim \sum_{k=1}^{N} {k^{-1-3\alpha}}\sim N^{-3\alpha} $ when $\alpha\neq 1$ and $\alpha\neq 0$,
			\end{itemize}
			and under the same restrictions on $\alpha$: 
			\[ D_N^{-1} \sim N^{\alpha-1},\quad  \frac{\sum_{k=1}^{N} L^2\delta^2_{k-1}}{D_N}\sim N^{-\alpha},\quad \frac{\sum_{k=1}^{N} \frac{L^4\delta^4_{k-1}}{D_k}}{D_N}\sim N^{-1-2\alpha}.\]
			
			\paragraph{Stochastic gradient descent with primal averaging.} From Theorem~\ref{thm:pavg_BV}, we have
			\[(d_{k}+\delta_k L)[f(y_{k+1})-f_\star]+\tfrac{L}{2}\E_{i_k}\normsq{x_{k+1}^{({i_k})}-x_\star}\leq d_{k} (f(y_k)-f_\star)+\tfrac{L}{2}\normsq{x_k-x_\star}+\tfrac{L\delta_k^2}{2}\sigma^2,\]
			hence choosing $d_0=0$ and using the same estimates as before, we get
			\[ \E f(y_{N})-f_\star\leq \tfrac{L}{2}\normsq{x_0-x_\star}\bO(N^{\alpha-1})+\tfrac{\sigma^2}{2L}\bO(N^{-\alpha}).\]
			
			\subsection{Better worst-case guarantees and rates} \label{sec:sgd_fast}
			Different techniques can be used for obtaining stochastic methods with improved worst-case bounds. Among others, one can (i) assume the number of iterations to be fixed in advance, (ii) assume the domain to be compact, or (iii) use more past information --- essentially through use of a \emph{dual averaging} scheme. Algorithms using (iii) do not directly fit within our framework, which assumes (very) limited memory through~\eqref{eq:algo}---we believe it could be adapted, but let it for further investigations. The previous points (i) and (ii) are used by~\citet{lan2012optimal}, point (ii) is used e.g., by~\citet{hu2009accelerated}. Point (iii) is used by~\citet[Section 6]{devolder2011stochastic} and by~\citet[Section 7.1 and Appendix D]{xiao2010dual}.
			
			In this section, we briefly discuss two ways of improving the complexity results using namely techniques (i) and an alternative to (iii) for using more past information by adding a new sequence to~\eqref{eq:algo}. We start by assuming a known number of iterations $N$. In this setting, we get the following results.		
			\paragraph{Fixed number of iterations.} Consider a stochastic gradient scheme with primal averaging and a constant step-size $\delta_k=\delta$. Denoting $R=\norm{x_0-x_\star}$ and choosing $\delta=\frac{1}{\sqrt{N}\frac\sigma{R}+L}$, we get
			\[\E f(y_N)-f_\star\leq \tfrac{R^2}{2 \delta  N}+\tfrac{\delta  \sigma ^2}{2}= \tfrac{L^2 R^3 }{2  L R N+2 N^{3/2} \sigma }+\tfrac{R \sigma }{\sqrt{N}}\leq \tfrac{LR^2}{2N}+\tfrac{R \sigma }{\sqrt{N}}\]
			(note though that the optimal step-size in that setting is $\delta=\frac{R}{\sqrt{N} \sigma }$ (when $\frac{R}{\sqrt{N} \sigma }\leq \tfrac1L$ following the statement of Theorem~\ref{thm:pavg_BV}), leading to $\E f(y_N)-f_\star\leq \frac{R \sigma }{\sqrt{N}}$).
			
			\paragraph{Dual averaging.} Through the use of the dual averaging technique (iii)~\citet[Section 6]{devolder2011stochastic} and~\citet[Section 7.1 and Appendix D]{xiao2010dual}, one can observe that $x_0$ plays an important role as its weight grows in the dual averaging process (through coefficients $\beta_i$'s in both \citep{devolder2011stochastic} and \citep{xiao2010dual}). In a certain sense, one can interpret $x_0$ as a \emph{magnet} or \emph{anchor} whose goal is to stabilize (and slow down) the iterative process. As such, our framework does not allow treating estimate sequences with the damping with $x_0$ (recall that the whole point of potential-based proofs is to forget how the current iterate was obtained). Still, let us have a glance at the answers provided by the framework without dual averaging.
			
			\paragraph{Momentum without dual averaging nor damping.}
			Using the line-search design procedure, we obtain the following methods that appears to be symptomatic of all accelerated methods without damping so far: a huge error accumulation. This phenomenon is not surprising nor new.  Using the previous parameter selection technique, we arrive to the following theorem, which establishes the (negative and non-surprising) upper bound obtained through our framework, the main message being that the accumulation term $\sum_{k=0}^{N-1}e_k=\bO(N^3)$ in both cases presented below (whereas $d_N=\bO(N^2)$) can hardly be avoided while using \emph{pure acceleration} without damping---usually achieved by appropriate dual averaging schemes~\citep{xiao2010dual,devolder2011stochastic}. 	
				\begin{theorem}
					Let $k\in\N$, $d_{k},d_{k+1}\geq 0$, $f\in\FL$ and the algorithms
					\begin{equation*}
					\begin{aligned}
					y_{k+1}&=\argm{x}{f(x) \st x \in x_k+\sspan \{z_k-x_k\}},  \\
					x_{k+1}^{(i_{k})}&=\argm{x}{f(x) \st x \in y_{k+1}+\sspan\{G(y_{k+1};i_{k})\}},\\
					z_{k+1}^{(i_{k})}&=z_k - \tfrac{d_{k+1}-d_k}L G(y_{k+1};i_{k}),\\
					\end{aligned}
					\end{equation*}
					and
					\begin{equation*}
					\begin{aligned}
					y_{k+1}&=x_k+\left(1-\tfrac{d_k}{d_{k+1}}\right) \, (z_k-x_k),  \\
					x_{k+1}^{(i_{k})}&= y_{k+1}- \tfrac{\eta_k}{d_{k+1}}G(y_{k+1};i_{k}),\\
					z_{k+1}^{(i_{k})}&=z_k - \tfrac{d_{k+1}-d_k}L G(y_{k+1};i_{k}).\\
					\end{aligned}
					\end{equation*}
					In both cases, we have
					\[ d_{k+1}\E_{i_k}(f(x_{k+1}^{(i_k)})-f_\star)+\tfrac L 2\E_{i_k}\normsq{z_{k+1}^{(i_k)}-x_\star}\leq d_k (f(x_k)-f_\star)+\tfrac{L}2\normsq{z_k-x_\star}+e_k\sigma^2 \]
					for all choices of $e_k,\eta_k\geq 0$ satisfying both $\tfrac{d_{k+1}-d_k}{2 L}-e_k+\eta_k\geq 0$ and $ 2 d_{k+1} e_k L-d_{k+1} (d_{k+1}-d_k)^2-\eta_k ^2 L^2\geq 0$. In particular, the following three choices are valid:
					\begin{itemize}
						\item $d_k=\tfrac{k(k+1)}{4}$ for all $k\geq 0$ together with $\eta_k =\frac{(k+1) \left(k+2\right)}{4L}$ and $e_k=\tfrac{(k+1) (k+3)}{4 L}$, or
						\item $d_k=\tfrac{k(k+1)}{4}$ for all $k\geq 0$  together with the slightly better $\eta_k=\tfrac{(k+1) \left(k+2-\sqrt{3} \sqrt{k+2}\right)}{4 L}$, $e_k=\tfrac{(k+1) \left(k+3-\sqrt{3} \sqrt{k+2}\right)}{4 L}$, or
						\item (primal averaging) $d_{k+1}=d_k+\delta_k L$ together with $\eta_k=0$, $e_k=\tfrac{L\delta_k^2}{2}$ when $\delta_k\leq \tfrac{1}L$.
					\end{itemize}
				\end{theorem}
				\begin{proof}
					Perform a weighted sum of the following inequalities:
					\begin{itemize}
						\item smoothness and convexity between $x_\star$ and $y_{k+1}$ with weight $\lambda_1=d_{k+1}-d_k$
						\[f_\star\geq f(y_{k+1})+\inner{f'(y_{k+1})}{x_\star-y_{k+1}}+\tfrac1{2L}\normsq{f'(y_{k+1})}, \]
						\item averaged smoothness and convexity between $y_{k+1}$ and $x_{k+1}^{(i_k)}$ with weight $\lambda_2=d_{k+1}$
						\[f(y_{k+1}) \geq \E_{i_k}f(x_{k+1}^{(i_k)})+\E_{i_k}\inner{f'(x_{k+1}^{(i_k)})}{y_{k+1}-x_{k+1}^{(i_k)}}+\tfrac1{2L}\E_{i_k}\normsq{f'(y_{k+1})-f'(x_{k+1}^{(i_k)})}, \]
						\item convexity between $x_k$ and $y_{k+1}$ with weight $\lambda_3=d_k$
						\[f(x_k)\geq f(y_{k+1})+\inner{f'(y_{k+1})}{x_k-y_{k+1}}, \]
						\item first line-search optimality condition of $y_{k+1}$ with weight $\lambda_4=d_{k+1}$
						\[\inner{f'(y_{k+1})}{x_k-y_{k+1}}\geq 0, \]
						\item second line-search optimality condition of $y_{k+1}$ with weight $\lambda_5=d_{k+1}-d_k$
						\[\inner{f'(y_{k+1})}{z_k-x_k}\geq 0,\]
						\item first (averaged) line-search optimality condition for $x_{k+1}^{(i_k)}$ with weight $\lambda_6=d_{k+1}$
						\[\E_{i_k}\inner{f'(x_{k+1}^{(i_k)})}{y_{k+1}-x_{k+1}^{(i_k)}}\geq 0, \]
						\item second (averaged) line-search optimality condition for $x_{k+1}^{(i_k)}$ with weight $\lambda_7=\eta_k$:
						\[\E_{i_k}\inner{f'(x_{k+1}^{(i_k)})}{G(y_{k+1};i_k)}\leq 0, \]
						\item bounded variance at $y_{k+1}$ with weight $\lambda_8=e_k$
						\[\E_{i_k}\normsq{G(y_{k+1};i_k)-f'(y_{k+1})}\leq \sigma^2. \]
					\end{itemize}
					Now, the weighted sum corresponds to
					\begin{equation*}
					\begin{aligned}
					d_{k+1}&\E_{i_k}(f(x_{k+1}^{(i_k)})-f_\star)+\tfrac L 2\E_{i_k}\normsq{z_{k+1}^{(i_k)}-x_\star}\\\leq& d_k (f(x_k)-f_\star)+\tfrac{L}2\normsq{z_k-x_\star}+e_k\sigma^2 \\
					&-\tfrac{d_{k+1}}{2 L} \E_{i_k}\normsq{\tfrac{\eta_k L}{d_{k+1}}G(y_{k+1};i_k) +f'(x_{k+1}^{(i_k)})-f'(y_{k+1})}\\
					&- \left(\tfrac{d_{k+1}-d_k}{2 L}-e_k+\eta_k \right)\normsq{f'(y_{k+1})}\\
					&- \tfrac{ 2 d_{k+1} e_k L-d_{k+1} (d_{k+1}-d_k)^2-\eta_k ^2 L^2}{2 d_{k+1} L}\E_{i_k}\normsq{G(y_{k+1};i_k)},
					\end{aligned}
					\end{equation*}
					therefore, this weighted sum corresponds to the statement of the theorem when the three conditions are met:
					\begin{equation*}
					\begin{aligned}
					\tfrac{d_{k+1}}{2 L}&\geq0,\\
					\left(\tfrac{d_{k+1}-d_k}{2 L}-e_k+\eta_k \right)&\geq 0, \\
					\tfrac{ 2 d_{k+1} e_k L-d_{k+1} (d_{k+1}-d_k)^2-\eta_k ^2 L^2}{2 d_{k+1} L}&\geq 0,
					\end{aligned}
					\end{equation*}
					which are easy to check in the case of primal averaging. Another possible choice is to set $e_k=\tfrac{d_{k+1}-d_k}{2 L}+\eta_k$ and $\eta_k$ such that $\tfrac{ 2 d_{k+1} e_k L-d_{k+1} (d_{k+1}-d_k)^2-\eta_k ^2 L^2}{2 d_{k+1} L}\geq 0$; in particular, when $d_k=\tfrac{k(k+1)}{4}$ for all $k\geq 0$ we get
					\[ \eta_k \in \left[\tfrac{(k+1) \left(k+2-\sqrt{3} \sqrt{k+2}\right)}{4 L},\tfrac{(k+1) \left(k+2+\sqrt{3} \sqrt{k+2}\right)}{4 L}\right],\]
					a possible choice is $\eta_k =\frac{(k+1) \left(k+2\right)}{4 L}$, resulting in a critical accumulation in the variance term 
					\begin{equation*}
					\begin{aligned}
					d_{k+1}&\E_{i_k}(f(x_{k+1}^{(i_k)})-f_\star)+\tfrac L 2\E_{i_k}\normsq{z_{k+1}^{(i_k)}-x_\star}\\\leq& d_k (f(x_k)-f_\star)+\tfrac{L}2\normsq{z_k-x_\star}+\tfrac{(k+1) (k+3)}{4 L}\sigma^2 \\
					&-\tfrac{(k+1) (k+2)}{8 L} \E_{i_k}\normsq{f'(x_{k+1}^{(i_k)})-f'(y_{k+1})+G(y_{k+1};i_k)}\\
					&-\tfrac{3(k+1)}{8 L}\E_{i_k}\normsq{G(y_{k+1};i_k)}\\ \leq& d_k (f(x_k)-f_\star)+\tfrac{L}2\normsq{z_k-x_\star}+\tfrac{(k+1) (k+3)}{4 L}\sigma^2,\\
					\end{aligned}
					\end{equation*}
					where the accumulation arrives to $\tfrac{\sigma^2}{4L}\sum_{k=0}^{N-1}(k+1)(k+3)=\bO(N^3)$. It is also possible to choose the better (smaller) $\eta_k=\tfrac{(k+1) \left(k+2-\sqrt{3} \sqrt{k+2}\right)}{4 L}$ leading to
					\begin{equation*}
					\begin{aligned}
					d_{k+1}&\E_{i_k}(f(x_{k+1}^{(i_k)})-f_\star)+\tfrac L 2\E_{i_k}\normsq{z_{k+1}^{(i_k)}-x_\star}\\\leq& d_k (f(x_k)-f_\star)+\tfrac{L}2\normsq{z_k-x_\star}+\tfrac{(k+1) \left(k+3-\sqrt{3} \sqrt{k+2}\right)}{4 L}\sigma^2 \\
					&-\tfrac{(k+1) (k+2)}{8 L} \E_{i_k}\normsq{f'(x_{k+1}^{(i_k)})-f'(y_{k+1})+\left(1-\tfrac{\sqrt{3}}{\sqrt{k+2}}\right)G(y_{k+1};i_k)}\\\leq& d_k (f(x_k)-f_\star)+\tfrac{L}2\normsq{z_k-x_\star}+\tfrac{(k+1) \left(k+3-\sqrt{3} \sqrt{k+2}\right)}{4 L}\sigma^2,
					\end{aligned}
					\end{equation*}
					but the result remains philosophically the same. The proof is valid for all methods proposed above as they all satisfy both conditions
					\begin{equation*}
					\begin{aligned}
						&\inner{f'(y_{k+1})}{d_{k+1}(x_k-y_k)+(d_{k+1}-d_k)(z_k-x_k)}\geq 0,\\ &\E_{i_k}\inner{f'(x_{k+1}^{(i_k)})}{d_{k+1}(y_{k+1}-x_{k+1}^{(i_k)})-\eta_k G(y_{k+1};i_k)}\geq0.
					\end{aligned}
					\end{equation*}
				\end{proof}
			\subsection{Linear matrix inequalities}\label{sec:SDP_formulation_bv}
			In this section, we provide the linear matrix inequalities for the basic stochastic gradient method $x_{k+1}^{(i_k)}=x_k-\delta_kG(x_k;i_k)$ with the following family of potentials
			\begin{equation}\label{eq:pot_sgd}
			\begin{aligned}
			\phi_k^f=\begin{pmatrix}x_k-x_\star\\ f'(x_k)\end{pmatrix}^\top \left[Q_k\otimes I_d\right]\begin{pmatrix}x_k-x_\star\\ f'(x_k)\end{pmatrix} + d_k\, (f(x_k)-f_\star).
			\end{aligned}
			\end{equation}
			with $Q_k\in\Sb^2$. The following lines can serve for reproducing Figure~\ref{fig:GM}; the code implementing the following lines is provided in \secref{sec:ccl}. 
			
			We use the following matrices $P$ and $F$ for encoding (stochastic) gradients and function values:
			\begin{equation*}
			\begin{aligned}
			&P= [ \ x_k \ | \ G(x_k;1) \ \hdots \ G(x_k;n) \ | \ f'(x_{k+1}^{(1)}) \ \hdots \ f'(x_{k+1}^{(n)})\ ]\\
			&F= [ \ f(x_k) \ | \ f(x_{k+1}^{(1)}) \ \hdots \ f(x_{k+1}^{(n)}) \ ].
			\end{aligned}
			\end{equation*}
			Define the following vectors for selecting entries in $P$ and $F$: $\bx_\star$, $\bx_0,\ \bx_{1},\ \hdots,\ \bx_{n}$, $\bG_{1},\ \hdots,\ \bG_{n}$, $\bg_\star,\ \bg_{0},\ \bg_{1},\ \hdots, \ \bg_{n}\in\R^{1+2n}$ and $\bfu_\star,\ \bfu_0$, $\bfu_{1},\ \hdots, \ \bfu_{n}\in\R^{1+n}$ such that
			\begin{equation*}
			\begin{aligned}
			 x_k=P\bx_0,\ f'(x_k)=P\bg_0,\ f(x_k)=F\bfu_0,\ G(x_k;i)=P\bG_{i}, \\
			 x_{k+1}^{(i)}=P\bx_i,\ f'(x_{k+1}^{(i)})=P\bg_{i},\ f(x_{k+1}^{(i)})=P\bfu_i,
			\end{aligned}
			\end{equation*}			
			for $i\in\{1,\hdots,n\}$ and $x_\star=P\bx_\star,\ g_\star=P\bg_\star,\ f_\star=F\bfu_\star$. More precisely we choose the following vectors in $\R^{1+2n}$
			\begin{equation*}
			\begin{aligned}
			&\bx_0:=e_1, \quad \bG_i :=e_{1+i} \quad (\text{for }i\in\{1,\hdots,n\}),\quad \bg_{0}:=\tfrac1n\sum_{i=1}^n\bG_i,\\
			&\bx_i:=\bx_0-\delta_k \bG_i \quad (\text{for }i\in\{1,\hdots,n\}),\quad \bg_i:=e_{1+n+i}\quad (\text{for }i\in\{1,\hdots,n\}),\\
			&\bx_\star := 0, \quad \bg_\star:=0,
			\end{aligned}
			\end{equation*}
			and the following vectors in $\R^{n+1}$: $
			\bfu_0:=e_1$, $\bfu_i:=e_{1+i}$ (for $i\in\{1,\hdots,n\}$), $\bfu_\star=0$.
			We also encode variance using a matrix $A_{\text{var}}$ such that
			\[ \tra (A_{\text{var}} P^\top P)=\E_i\normsq{G(x_k;i)-f'(x_k)},\]
			in others words $A_{\text{var}}:=\tfrac1n \sum_{i=1}^n (\bG_i-\bg_0)(\bG_i-\bg_0)^\top$. We use $\Sb^{1+2n}$ to denote the set of $(1+2n)\times (1+2n)$ symmetric matrices, and use $M$ as defined in~\eqref{eq:MM}. \correc{Strong duality, for obtaining the equivalence between the SDP and the following LMI formulation, might be obtained from a standard Slater argument (see e.g.,~\citet[Theorem 6]{taylor2017smooth})	and are not provided here.}
		\begin{oframed}
			\noindent Given two doublets $(d_k,Q_k)$, $(d_{k+1},Q_{k+1})$ and a step-size $\delta_k$, the inequality $\E_{i_k}\phi^f_{k+1}(x_{k+1}^{(i_k)})\leq \phi^f_{k}(x_k)+e_k\sigma^2$ ($\phi_k^f$ defined in~\eqref{eq:pot_sgd}) holds for all $f\in\FL(\Rd)$, all $d\in\N$, all $x_k\in\Rd$  and all sets $\{G(x_k;i)\}_{i=1,\hdots,n}\subset \R^d$ satisfying $\E_i G(x_k;i)=f'(x_k)$ and $\E_i\normsq{G(x_k;i)-f'(x_k)}\leq \sigma^2$ used to generate $x_{k+1}^{(i_k)}$ with $x_{k+1}^{(i_k)}=x_k-\delta_kG(x_k;i_k)$  if and only if there exists $\{\lambda_{i,j}\}_{i,j\in I}$, with $I:=\{\star,0,1,\hdots,n\}$ and $e_k\in\R$ such that
			\begin{equation*}
			\begin{aligned}
			& \lambda_{i,j} \geq 0 \text{ for all } i,j\in I \text{ and } e_k\geq 0,\\
			& d_{k+1} \tfrac1n \sum_{i=1}^n(\bfu_{i}-\bfu_\star)-d_k (\bfu_0-\bfu_\star) + \sum_{i,j\in I} \lambda_{i,j} (\bfu_{i}-\bfu_j) = 0 \text{ (linear constraint in $\R^{n+1}$)},\\
			&  \tfrac1n \sum_{i=1}^n V_{k+1}^{(i)}-V_k -e_k A_{\text{var}} + \sum_{i,j\in I} \lambda_{i,j} M_{i,j} \preceq 0 \text{ (linear matrix inequality in $\Sb^{1+2n}$)},
			\end{aligned}
			\end{equation*}
			with 
			\begin{equation*}
			\begin{aligned}
			V_{k}&:= \begin{pmatrix}
			\bx_{0}-\bx_\star & \bg_{0} 
			\end{pmatrix}Q_k\begin{pmatrix}
			\bx_{0}^\top-\bx_\star^\top \\ \bg_{0}^\top 
			\end{pmatrix}\in\Sb^{1+2n} ,\\
			V_{k+1}^{(i)}&:=\begin{pmatrix}
			\bx_{i}-\bx_\star & \bg_{i} 
			\end{pmatrix}Q_{k+1}\begin{pmatrix}
			\bx_{i}^\top-\bx_\star^\top \\ \bg_{i}^\top 
			\end{pmatrix}\in\Sb^{1+2n},\\
			M_{i,j}&:=\begin{pmatrix}
			\bx_i & \bx_j & \bg_i & \bg_j
			\end{pmatrix} M \begin{pmatrix}
			\bx_i & \bx_j & \bg_i & \bg_j
			\end{pmatrix}^\top\in\Sb^{1+2n}.
			\end{aligned}
			\end{equation*}
		\end{oframed}
			
			\paragraph{Parameter selection.} For adapting the strategy to our parameter selection technique, we refer to \appref{sec:LMI_design}, as the exact same tricks apply, though a bit more tedious in the stochastic setting.

			\clearpage 
			
			\section{Stochastic gradients for over-parameterized models}\label{sec:overparameterized}
			This section contains the proof of Theorem~\ref{thm:overparameterized}, a discussion on how we ended up with primal averaging using the parameter selection technique, and an LMI formulation for studying primal averaging.
			\subsection{Proof of Theorem~\ref{thm:overparameterized}}\label{sec:paving_pot_overp}
			\begin{proof} Let us start with the case $\delta_k\leq \tfrac1L$ and $d_{k+1}=d_k+\delta_k L$. As before, the proof consists in linearly combining the following inequalities.
				\begin{itemize}
					\item Averaged smoothness and convexity between $y_{k+1}$ and $x_\star$ with weight $\lambda_1=\delta_k L$:
					\[f_\star \geq  f(y_{k+1})+ \inner{f'(y_{k+1})}{x_\star-y_{k+1}}+\tfrac{1}{2L}\E_{i_k}\normsq{f'_{{i_k}}(y_{k+1})},\]
					\item convexity between ${\color{red}y}_{k+1}$ and ${\color{red}y_k}$ with weight $\lambda_2=d_k$:
					\[f(y_k) \geq f(y_{k+1}) + \inner{f'(y_{k+1})}{y_k-y_{k+1}}. \]
				\end{itemize}
				The weighted sum of those inequalities can be rewritten as
				\begin{equation*}
				\begin{aligned}
				0\geq & \lambda_1 \left[f(y_{k+1})-f_\star+ \inner{f'(y_{k+1})}{x_\star-y_{k+1}}+\tfrac{1}{2L}\E_{i_k}\normsq{f'_{{i_k}}(y_{k+1})}\right]\\
				&+\lambda_2\left[ f(y_{k+1}) -f(y_k)+ \inner{f'(y_{k+1})}{y_k-y_{k+1}}\right]\\
				=&(d_k+\delta_k L)(f(y_{k+1})-f_\star)+\tfrac{L}{2}\E_{i_k}\normsq{x_{k+1}^{({i_k})}-x_\star}- d_k (f(y_k)-f_\star)-\tfrac{L}{2}\normsq{x_k-x_\star}\\
				&\quad +\tfrac{\delta_k  }{2}\left( 1-\delta_k L\right)\E_{i_k}\normsq{f_{i_k}'(y_{k+1})},
				\end{aligned}
				\end{equation*}
				which can be rearranged to
				\begin{equation*}
				\begin{aligned}
				(d_k+\delta_k L)(f(y_{k+1})-f_\star)&+\tfrac{L}{2}\E_{i_k}\normsq{x_{k+1}^{({i_k})}-x_\star}\\ &\leq d_k (f(y_k)-f_\star)+\tfrac{L}{2}\normsq{x_k-x_\star}-\tfrac{\delta_k  }{2}\left( 1-\delta_k L\right)\E_{i_k}\normsq{f_{i_k}'(y_{k+1})}\\
				&\leq d_k (f(y_k)-f_\star)+\tfrac{L}{2}\normsq{x_k-x_\star},
				\end{aligned}
				\end{equation*}
				where the last inequality follows from $\delta_k\leq \tfrac1L$. In the case where $\delta_k\geq\tfrac1L$, we use instead 
				$d_{k+1}=d_k+2 \delta_k L - \delta_k^2 L^2$ and compensate the nonpositive term in the righthand side by using an additional inequality:
				\begin{itemize}
					\item averaged smoothness between $y_{k+1}$ and $x_\star$ with weight $\lambda_3=\delta_k L(\delta_k L-1)$ (which is nonnegative \correc{as soon as $\delta_k\geq \tfrac1L$})
					\[f(y_{k+1})\geq f_\star +\tfrac1{2L}\E_{i_k}\normsq{f_{i_k}'(y_{k+1})}.\]
				\end{itemize}
				Adding this inequality to the previous sum, we get:
				\begin{equation*}
				\begin{aligned}
				0\geq & \lambda_1 \left[f(y_{k+1})-f_\star+ \inner{f'(y_{k+1})}{x_\star-y_{k+1}}+\tfrac{1}{2L}\E_{i_k}\normsq{f'_{{i_k}}(y_{k+1})}\right]\\
				&+\lambda_2\left[ f(y_{k+1}) -f(y_k)+ \inner{f'(y_{k+1})}{y_k-y_{k+1}}\right]\\
				&+\lambda_3\left[f_\star -f(y_{k+1})+\tfrac1{2L}\E_{i_k}\normsq{f_{i_k}'(y_{k+1})}\right]\\
				=&(d_k+2 \delta_k L-\delta_k^2L^2)(f(y_{k+1})-f_\star)+\tfrac{L}{2}\E_{i_k}\normsq{x_{k+1}^{({i_k})}-x_\star}- d_k (f(y_k)-f_\star)-\tfrac{L}{2}\normsq{x_k-x_\star},
				\end{aligned}
				\end{equation*}
				which can be reformulated as
				\[(d_k+2 \delta_k L-\delta_k^2L^2)(f(y_{k+1})-f_\star)+\tfrac{L}{2}\E_{i_k}\normsq{x_{k+1}^{({i_k})}-x_\star}\leq d_k (f(y_k)-f_\star)+\tfrac{L}{2}\normsq{x_k-x_\star}.\]
			\end{proof}
			
			\subsection{Parameter selection} \label{sec:overp_param_selec}
			
			In this section, we quickly discuss how we ended up with primal averaging for overparametrized models (see Theorem~\ref{thm:overparameterized}). Essentially, we used the first parameter selection technique from \appref{sec:paramSelec}, which we adapted to
			\begin{equation*}
			\begin{aligned}
			y_{k+1}&= (1-\tau_k) x_k+\tau_k z_k,\\
			x_{k+1}^{(i_k)}&=y_{k+1}-\alpha_k f_{i_k}'(y_{k+1}),\\
			z_{k+1}^{(i_k)}&=(1-\delta_k) y_{k+1}+\delta_k z_k -\gamma_k f_{i_k}'(y_{k+1}),\\
			\end{aligned}
			\end{equation*}
			for which we wish to optimize the parameters $\{(\tau_k,\alpha_k,\delta_k,\gamma_k)\}_k$. Instead, we consider the line-search version
			\begin{equation*}
			\begin{aligned}
			y_{k+1}&=\argm{x}{f(x)\, \st\, x\in x_k+\sspan\{z_k-x_k\}},\\
			x_{k+1}&=\argm{x}{f(x)\, \st\, x\in y_{k+1}+\sspan \{f_{i_k}'(y_{k+1})\}},\\
			z_{k+1}&=(1-\delta_k) y_{k+1}+\delta_k z_k -\gamma_k f_{i_k}'(y_{k+1}),\\
			\end{aligned}
			\end{equation*}
			along with the following family of potentials 
			\begin{equation*}
			\begin{aligned}
			\phi_k^f= &q_{1,k}\, \normsq{x_k-x_\star}+q_{2,k}\, \normsq{f'(x_k)}+q_{3,k}\, \E_i \normsq{f_i'(x_k)}+q_{4,k}\,\inner{f'(x_k)}{x_k-x_\star}\\&+ d_k\, (f(x_k)-f_\star)+a_k \normsq{z_k-x_\star},
			\end{aligned}
			\end{equation*}
			which was chosen based on expected symmetries with respect to $i_k$. By picking $\phi_0^f=\tfrac{L}2 \normsq{x_0-x_\star}$ and $\phi_N^f=d_N \,(f(x_N)-f_\star)$, we naturally arrived to primal averaging by solving
			\begin{equation}\label{eq:pot_design_overp}
			\max_{\{(\delta_k,\gamma_k)\}_k}\quad \max_{\phi_1^f,\hdots,\phi_{N-1}^f,d_N} d_N \st (\phi_0^f,\phi_1^f)\in\tilde{\V}_0,\hdots,(\phi_{N-1}^f,\phi_N^f)\in\tilde{\V}_{N-1},
			\end{equation}
			as done in~\eqref{eq:FGM_lyap_ELS}; the numerical results are shown on Figure~\ref{fig:design_overparam}.
			
			\begin{center}
				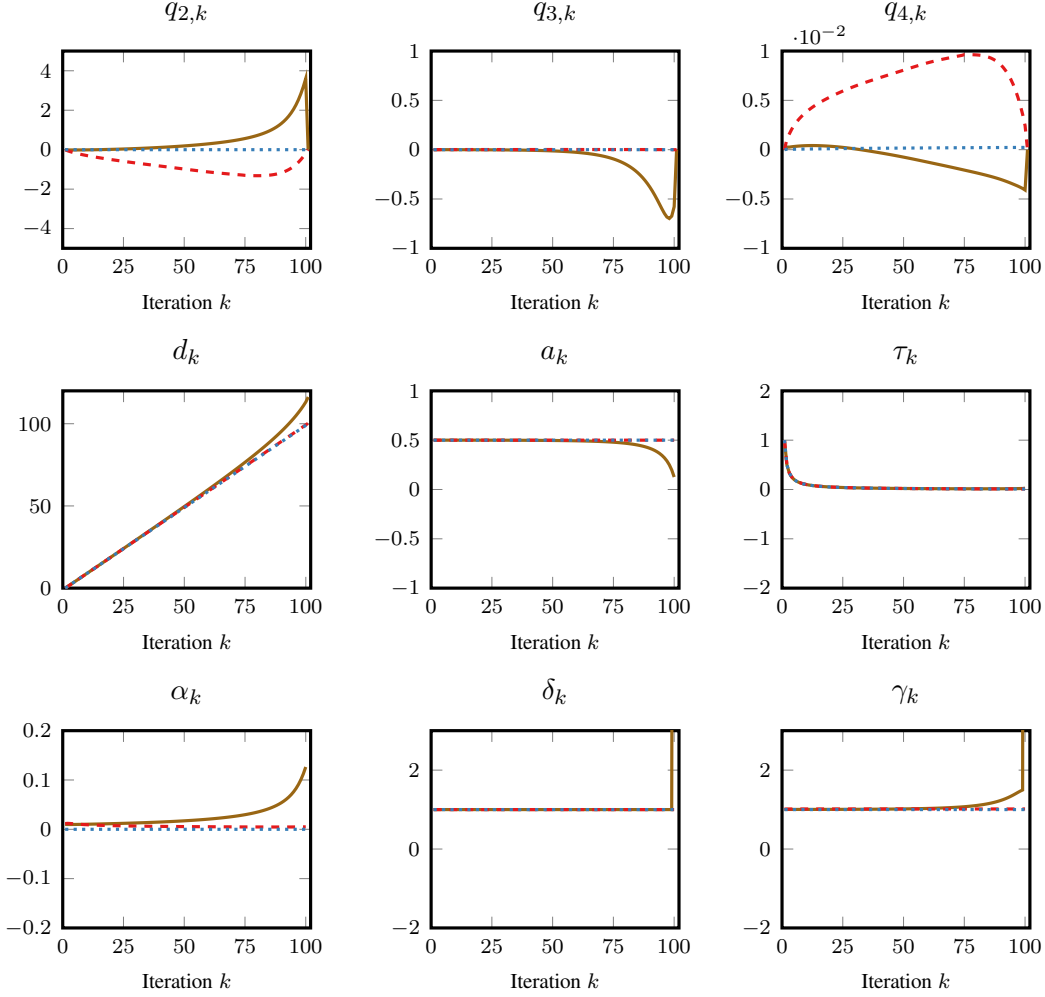
\begin{figure}[!ht]
					\begin{tabular}{rrr}
						\begin{tikzpicture}
						\begin{axis}[plotOptions, title={$q_{2,k}$}, ymin=-5, ymax=5,xmin=0,xmax=102,width=.32\linewidth]
						\addplot [colorP5] table [x=k,y=Q2] {Data/overparam_design_raw.dat};
						\addplot [colorP2, dashed] table [x=k,y=Q2] {Data/overparam_design_pure.dat};
						\addplot [colorP1, dotted] table [x=k,y=Q2] {Data/overparam_design_pavging.dat};
						\end{axis}
						\end{tikzpicture} &
						\begin{tikzpicture}
						\begin{axis}[plotOptions, title={$q_{3,k}$}, ymin=-1, ymax=1,xmin=0,xmax=102,width=.32\linewidth]
						\addplot [colorP5] table [x=k,y=Q3] {Data/overparam_design_raw.dat};
						\addplot [colorP2, dashed] table [x=k,y=Q3]{Data/overparam_design_pure.dat};
						\addplot [colorP1, dotted] table [x=k,y=Q3]{Data/overparam_design_pavging.dat};
						\end{axis}
						\end{tikzpicture} &
						\hspace{-.5cm}
						\begin{tikzpicture}
						\begin{axis}[plotOptions, title={$q_{4,k}$}, ymin=-.01, ymax=.01,xmin=0,xmax=102,width=.32\linewidth]
						\addplot [colorP5] table [x=k,y=Q4] {Data/overparam_design_raw.dat};
						\addplot [colorP2, dashed] table [x=k,y=Q4] {Data/overparam_design_pure.dat};
						\addplot [colorP1, dotted] table [x=k,y=Q4] {Data/overparam_design_pavging.dat};
						\end{axis}
						\end{tikzpicture}\\
						\begin{tikzpicture}
						\begin{axis}[plotOptions, title={$d_{k}$}, ymin=0, ymax=120,xmin=0,xmax=102,width=.32\linewidth]
						\addplot [colorP5] table [x=k,y=dk] {Data/overparam_design_raw.dat};
						\addplot [colorP2, dashed] table [x=k,y=dk] {Data/overparam_design_pure.dat};
						\addplot [colorP1, dotted] table [x=k,y=dk] {Data/overparam_design_pavging.dat};
						\end{axis}
						\end{tikzpicture} &
						\begin{tikzpicture}
						\begin{axis}[plotOptions, title={$a_k$}, ymin=-1, ymax=1,xmin=0,xmax=102,width=.32\linewidth]
						\addplot [colorP5] table [x=k,y=apk, skip coords between index={100}{101}] {Data/overparam_design_raw.dat};
						\addplot [colorP2, dashed] table [x=k,y=apk, skip coords between index={100}{101}] {Data/overparam_design_pure.dat};
						\addplot [colorP1,dotted] table [x=k,y=apk, skip coords between index={100}{101}] {Data/overparam_design_pavging.dat};
						\end{axis}
						\end{tikzpicture}&
						\begin{tikzpicture}
						\begin{axis}[plotOptions, title={$\tau_k$}, ymin=-2, ymax=2,xmin=0,xmax=102,width=.32\linewidth]
						\addplot [colorP5] table [x=k,y=tauk, skip coords between index={100}{101}] {Data/overparam_design_raw.dat};
						\addplot [colorP2, dashed] table [x=k,y=tauk, skip coords between index={100}{101}] {Data/overparam_design_pure.dat};
						\addplot [colorP1, dotted] table [x=k,y=tauk, skip coords between index={100}{101}] {Data/overparam_design_pavging.dat};
						\end{axis}
						\end{tikzpicture}\\
						\begin{tikzpicture}
						\begin{axis}[plotOptions, title={$\alpha_k$}, ymin=-.2, ymax=.2,xmin=0,xmax=102,width=.32\linewidth]
						\addplot [colorP5] table [x=k,y=alphak, skip coords between index={100}{101}] {Data/overparam_design_raw.dat};
						\addplot [colorP2, dashed] table [x=k,y=alphak, skip coords between index={100}{101}] {Data/overparam_design_pure.dat};
						\addplot [colorP1, dotted] table [x=k,y=alphak, skip coords between index={100}{101}] {Data/overparam_design_pavging.dat};
						\end{axis}
						\end{tikzpicture}&
						\begin{tikzpicture}
						\begin{axis}[plotOptions, title={$\delta_k$}, ymin=-2, ymax=3,xmin=0,xmax=102,width=.32\linewidth]
						\addplot [colorP5] table [x=k,y=deltak, skip coords between index={100}{101}] {Data/overparam_design_raw.dat};
						\addplot [colorP2, dashed] table [x=k,y=deltak, skip coords between index={100}{101}] {Data/overparam_design_pure.dat};
						\addplot [colorP1, dotted] table [x=k,y=deltak, skip coords between index={100}{101}] {Data/overparam_design_pavging.dat};
						\end{axis}
						\end{tikzpicture}&
						\begin{tikzpicture}
						\begin{axis}[plotOptions, title={$\gamma_k$}, ymin=-2, ymax=3,xmin=0,xmax=102,width=.32\linewidth]
						\addplot [colorP5] table [x=k,y=gammak, skip coords between index={100}{101}] {Data/overparam_design_raw.dat};
						\addplot [colorP2, dashed] table [x=k,y=gammak,skip coords between index={100}{101}] {Data/overparam_design_pure.dat};
						\addplot [colorP1, dotted] table [x=k,y=gammak,skip coords between index={100}{101}] {Data/overparam_design_pavging.dat};
						\end{axis}
						\end{tikzpicture}
					\end{tabular}\vspace{-.35cm}
					\caption{Numerical solution to~\eqref{eq:pot_design_overp} for $N=100$, $n=2$ and $L=1$ (plain brown, large values for $\delta_{100}$ and $\gamma_{100}$ were capped for readability purposes; they are due to the fact we impose no control on $z_N$ with our initial choice $\phi_{N}^f$), forced $a_k=\tfrac{L}2$ (dashed red), forced  $a_k=\tfrac{L}2$ and $\alpha_k=0$ (dotted blue). For convenience we do not show the values for $q_{1,k}$; they numerically appeared to be negligible for all values of $k$ (about $10^{-7}$) in all three scenarios. \edit{Total computational time: $\sim90$ sec{.} on single core of Intel Core i$7$ $1.8$GHz CPU.}}\label{fig:design_overparam}
				\end{figure}
			\end{center}			
			
			\subsection{Linear matrix inequalities for over-parametrized models} \label{sec:overp_SDP}
			The derivations of the linear matrix inequalities follow from the exact same line as in the previous sections. As an example, we provide the LMI formulation for analyzing primal averaging:
			\begin{equation}\label{eq:pavging_app}
			\begin{aligned}
			y_{k+1}&=\tau_k y_k+(1-\tau_k)x_k,\\
			x_{k+1}^{({i_k})}&=x_k-\gamma_k f_{{i_k}}'(y_{k+1}),
			\end{aligned}
			\end{equation}	
			along with the following family of potentials
			\begin{equation}\label{eq:pavging_pot}
			\begin{aligned}
			\phi_k^f=\begin{pmatrix}x_k-x_\star\\y_{k}-x_\star\\ f_1'(y_k)-f_1'(x_\star)\\\vdots\\ f_n'(y_k)-f_n'(x_\star)\end{pmatrix}^\top \left[Q_k\otimes I_d\right]\begin{pmatrix}x_k-x_\star\\y_k-x_\star\\ f_1'(y_k)-f_1'(x_\star)\\\vdots\\ f_n'(y_k)-f_n'(x_\star)\end{pmatrix}+d_k\, (f(y_k)-f_\star),
			\end{aligned}
			\end{equation}
			which can be simplified using appropriate symmetry arguments, but we simply formulate the problem without using them for tutorial purposes, and simply choose $Q_k\in\Sb^{2+n}$.
			One could additionally use function values at $x_{k+1}^{(i_k)}$ in the formulation, however, this did not improve the results in our experiments while increasing quite a bit the size of the Gram matrix by adding $n^2$ rows and columns to it (corresponding to all gradients \[f_1'(x_{k+1}^{(1)}),\hdots,f_n'(x_{k+1}^{(1)}),\hdots,f_1'(x_{k+1}^{(n)}),\hdots,f_n'(x_{k+1}^{(n)})).\]
			In order to reformulate~\eqref{eq:tt} we therefore only need to encode function values $F$ and gradient/coordinate in a matrix $P$, which we use for formulating the Gram matrix $G=P^\top P\succeq 0$:
			\begin{equation*}
			\begin{aligned}
			&P=[\ x_k \ | \ y_k \ | \ f_1'(y_k) \ \hdots\ f_n'(y_{k}) \ | \ f_1'(y_{k+1}) \ \hdots \ f_n'(y_{k+1})\, ], \\
			&F=[\ f_1(y_k) \ \hdots \ f_n(y_k) \ | \ f_1(y_{k+1}) \ \hdots \ f_n(y_{k+1}) \ ].
			\end{aligned}
			\end{equation*} 
			We also denote by $\by_\star,\ \bx_k,\ \by_k,\ \bg_{i,\star},\ \bg_{i,k},\ \bg_{i,k+1}\in\R^{2+2n}$ (for all $i\in\{1,\hdots,n\}$) such that
			\[ x_\star=P\by_\star,\, x_k=P\bx_k,\, y_k=P\by_k,\,f_i'(x_\star)=P\bg_{i,\star} ,\, f_i'(y_k)=P\bg_{i,k},\, f_i'(y_{k+1})=P\bg_{i,k+1}, \]
			that is $\by_\star=\bg_{i,\star}:=0$, $\bx_k:=e_1$, $\by_k:=e_2$, $\bg_{i,k}:=e_{2+i}$, $\bg_{i,k+1}:=e_{2+n+i}$. For iteration $k+1$:
			\begin{equation*}
			\begin{aligned}
				\by_{k+1}&:=\tau_k \by_k+(1-\tau_k)\bx_k,\\
				\bx_{k+1}^{({l})}&:=\bx_k-\gamma_k \bg_{l,k},
			\end{aligned}
			\end{equation*}
			in order to have $y_{k+1}=P\by_{k+1}$ and $x_{k+1}^{(l)}=P\bx_{k+1}^{({l})}$. Finally, we do the same for function values: $\bfu_{i,\star},\ \bfu_{i,k},\ \bfu_{i,k+1}\in\R^{2n}$ (for all $i\in\{1,\hdots,n\}$) with	
			\[ f_\star=F\bfu_\star,\, f_i(y_k)=F\bfu_{i,k},\, f_i(y_{k+1})=F\bfu_{i,k+1}, \]		
			that is $\bfu_{i,\star}:=0$, $\bfu_{i,k}:=e_i$ and $\bfu_{i,k+1}:=e_{n+i}$. For each function $f_l$ (with $l\in\{1,\hdots,n\}$) we introduce a set of multipliers $\{\lambda^{(l)}_{i,j}\}_{i,j\in I_k}$ with $I_k=\{\star,k,k+1\}$.  We use  $\Sb^{2+2n}$ is the space of $(2+2n)\times (2+2n)$ symmetric matrices, and use $M$ as defined in~\eqref{eq:MM}.
			
			\begin{oframed}
				\noindent Given two doublets $(d_k,Q_k)$ and $(d_{k+1},Q_{k+1})$, a pair $(\tau_k,\gamma_k)$, and some $n\in\N$, the inequality $\E_{i_k}\phi^f_{k+1}(x_{k+1}^{(i_k)},y_{k+1})\leq \phi^f_{k}(x_k,y_k)$ ($\phi_k^f$ defined in~\eqref{eq:pavging_pot}) holds for all $f(x)=\tfrac1n\sum_{i=1}^n f_i(x)$ such that $f_i\in\FL(\Rd)$, for all $d\in\N$, and all $x_k,y_k\in\Rd$ used to generate $x_{k+1}^{(i_k)},y_{k+1}$ with method~\eqref{eq:pavging_app} if and only if for all $l\in\{1,\hdots,n\}$ there exists $\{\lambda_{i,j}^{(l)}\}_{i,j\in I_k}$ with $I_k=\{\star,k,k+1\}$ such that
				\begin{equation*}
				\begin{aligned}
				& \lambda_{i,j}^{(l)} \geq 0 \text{ for all } i,j\in I_k \text{ and all } l\in\{1,\hdots,n\} ,\\
				& \tfrac{ d_{k+1}}{n}\sum_{l=1}^n(\bfu_{l,k+1}-\bfu_{l,\star})- \tfrac{d_k}n\sum_{l=1}^n(\bfu_{l,k}-\bfu_{l,\star}) + \sum_{\substack{i,j\in I_k\\ l\in\{1,\hdots,n\}}} \lambda_{i,j}^{(l)} (\bfu_{l,i}-\bfu_{l,j}) = 0 \text{ (in $\R^{2n}$)},\\
				&  \tfrac1n \sum_{l\in\{1,\hdots,n\}} V_{k+1}^{(l)}-V_k + \sum_{\substack{i,j\in I_k\\ l\in\{1,\hdots,n\}}} \lambda_{i,j}^{(l)} M_{i,j}^{(l)} \preceq 0 \text{ (linear matrix inequality in $\Sb^{2+2n}$)},
				\end{aligned}
				\end{equation*}
				with 
				\begin{equation*}
				\begin{aligned}
				V_{k}&:= \begin{pmatrix}
				\bx_{k}-\by_\star & \by_{k}-\by_\star & \bg_{1,k} & \hdots & \bg_{n,k} 
				\end{pmatrix}Q_k\begin{pmatrix}
				\bx_{k}^\top-\by_\star^\top \\ \by_{k}^\top-\by_\star^\top \\ \bg_{1,k}^\top \\ \vdots \\ \bg_{n,k}^\top 
				\end{pmatrix}\in\Sb^{2+2n},\\
				V_{k+1}^{(l)}&:= \begin{pmatrix}
				\bx_{k+1}^{(l)}-\by_\star & \by_{k+1}-\by_\star & \bg_{1,k+1} & \hdots & \bg_{n,k+1} 
				\end{pmatrix}Q_{k+1}\begin{pmatrix}
				\bx_{k+1}^{(l)\top}-\by_\star^\top \\ \by_{k+1}^\top-\by_\star^\top \\ \bg_{1,k+1}^\top \\ \vdots \\ \bg_{n,k+1}^\top 
				\end{pmatrix}\in\Sb^{2+2n},\\
				M_{i,j}^{(l)}&:=\begin{pmatrix}
				\by_i & \by_j & \bg_{l,i} & \bg_{l,j}
				\end{pmatrix} M \begin{pmatrix}
				\by_i & \by_j & \bg_{l,i} & \bg_{l,j}
				\end{pmatrix}^\top\in\Sb^{2+2n}.
				\end{aligned}
				\end{equation*}
			\end{oframed}
			\clearpage 
			\section{Stochastic gradients under weak growth conditions}\label{sec:weakgrowth_proofs}
			This section deals with weak growth conditions. It is straightforward to adapt the methodology to e.g., strong growth conditions. However, since recent works already covered the topic in a comprehensive way (see e.g.,~\citet{vaswani2018fast} and the references therein), we chose not to cover~it.
			
			Under similar motivations as those of the \secref{sec:ex_pravg_sgd}, consider minimizing $f\in\FL$
			\[ \min_{x\in\Rd} f(x),\]
			with an unbiased gradient estimate $G(x;i)$ satisfying a weak growth condition for some $\rho\geq 1$:
			\begin{equation*}
			\E_i G(x;i)=f'(x),\quad \E_i \normsq{G(x;i)}\leq 2\rho L (f(x)-f_\star),
			\end{equation*}
			for all $x\in\Rd$. In that setting, it was shown by~\citet{vaswani2018fast} that averaging was actually reaching the $\bO(k^{-1})$ convergence for $\E f(\bar{x}_k)-f_\star$. The best method we could obtain with the previous methodology achieves a $\bO(k^{-1})$ convergence for function values at its last iterate.
			
			\begin{theorem}\label{thm:wgc} Let $x_k\in\Rd$, $f\in\FL$ and an unbiased stochastic oracle satisfying a weak growth condition:
				\begin{equation*}
				\E_i G(x;i)=f'(x),\quad \E_i \normsq{G(x;i)}\leq 2\rho L (f(x)-f_\star)
				\end{equation*}
				for all $x\in\Rd$ and $i\in I$.	Then, the iterative scheme
				\begin{equation*}
				\begin{aligned}
				&y_{k+1} = \tfrac{d_k}{d_k+\delta_k L} y_k + \tfrac{\delta_k L}{d_k+\delta_k L} x_k,\\
				&x_{k+1}^{({i_k})} = x_k - \delta_k G(y_{k+1};{i_k}),
				\end{aligned}
				\end{equation*}
				satisfies
				\[d_{k+1}[f(y_{k+1})-f_\star]+\tfrac{L}{2}\E_{i_k}\normsq{x_{k+1}^{({i_k})}-x_\star}\leq d_{k} (f(y_k)-f_\star)+\tfrac{L}{2}\normsq{x_k-x_\star},\]
				for all values of $d_k,\delta_k\geq0$ and $d_{k+1}=d_k+\delta_k L - \rho \delta_k^2 L^2$.
			\end{theorem}
			Using $\delta_k=\tfrac{1}{2\rho L}$ (choice that maximizes $d_{k+1}$) and $d_0=0$ leads to $d_k=\tfrac{k}{4\rho}$ and to the algorithm
			\begin{equation*}
			\begin{aligned}
			&y_{k+1} = \tfrac{k}{k+2} y_k + \tfrac{2}{k+2} x_k,\\
			&x_{k+1}^{({i_k})} = x_k - \tfrac{1}{2 \rho L} G(y_{k+1};{i_k}),
			\end{aligned}
			\end{equation*}
			for which the bound $\E f(y_{k})-f_\star\leq \tfrac{2\rho L\normsq{x_0-x_\star}}{k}$ holds for $k\geq 0$. As in the previous section (over-parametrized models, \appref{sec:overparameterized}), the use of primal averaging for weak growth conditions appeared naturally through the use of the parameter selection technique starting from
			\begin{equation*}
			\begin{aligned}
			y_{k+1}&= (1-\tau_k) x_k+\tau_k z_k,\\
			x_{k+1}^{(i_k)}&=y_{k+1}-\alpha_k G(y_{k+1};i_k),\\
			z_{k+1}^{(i_k)}&=(1-\delta_k) y_{k+1}+\delta_k z_k -\gamma_k G(y_{k+1};i_k),\\
			\end{aligned}
			\end{equation*}
			for which we wish to optimize the policy $\{(\tau_k,\alpha_k,\delta_k,\gamma_k)\}_k$. Using the line-search workaround
			\begin{equation*}
			\begin{aligned}
			y_{k+1}&=\argm{x}{f(x)\, \st\, x\in x_k+\sspan\{z_k-x_k\}},\\
			x_{k+1}^{(i_k)}&=\argm{x}{f(x)\, \st\, x\in y_{k+1}+\sspan \{G(y_{k+1};i_k)\}},\\
			z_{k+1}^{(i_k)}&=(1-\delta_k) y_{k+1}+\delta_k z_k -\gamma_k G(y_{k+1};i_k),\\
			\end{aligned}
			\end{equation*}
			the parameters of primal averaging appeared as a feasible point to the parameter selection technique.
			
			Before going into next section, let us briefly mention that this setting can be embedded within similar LMIs as before, by encoding the $G(x;i)$'s for all $i$'s and all $x$'s where a stochastic gradient is used. Then, one can rely on interpolation conditions for constraining gradients and stochastic gradients (stochastic gradients are embedded within interpolation conditions by averaging them $f'(x)=\tfrac1n\sum_{i=1}^n G(x;i)$). For example, for analyzing a stochastic gradient scheme $x_{k+1}^{(i_k)}=x_k-\delta_k G(x_k;i_k)$ under a weak growth condition, for a potential $\phi_k^f=d_k \, (f(x_k)-f_\star)+\tfrac{L}2 \normsq{x_k-x_\star}$, one can use a Gram matrix $G=P^\top P\succeq 0$ with
			\[P = [ \ x_k \ | \ G(x_k;1)\ \hdots\ G(x_k;n) \ | \ f'(x_{k+1}^{(1)}) \ \hdots \ f'(x_{k+1}^{(n)})\ ], \]
			along with function values $F=[ \ f(x_k) \ | \ f(x_{k+1}^{(1)}) \ \hdots \ f(x_{k+1}^{(n)}) \ ]$; the subtlety in that setting is to include $A_{\textrm{var}}$ and $a_{\textrm{var}}$ for encoding the variance of $G(x_k;i)$'s, with
			\[ A_{\textrm{var}} = \tfrac1n \sum_{i=1}^n e_{1+i}e_{1+i}^\top\in\Sb^{1+2n},\quad a_{\textrm{var}}=2\rho L \,e_1\in\R^{n+2},\]
			for requiring $\tra (A_{\textrm{var}} G)=\E_i\normsq{G(x_k;i)}\leq 2\rho L (f(x_k)-f_\star)=F a_{\textrm{var}}$.
			\subsection{Proof of Theorem~\ref{thm:wgc}}\label{sec:wcg_pavging}

			\begin{proof} Combine the following inequalities with their corresponding weights.
				\begin{itemize}
					\item Convexity between $x_\star$ and $y_{k+1}$  with weight $\lambda_1=\delta_k L$,
					\[ f_\star\geq f(y_{k+1})+\inner{f'(y_{k+1})}{x_\star-y_{k+1}},\]
					\item convexity between $y_k$ and $y_{k+1}$  with weight $\lambda_2=d_k$,
					\[f(y_{k})\geq f(y_{k+1})+\inner{f'(y_{k+1})}{y_k-y_{k+1}},\]
					\item weak growth condition with weight $\lambda_3=\tfrac{\delta_k^2 L}{2}$
					\[\E_{i_k}\normsq{G(y_{k+1};{i_k})}\leq 2\rho L (f(y_{k+1})-f_\star). \]
				\end{itemize}
				The weighted sum of those inequalities yields the desired result:
				\begin{equation*}
				\begin{aligned}
				0\geq\, &\lambda_1 \left[f(y_{k+1})-f_\star+\inner{f'(y_{k+1})}{x_\star-y_{k+1}} \right]+\lambda_2\left[f(y_{k+1})- f(y_{k})+\inner{f'(y_{k+1})}{y_k-y_{k+1}}\right]\\
				&+\lambda_3 \left[ \E_{i_k}\normsq{G(y_{k+1};{i_k})} - 2\rho L (f(y_{k+1})-f_\star) \right]\\
				=& (d_k+\delta_k L-\rho \delta_k^2 L^2)(f(y_{k+1})-f_\star)+\tfrac{L}2 \E_{i_k}\normsq{x_{k+1}^{({i_k})}-x_\star}-d_k (f(y_k)-f_\star)-\tfrac{L}2 \normsq{x_{k}-x_\star},
				\end{aligned}
				\end{equation*}
				which can be rearranged to the desired:
				\[ (d_k+\delta_k L-\rho \delta_k^2 L^2)(f(y_{k+1})-f_\star)+\tfrac{L}2 \E_{i_k}\normsq{x_{k+1}^{({i_k})}-x_\star}\leq d_k (f(y_k)-f_\star)+\tfrac{L}2 \normsq{x_{k}-x_\star}. \]
			\end{proof}
			
			\clearpage 
			
			
			\section{Bounded variance at optimum}\label{sec:boundedVaratOpt}
			As in the previous sections, it is possible to apply the methodology and the parameter selection technique for designing a method to solve
			\[ \min_{x\in\Rd} \{ f(x)\equiv \ \E_i f_i(x) \},  \]
			with $\E_i\normsq{f'(x_\star)}\leq\sigma_\star^2$ (over-parametrized models from \secref{sec:zeroVarOpt} arise as particular case with $\sigma_\star=0$). Again, we obtain a primal averaging scheme, but with smaller allowed step-sizes.
			\begin{theorem}\label{thm:boundedvaratopt} Let $x_k\in\Rd$, $f_i\in\FL$ and an optimal point $x_\star$ such that $\E_i\normsq{f_i'(x_\star)}\leq \sigma^2_\star$. Then the iterative scheme
				\begin{equation*}
				\begin{aligned}
					y_{k+1}&=\tfrac{d_k}{d_k+\delta_k L}y_k+\tfrac{\delta_k L}{d_k+\delta_k L}x_k,\\
					x_{k+1}^{({i_k})}&=x_k-\delta_k f_{{i_k}}'(y_{k+1}),
				\end{aligned}
				\end{equation*}
				satisfies
				\[d_{k+1} (f(y_{k+1})-f_\star)+\tfrac{L}{2}\E_{i_k}\normsq{x_{k+1}^{({i_k})}-x_\star}\leq d_{k} (f(y_k)-f_\star)+\tfrac{L}{2}\normsq{x_k-x_\star}+e_k \sigma^2_\star,\]
				for all values $d_k,d_{k+1},\delta_k,e_k\geq0$ satisfying $0\leq \delta_k<\tfrac1L$, $d_k\geq0$, $d_{k+1}\leq d_k+\delta_k L$ and $e_k=\tfrac{\delta_k ^2 L}{2 (1-\delta_k  L)}$.
			\end{theorem}
			\begin{proof} Let us reformulate the following weighted sum.
				\begin{itemize}
					\item Averaged smoothness and convexity between $y_{k+1}$ and $x_\star$ with weight $\lambda_1=\delta_k L$:
					\[f_\star \geq  f(y_{k+1})+ \inner{f'(y_{k+1})}{x_\star-y_{k+1}}+\tfrac{1}{2L}\E_{i_k}\normsq{f'_{{i_k}}(y_{k+1})-f'_{i_k}(x_\star)},\]
					\item convexity between ${\color{red}y_{k+1}}$ and ${\color{red}y_k}$ with weight $\lambda_2=d_k$:
					\[f(y_k) \geq f(y_{k+1}) + \inner{f'(y_{k+1})}{y_k-y_{k+1}},\]
					\item bounded variance of the gradients at $x_\star$ with weight $\lambda_3=e_k$
					\[\E_{i_k}\normsq{f'_{i_k}(x_\star)}\leq \sigma^2_\star.\]
				\end{itemize}
				The weighted sum of those inequalities can be rewritten as {{\bf \color{red}[Update: there was a missing term in the weighted sum]}}
				\begin{equation*}
					\begin{aligned}
						0\geq & \lambda_1 \left[f(y_{k+1})-f_\star+ \inner{f'(y_{k+1})}{x_\star-y_{k+1}}+\tfrac{1}{2L}\E_{i_k}\normsq{f'_{{i_k}}(y_{k+1})}\right]\\
						&+\lambda_2\left[ f(y_{k+1}) -f(y_k)+ \inner{f'(y_{k+1})}{y_k-y_{k+1}}\right]+{\color{red}\lambda_3\left[\E_{i_k}\normsq{f'_{i_k}(x_\star)}-\sigma^2_\star\right]}\\
						=&(d_k+\delta_k L)(f(y_{k+1})-f_\star)+\tfrac{L}{2}\E_{i_k}\normsq{x_{k+1}^{({i_k})}-x_\star}- d_k (f(y_k)-f_\star)-\tfrac{L}{2}\normsq{x_k-x_\star}\\
						&+ \tfrac{\delta_k }{2-2 \delta_k  L} \E_{i_k}\normsq{f'_{i_k}(x_\star)+(\delta_k  L-1)f'_{{i_k}}(y_{k+1})}+ \tfrac{2 e_k (\delta_k  L-1)+\delta_k ^2 L}{2 \delta_k  L-2} \E_{i_k}\normsq{f'_{i_k}(x_\star)}-e_k\sigma^2_\star,
					\end{aligned}
				\end{equation*}
				which can be rearranged to (note that $e_k$ is chosen such that $\tfrac{2 e_k (\delta_k  L-1)+\delta_k ^2 L}{2 \delta_k  L-2}=0$)
				\begin{equation*}
					\begin{aligned}
						&(d_k+\delta_k L)(f(y_{k+1})-f_\star)+\tfrac{L}{2}\E_{i_k}\normsq{x_{k+1}^{({i_k})}-x_\star}\\ &\leq d_k (f(y_k)-f_\star)+\tfrac{L}{2}\normsq{x_k-x_\star}+e_k\sigma^2_\star-\tfrac{\delta_k }{2-2 \delta_k  L} \E_{i_k}\normsq{f'_{i_k}(x_\star)+(\delta_k  L-1)f'_{{i_k}}(y_{k+1})}\\
						&\leq d_k (f(y_k)-f_\star)+\tfrac{L}{2}\normsq{x_k-x_\star}+e_k\sigma^2_\star,
					\end{aligned}
				\end{equation*}
				where the last inequality follows from $\delta_k<\tfrac1L$.
			\end{proof}

			\clearpage 
			\section{Randomized block-coordinate descent}\label{sec:coordinatedescent}
			In this section, we illustrate the use of the methodology for block-coordinate type schemes. In contrast with standard references on the topic (see e.g., \citet{nesterov2012efficiency,richtarik2014iteration,fercoq2015accelerated}), we make use of the global Lipschitz constant of the function to be minimized, instead of Lipschitz constants of the individual blocks. This choice is made for convenience and illustrative purpose, and the results presented here  can be adapted for using the Lipschitz constants of the blocks (see e.g.,~\citet{shi2017better} where it is done for cyclic coordinate descent), or even for dealing with (separable) proximal operators~\citep{richtarik2014iteration}. 
			
			The setting is as follows: consider the problem
			\[\min_{x\in\Rd} f(x),\]
			where $f\in\FmuL$, and $x$ is partitioned into $n$ blocks: $x = \sum_{i=1}^n \bU_i \, x,$
			with the partition described by $[\bU_1\, \bU_2 \, \hdots \, \bU_n ] =I_d$. For solving the problem, we use $x_{k+1}^{(i_k)}=x_{k}-{\delta_k} \bU_{i_k} f'(x_{k})$,
			where $i_k$ is chosen uniformly at random in the set $\{1,\hdots,n\}$, $\delta_k$ is a step-size, and $\bU_{i} f'(x_{k})=\nabla_{i}f(x_k)$ is the directional derivate along the i$^\text{th}$ block of coordinates. We do not detail the LMI formulations here but rather note that they follow the exact same lines as in the previous sections after adapting the Gram matrix $G$, formulating it with $G=P^\top P\succeq 0$ and
			\[ P=[x_{\color{red}k} \, | \, \bU_1 f'(x_k) \, \hdots \, \bU_n f'(x_k) \,  | \, f'(x_{k+1}^{(1)}) \, \hdots \, f'(x_{k+1}^{(n)}) ].\]
			\subsection{Block-coordinate descent for smooth convex minimization}\label{sec:CD_sublin}
			
			\begin{theorem}\label{thm:RBCD_Lyap}
				Let $x_k\in\Rd$, $f\in\FL$, and $x_{k+1}^{({i_k})}=x_{k}-\delta_k \bU_{{i_k}} f'(x_{k})$ with $0\leq \delta_k\leq \tfrac1L$. The inequality $\E_{{i_k}}\phi_{k+1}^f(x_{k+1}^{({i_k})})\leq \phi_k^f(x_{k})$ holds with 
				\[\phi_k^f=d_k (f(x_k)-f_\star)+ \tfrac{L}{2}\normsq{x_k-x_\star},\]
				for all values $d_k\geq1$, and $d_{k+1}= d_{k}+\tfrac{\delta_k L}{n}$.
			\end{theorem}
			\begin{proof}Once more, the proof consists in linear combinations of inequalities:
				\begin{itemize}
					\item convexity between $x_\star$ and $x_k$ with weight $\lambda_1=\tfrac{\delta_k L}n$:
					\[f_\star \geq  f(x_k) + \inner{f'(x_k)}{x_\star-x_k},\]
					\item averaged smoothness between $x_{k+1}^{(i)}$ and $x_k$ with weight $\lambda_2=d_k+\tfrac{\delta_k L}n$:
					\[ \E_{i_k} f(x_{k+1}^{({i_k})})\leq f(x_k)+\E_{i_k} \inner{f'(x_k)}{x_{k+1}^{({i_k})}-x_k}+\tfrac L 2 \E_{i_k}\normsq{x_{k+1}^{({i_k})}-x_k}.\]
				\end{itemize}
				Recalling that $f'(x_k)=\sum_{i=1}^n \bU_i f'(x_k)$ (and therefore $\sum_{i=1}^n \normsq{\bU_i f(x_k)}=\normsq{f(x_k)}$) and that $x_{k+1}^{({i_k})}=x_k-\delta_k \bU_{{i_k}}f'(x_k)$, the weighted sum can be reformulated as
				\begin{equation*}
				\begin{aligned}
				0\geq &\lambda_1 \left[ f(x_k)-f_\star + \inner{f'(x_k)}{x_\star-x_k}\right]\\
				&+\lambda_2 \left[\E_{i_k} f(x_{k+1}^{({i_k})})- f(x_k)-\E_{i_k} \inner{f'(x_k)}{x_{k+1}^{({i_k})}-x_k}-\tfrac L 2 \E_{i_k}\normsq{x_{k+1}^{({i_k})}-x_k}\right]\\
				=&(d_k+\tfrac{\delta_k L}{n}) \E_{i_k} (f(x_{k+1}^{({i_k})})-f_\star)+\tfrac{L}2\E_{i_k}\normsq{x_{k+1}^{({i_k})}-x_\star}\\&-d_k (f(x_k)-f_\star)-\tfrac L2\normsq{x_k-x_\star}+\tfrac{\delta_k}n\left((1-\tfrac{\delta_k L}{2})(d_k+\tfrac{\delta_k L}{n})-\tfrac{\delta_k L}{2}\right)\normsq{f'(x_k)},
				\end{aligned}
				\end{equation*}
				which can in turn be reorganized as
				\begin{equation*}
				\begin{aligned}
				&(d_k+\tfrac{\delta_k L}{n}) \E_{i_k} (f(x_{k+1}^{({i_k})})-f_\star)+\tfrac{L}2\E_{i_k}\normsq{x_{k+1}^{({i_k})}-x_\star}\\&\,\leq d_k (f(x_k)-f_\star)+\tfrac L2\normsq{x_k-x_\star}-\tfrac{\delta_k}n\left((1-\tfrac{\delta_k L}{2})(d_k+\tfrac{\delta_k L}{n})-\tfrac{\delta_k L}{2}\right)\normsq{f'(x_k)},\\&\,\leq d_k (f(x_k)-f_\star)+\tfrac L2\normsq{x_k-x_\star},
				\end{aligned}
				\end{equation*}
				where the last inequality follows from $\left((1-\tfrac{\delta_k L}{2})(d_k+\tfrac{\delta_k L}{n})-\tfrac{\delta_k L}{2}\right)\geq 0$, which can be verified as follows. Define 
				\[\phi(\delta_k):=\left((1-\tfrac{\delta_k L}{2})(d_k+\tfrac{\delta_k L}{n})-\tfrac{\delta_k L}{2}\right),\]
				which is a negative definite quadratic in $\delta_k$. In addition, we have $\phi(\tfrac1L)=\frac{1}{2} \left(d_k+\frac{1}{n}-1\right)>0$ (because $d_k\geq 1$ by assumption) and $\phi(0)=d_k>0$. Since by assumption we had $0\leq \delta_k\leq \tfrac1L$, this discussion completes the proof.
			\end{proof}

			\subsection{Block-coordinate descent for smooth strongly convex minimization}\label{sec:CD_lin}
			The following result was obtained using the same methodology as before: finding a feasible point to~\eqref{eq:tt}. We obtain a simple linear convergence expressed in terms of distance to optimality:
			\[ \E_{i_k}\normsq{x_{k+1}^{({i_k})}-x_\star} \leq \rho^2\normsq{x_k-x_\star},\]
			with $\rho^2:=\max\left\{\left(\tfrac{(\delta_k  \mu -1)^2+n-1}{n}\right),\left(\tfrac{(\delta_k  L -1)^2+n-1}{n}\right)\right\}$.
			
			\begin{theorem}\label{thm:RBCD_Lyap_mstr}
				Let $x_k\in\Rd$, $f\in\FmuL$ (class of $L$-smooth $\mu$-strongly convex functions), and $x_{k+1}^{({i_k})}=x_{k}-\delta_k \bU_{{i_k}} f'(x_{k})$. The inequality $\E_{{i_k}}\phi_{k+1}^f(x_{k+1}^{({i_k})})\leq \phi_k^f(x_{k})$ holds with 
				\[\phi_k^f=  a_k\normsq{x_k-x_\star},\]
				for all values $a_k>0$, and $a_{k+1}= a_k \max\left\{\left(\tfrac{(\delta_k  \mu -1)^2+n-1}{n}\right),\left(\tfrac{(\delta_k  L -1)^2+n-1}{n}\right)\right\}^{-1}$.
			\end{theorem}
			
			\begin{proof}
				As in the previous sections, the proof consists in a linear combination of inequalities:
				\begin{itemize}
					\item smoothness and strong convexity between $x_k$ and $x_\star$ with weight $\lambda$
					\begin{equation*}
					\begin{aligned}
					f_\star\geq &f(x_k)+\inner{f'(x_k)}{x_\star-x_k}\\&+\tfrac{1}{2(1-\tfrac\mu L)}\left(\tfrac1L\normsq{f'(x_k)}+\mu\normsq{x_k-x_\star}-2\tfrac\mu L\inner{f'(x_k)}{x_k-x_\star}\right),
					\end{aligned}
					\end{equation*}
					\item smoothness and strong convexity between $x_\star$ and $x_k$ with weight $\lambda$
					\[f(x_k)\geq f_\star+\tfrac{1}{2(1-\tfrac\mu L)}\left(\tfrac1L\normsq{f'(x_k)}+\mu\normsq{x_k-x_\star}-2\tfrac\mu L\inner{f'(x_k)}{x_k-x_\star}\right).\]
				\end{itemize}
				Summing up those two inequalities leads to
				\[0\geq \lambda \left[\inner{f'(x_k)}{x_\star\edit{-x_k}}+\tfrac{1}{(1-\tfrac\mu L)}\left(\tfrac1L\normsq{f'(x_k)}+\mu\normsq{x_k-x_\star}-2\tfrac\mu L\inner{f'(x_k)}{x_k-x_\star}\right)\right],\]
				which we reformulate below. The proof is split in two cases:
				\begin{itemize}
					\item $\rho^2=\tfrac{(\delta_k\mu-1)^2+n-1}{n}$; in that case set $\lambda=\tfrac{2\delta_k}{n}(1-\delta_k\mu)$.
					Recalling that $f'(x_k)=\sum_{{i_k}=1}^n \bU_{i_k} f'(x_k)$ and that $x_{k+1}^{({i_k})}=x_k-\delta_k \bU_{{i_k}}f'(x_k)$, we get
					\begin{equation*}
					\begin{aligned}
					0&\geq \lambda \left[\inner{f'(x_k)}{x_\star-x_k}+\tfrac{1}{(1-\tfrac\mu L)}\left(\tfrac1L\normsq{f'(x_k)}+\mu\normsq{x_k-x_\star}-2\tfrac\mu L\inner{f'(x_k)}{x_k-x_\star}\right)\right]\\
					&=\tfrac{1}{n} \sum_{{i_k}=1}^n\normsq{x_k-\delta_k\bU_{i_k} f'(x_k)-x_\star}-\tfrac{(1-\delta_k \mu)^2+n-1}{n}\normsq{x_k-x_\star}\\&\quad +\tfrac{\delta_k(2-\delta_k(L+\mu))}{n(L-\mu)}\normsq{\mu(x_k-x_\star)-f'(x_k)},
					\end{aligned}
					\end{equation*}
					which can be reorganized as
					\begin{equation*}
					\begin{aligned}
					\E_{i_k}\normsq{x_{k+1}^{({i_k})}-x_\star}&\leq \tfrac{(1-\delta_k \mu)^2+n-1}{n}\normsq{x_k-x_\star}-\tfrac{\delta_k(2-\delta_k(L+\mu))}{n(L-\mu)}\normsq{\mu(x_k-x_\star)-f'(x_k)},\\
					&\leq \tfrac{(1-\delta_k \mu)^2+n-1}{n}\normsq{x_k-x_\star},
					\end{aligned}
					\end{equation*}
					where the last inequality is valid because \[\rho^2=\tfrac{(\delta_k\mu-1)^2+n-1}{n}=\max\left\{\left(\tfrac{(\delta_k  \mu -1)^2+n-1}{n}\right),\left(\tfrac{(\delta_k  L -1)^2+n-1}{n}\right)\right\}\] in this setting (see choice of $\rho$), and therefore $\delta_k\leq \tfrac2{L+\mu}$.
					\item {\color{red}$\rho^2=\tfrac{(\delta_k L-1)^2+n-1}{n}$}; in that case set $\lambda=\tfrac{2\delta_k}{n}(\delta_k L-1)$. By recalling again that $f'(x_k)=\sum_{{i_k}=1}^n \bU_{i_k} f'(x_k)$ and that $x_{k+1}^{({i_k})}=x_k-\delta_k \bU_{{i_k}}f'(x_k)$, we get
					\begin{equation*}
					\begin{aligned}
					0&\geq \lambda \left[\inner{f'(x_k)}{x_\star\edit{-x_k}}+\tfrac{1}{(1-\tfrac\mu L)}\left(\tfrac1L\normsq{f'(x_k)}+\mu\normsq{x_k-x_\star}-2\tfrac\mu L\inner{f'(x_k)}{x_k-x_\star}\right)\right]\\
					&=\tfrac{1}{n} \sum_{{i_k}=1}^n\normsq{x_k-\delta_k\bU_{i_k} f'(x_k)-x_\star}-\tfrac{(1-\delta_k L)^2+n-1}{n}\normsq{x_k-x_\star}\\&\quad +\tfrac{\delta_k(-2+\delta_k(L+\mu))}{n(L-\mu)}\normsq{L(x_k-x_\star)-f'(x_k)},
					\end{aligned}
					\end{equation*}
					which can be reorganized as
					\begin{equation*}
					\begin{aligned}
					\E_{i_k}\normsq{x_{k+1}^{({i_k})}-x_\star}&\leq \tfrac{(1-\delta_k L)^2+n-1}{n}\normsq{x_k-x_\star}-\tfrac{\delta_k(-2+\delta_k(L+\mu))}{n(L-\mu)}\normsq{L(x_k-x_\star)-f'(x_k)},\\
					&\leq \tfrac{(1-\delta_k L)^2+n-1}{n}\normsq{x_k-x_\star},
					\end{aligned}
					\end{equation*}
					where the last inequality is valid because \[\rho^2=\tfrac{(\delta_k L-1)^2+n-1}{n}=\max\left\{\left(\tfrac{(\delta_k  \mu -1)^2+n-1}{n}\right),\left(\tfrac{(\delta_k  L -1)^2+n-1}{n}\right)\right\},\] in this setting (see choice of $\rho$), and therefore $\delta_k\geq \tfrac2{L+\mu}$.
				\end{itemize}
			\end{proof}

		\end{document}